\documentclass{article}
\usepackage{amsfonts,amsmath,amssymb,amsthm}
\usepackage{tikz}
\usepackage{fullpage,hyperref}

\newcommand{\inv}{\mathrm{inv}}
\newcommand{\ms}{\mathfrak{s}}
\newcommand{\ep}{\epsilon}
\newcommand{\xyza}{\begin{tikzpicture}[scale=0.7,baseline=-1mm]
\draw (4,-0.5) circle (3pt);
\draw (5,0.5) circle (3pt);
\fill[black]  (6,0.5) circle (3pt);
\draw (7,0.5) circle (3pt);
\draw  (8,1.5) circle (3pt);
\draw [thick] (3.5,-1) -- (5.5,1) -- (6.5,0) -- (8.5,2);
\draw [dashed] (3,-1.5) -- (9,-1.5);
\fill[red] (3.5,-1) circle (2pt);
\fill[red] (4.5,0) circle (2pt);
\fill[red] (5.5,1) circle (2pt);
\fill[red] (6.5,0) circle (2pt);
\fill[red] (7.5,1) circle (2pt);
\fill[red] (8.5,2) circle (2pt);
\draw [thick] (3,-2) -- (9,-2);
\draw (3.5,-2) circle (3pt);
\draw (4.5,-2) circle (3pt);
\fill[black] (5.5,-2) circle (3pt);
\draw (6.5,-2) circle (3pt);
\draw (7.5,-2) circle (3pt);
\draw (8.5,-2) circle (3pt);
\draw (5.5,-2.5) node {$x$};
\end{tikzpicture}}

\newcommand{\xyzao}{\begin{tikzpicture}[scale=0.7,baseline=-1mm]
\draw (4,-0.5) circle (3pt);
\draw (5,0.5) circle (3pt);
\fill[black]  (7,1.5) circle (3pt);
\draw (6,1.5) circle (3pt);
\draw  (8,1.5) circle (3pt);
\draw [thick] (3.5,-1) -- (5.5,1) -- (6.5,2) -- (7.5,1) -- (8.5,2);
\draw [dashed] (3,-1.5) -- (9,-1.5);
\fill[red] (3.5,-1) circle (2pt);
\fill[red] (4.5,0) circle (2pt);
\fill[red] (5.5,1) circle (2pt);
\fill[red] (6.5,2) circle (2pt);
\fill[red] (7.5,1) circle (2pt);
\fill[red] (8.5,2) circle (2pt);
\draw [thick] (3,-2) -- (9,-2);
\draw (3.5,-2) circle (3pt);
\draw (4.5,-2) circle (3pt);
\fill[black] (6.5,-2) circle (3pt);
\draw (5.5,-2) circle (3pt);
\draw (7.5,-2) circle (3pt);
\draw (8.5,-2) circle (3pt);
\draw (6.5,-2.5) node {$x$};
\end{tikzpicture}}

\newcommand{\xyzb}{\begin{tikzpicture}[scale=0.7,baseline=-1mm]
\draw (4,-0.5) circle (3pt);
\draw (5,0.5) circle (3pt);
\fill[black]  (6,0.5) circle (3pt);
\draw (7,0.5) circle (3pt);
\draw  (8,1.5) circle (3pt);
\draw [thick] (3.5,-1) -- (5.5,1) -- (6.5,0) -- (8.5,2);
\draw [dashed] (3,-1.5) -- (9,-1.5);
\fill[red] (3.5,-1) circle (2pt);
\fill[red] (4.5,0) circle (2pt);
\fill[red] (5.5,1) circle (2pt);
\fill[red] (6.5,0) circle (2pt);
\fill[red] (7.5,1) circle (2pt);
\fill[red] (8.5,2) circle (2pt);
\draw [thick] (3,-2) -- (9,-2);
\draw (5.5,-2) circle (3pt);
\draw (4.5,-2) circle (3pt);
\fill[black] (6.5,-2) circle (3pt);
\draw (3.5,-2) circle (3pt);
\draw (7.5,-2) circle (3pt);
\draw (8.5,-2) circle (3pt);
\draw (6.5,-2.5) node {$x+1$};
\end{tikzpicture}}

\newcommand{\xyzbp}{\begin{tikzpicture}[scale=0.7,baseline=-1mm]
\draw (4,-0.5) circle (3pt);
\draw (5,0.5) circle (3pt);
\fill[black]  (6,0.5) circle (3pt);
\draw (7,0.5) circle (3pt);
\draw  (8,1.5) circle (3pt);
\draw [thick] (3.5,-1) -- (5.5,1) -- (6.5,0) -- (8.5,2);
\draw [dashed] (3,-1.5) -- (9,-1.5);
\fill[red] (3.5,-1) circle (2pt);
\fill[red] (4.5,0) circle (2pt);
\fill[red] (5.5,1) circle (2pt);
\fill[red] (6.5,0) circle (2pt);
\fill[red] (7.5,1) circle (2pt);
\fill[red] (8.5,2) circle (2pt);
\draw [thick] (3,-2) -- (9,-2);
\draw (5.5,-2) circle (3pt);
\draw (4.5,-2) circle (3pt);
\fill[black] (7.5,-2) circle (3pt);
\draw (3.5,-2) circle (3pt);
\draw (6.5,-2) circle (3pt);
\draw (8.5,-2) circle (3pt);
\draw (7.5,-2.5) node {$x+1$};
\end{tikzpicture}}

\newcommand{\xyzbo}{\begin{tikzpicture}[scale=0.7,baseline=-1mm]
\draw (4,-0.5) circle (3pt);
\draw (5,0.5) circle (3pt);
\fill[black]  (6,0.5) circle (3pt);
\draw (7,0.5) circle (3pt);
\draw  (8,1.5) circle (3pt);
\draw [thick] (3.5,-1) -- (5.5,1) -- (6.5,0) -- (8.5,2);
\draw [dashed] (3,-1.5) -- (9,-1.5);
\fill[red] (3.5,-1) circle (2pt);
\fill[red] (4.5,0) circle (2pt);
\fill[red] (5.5,1) circle (2pt);
\fill[red] (6.5,0) circle (2pt);
\fill[red] (7.5,1) circle (2pt);
\fill[red] (8.5,2) circle (2pt);
\draw [thick] (3,-2) -- (9,-2);
\draw (5.5,-2) circle (3pt);
\draw (4.5,-2) circle (3pt);
\fill[black] (6.5,-2) circle (3pt);
\draw (3.5,-2) circle (3pt);
\draw (7.5,-2) circle (3pt);
\draw (8.5,-2) circle (3pt);
\draw (6.5,-2.5) node {$x$};
\end{tikzpicture}}

\newcommand{\xyzz}{\begin{tikzpicture}[scale=0.7,baseline=-1mm]
\draw (4,-0.5) circle (3pt);
\draw (5,0.5) circle (3pt);
\fill[black]  (6,0.5) circle (3pt);
\draw (7,0.5) circle (3pt);
\draw  (8,1.5) circle (3pt);
\draw [thick] (3.5,-1) -- (5.5,1) -- (6.5,0) -- (8.5,2);
\draw [dashed] (3,-1.5) -- (9,-1.5);
\fill[red] (3.5,-1) circle (2pt);
\fill[red] (4.5,0) circle (2pt);
\fill[red] (5.5,1) circle (2pt);
\fill[red] (6.5,0) circle (2pt);
\fill[red] (7.5,1) circle (2pt);
\fill[red] (8.5,2) circle (2pt);
\draw [thick] (3,-2) -- (9,-2);
\draw (3.5,-2) circle (3pt);
\draw (5.5,-2) circle (3pt);
\fill[black] (4.5,-2) circle (3pt);
\draw (6.5,-2) circle (3pt);
\draw (7.5,-2) circle (3pt);
\draw (8.5,-2) circle (3pt);
\draw (4.5,-2.5) node {$x-1$};
\end{tikzpicture}}

\newcommand{\xyzzp}{\begin{tikzpicture}[scale=0.7,baseline=-1mm]
\draw (4,-0.5) circle (3pt);
\draw (5,0.5) circle (3pt);
\fill[black]  (6,0.5) circle (3pt);
\draw (7,0.5) circle (3pt);
\draw  (8,1.5) circle (3pt);
\draw [thick] (3.5,-1) -- (5.5,1) -- (6.5,0) -- (8.5,2);
\draw [dashed] (3,-1.5) -- (9,-1.5);
\fill[red] (3.5,-1) circle (2pt);
\fill[red] (4.5,0) circle (2pt);
\fill[red] (5.5,1) circle (2pt);
\fill[red] (6.5,0) circle (2pt);
\fill[red] (7.5,1) circle (2pt);
\fill[red] (8.5,2) circle (2pt);
\draw [thick] (3,-2) -- (9,-2);
\draw (3.5,-2) circle (3pt);
\draw (4.5,-2) circle (3pt);
\fill[black] (5.5,-2) circle (3pt);
\draw (6.5,-2) circle (3pt);
\draw (7.5,-2) circle (3pt);
\draw (8.5,-2) circle (3pt);
\draw (5.5,-2.5) node {$x-1$};
\end{tikzpicture}}

\newcommand{\aayxaa}{
\begin{tikzpicture}[scale=0.7,baseline=-1mm]
\draw [thick] (0,0) -- (6,0);
\draw (0.5,0) circle (3pt);
\draw (1.5,0) circle (3pt);
\fill[black] (3.5,0) circle (3pt);
\draw (2.5,0) circle (3pt);
\draw (4.5,0) circle (3pt);
\draw (5.5,0) circle (3pt);
\draw (3.5,0) node (yo)  {};
\draw (2.5,0) node (oy) {};
\draw (3.5,-0.5) node {$x$};
\draw (yo) edge[in=90,out=90,->,dashed] (oy);
\end{tikzpicture}
}
\newcommand{\aaxyaa}{
\begin{tikzpicture}[scale=0.7,baseline=-1mm]
\draw [thick] (0,0) -- (6,0);
\draw (0.5,0) circle (3pt);
\draw (1.5,0) circle (3pt);
\fill[black] (3.5,0) circle (3pt);
\draw (2.5,0) circle (3pt);
\draw (4.5,0) circle (3pt);
\draw (5.5,0) circle (3pt);
\draw (3.5,0) node (yo)  {};
\draw (4.5,0) node (oy) {};
\draw (3.5,-0.5) node {$x$};
\draw (yo) edge[in=90,out=90,->,dashed] (oy);
\end{tikzpicture}
}

\newcommand{\ayxaaa}{
\begin{tikzpicture}[scale=0.7,baseline=-1mm]
\draw [thick] (0,0) -- (6,0);
\draw (0.5,0) circle (3pt);
\draw (1.5,0) circle (3pt);
\fill[black] (2.5,0) circle (3pt);
\draw (3.5,0) circle (3pt);
\draw (4.5,0) circle (3pt);
\draw (5.5,0) circle (3pt);
\draw (2.5,0) node (yo)  {};
\draw (1.5,0) node (oy) {};
\draw (2.5,-0.5) node {$x$};
\draw (yo) edge[in=90,out=90,->,dashed] (oy);
\end{tikzpicture}
}
\newcommand{\axyaaa}{
\begin{tikzpicture}[scale=0.7,baseline=-1mm]
\draw [thick] (0,0) -- (6,0);
\draw (0.5,0) circle (3pt);
\draw (1.5,0) circle (3pt);
\fill[black] (2.5,0) circle (3pt);
\draw (3.5,0) circle (3pt);
\draw (4.5,0) circle (3pt);
\draw (5.5,0) circle (3pt);
\draw (2.5,0) node (yo)  {};
\draw (3.5,0) node (oy) {};
\draw (2.5,-0.5) node {$x$};
\draw (yo) edge[in=90,out=90,->,dashed] (oy);
\end{tikzpicture}
}

\newcommand{\axy}{
\begin{tikzpicture}[scale=0.7,baseline=-1mm]
\draw [thick] (1,0) -- (4,0);
\draw (1.5,0) circle (3pt);
\fill[black] (2.5,0) circle (3pt);
\draw (3.5,0) circle (3pt);
\draw (2.5,0) node (yo)  {};
\draw (3.5,0) node (oy) {};
\draw (2.5,-0.5) node {$x$};
\draw (yo) edge[in=90,out=90,->,dashed] (oy);
\end{tikzpicture}
}

\newcommand{\axz}{
\begin{tikzpicture}[scale=0.7,baseline=-1mm]
\draw [thick] (1,0) -- (4,0);
\draw (1.5,0) circle (3pt);
\fill[black] (2.5,0) circle (3pt);
\fill[black]  (3.5,0) circle (3pt);
\draw (2.5,0) node (yo)  {};
\draw (1.5,0) node (oy) {};
\draw (yo) edge[in=90,out=90,->,dashed] (oy);
\end{tikzpicture}
}

\newcommand{\yxa}{
\begin{tikzpicture}[scale=0.7,baseline=-1mm]
\draw [thick] (1,0) -- (4,0);
\draw (1.5,0) circle (3pt);
\fill[black] (2.5,0) circle (3pt);
\draw (3.5,0) circle (3pt);
\draw (2.5,0) node (yo)  {};
\draw (1.5,0) node (oy) {};
\draw (2.5,-0.5) node {$x$};
\draw (yo) edge[in=90,out=90,->,dashed] (oy);
\end{tikzpicture}
}

\newcommand{\aaa}{
\begin{tikzpicture}[scale=0.7,baseline=9mm]
\draw [thick] (0,0) -- (4,0);
\draw (0.5,0) circle (3pt);
\fill[black] (1.5,0) circle (3pt);
\fill[black] (2.5,0) circle (3pt);
\draw (3.5,0) circle (3pt);
\draw [dashed] (0,0.5) -- (4,0.5);
\draw [thick] (3.5,2.5) -- (2.5,1.5) -- (1.5,2.5) -- (0.5,1.5);
\foreach \y in {1}{
\draw (0+\y,1+\y) circle (3pt);
}
\draw (3,2) circle (3pt);
\fill[black]  (2,2) circle (3pt);
\fill[red] (0.5,1.5) circle (3pt);
\fill[red] (1.5,2.5) circle (3pt);
\fill[red] (2.5,1.5) circle (3pt);
\fill[red] (3.5,2.5) circle (3pt);
\end{tikzpicture}
}

\newcommand{\bbb}{
\begin{tikzpicture}[scale=0.7,baseline=9mm]
\draw [thick] (0,0) -- (4,0);
\draw (0.5,0) circle (3pt);
\fill[black] (1.5,0) circle (3pt);
\fill[black] (2.5,0) circle (3pt);
\draw (3.5,0) circle (3pt);
\draw [dashed] (0,0.5) -- (4,0.5);
\draw [thick] (3.5,2.5) -- (1.5,0.5) -- (0.5,1.5);
\foreach \y in {1,2}{
\draw (1+\y,0+\y) circle (3pt);
}
\fill[black]  (1,1) circle (3pt);
\fill[red] (0.5,1.5) circle (3pt);
\fill[red] (1.5,0.5) circle (3pt);
\fill[red] (2.5,1.5) circle (3pt);
\fill[red] (3.5,2.5) circle (3pt);
\end{tikzpicture}
}

\newcommand{\bbbx}{
\begin{tikzpicture}[scale=0.7,baseline=9mm]
\draw [thick] (0,0) -- (4,0);
\draw (1.5,0) circle (3pt);
\draw (3.5,0) circle (3pt);
\fill[black] (0.5,0) circle (3pt);
\fill[black] (2.5,0) circle (3pt);
\draw [dashed] (0,0.5) -- (4,0.5);
\draw [thick] (3.5,2.5) -- (1.5,0.5) -- (0.5,1.5);
\foreach \y in {1}{
\draw (1+\y,0+\y) circle (3pt);
}
\draw (3,2) circle (3pt);
\fill[black]  (1,1) circle (3pt);
\fill[red] (0.5,1.5) circle (3pt);
\fill[red] (1.5,0.5) circle (3pt);
\fill[red] (2.5,1.5) circle (3pt);
\fill[red] (3.5,2.5) circle (3pt);
\end{tikzpicture}
}

\newcommand{\ccc}{
\begin{tikzpicture}[scale=0.7,baseline=9mm]
\draw [thick] (0,0) -- (4,0);
\draw (0.5,0) circle (3pt);
\draw (2.5,0) circle (3pt);
\fill[black] (1.5,0) circle (3pt);
\fill[black] (3.5,0) circle (3pt);
\draw [dashed] (0,0.5) -- (4,0.5);
\draw [thick] (3.5,2.5) -- (1.5,0.5) -- (0.5,1.5);
\foreach \y in {1}{
\draw (1+\y,0+\y) circle (3pt);
}
\draw (3,2) circle (3pt);
\fill[black]  (1,1) circle (3pt);
\fill[red] (0.5,1.5) circle (3pt);
\fill[red] (1.5,0.5) circle (3pt);
\fill[red] (2.5,1.5) circle (3pt);
\fill[red] (3.5,2.5) circle (3pt);
\end{tikzpicture}
}

\newcommand{\axa}{
\begin{tikzpicture}[scale=0.7,baseline=9mm]
\draw [thick] (0,0) -- (3,0);
\draw (0.5,0) circle (3pt);
\fill[black] (1.5,0) circle (3pt);
\draw (2.5,0) circle (3pt);
\draw (1.5,-0.5) node {$x$};
\draw [dashed] (0,0.5) -- (3,0.5);
\draw [thick] (3,1) -- (0,4);
\foreach \y in {0,1,2,3}{
\fill[black] (3-\y,1+\y) circle (3pt);
}
\foreach \y in {1,2,3}{
\fill[red] (3.5-\y,0.5+\y) circle (3pt);
}
\end{tikzpicture}
}

\newcommand{\aax}{
\begin{tikzpicture}[scale=0.7,baseline=9mm]
\draw [thick] (0,0) -- (3,0);
\draw (0.5,0) circle (3pt);
\fill[black] (2.5,0) circle (3pt);
\draw (1.5,0) circle (3pt);
\draw (2.5,-0.5) node {$x+1$};
\draw [dashed] (0,0.5) -- (3,0.5);
\draw [thick] (3,1) -- (0,4);
\foreach \y in {0,1,2,3}{
\fill[black] (3-\y,1+\y) circle (3pt);
}
\foreach \y in {1,2,3}{
\fill[red] (3.5-\y,0.5+\y) circle (3pt);
}
\end{tikzpicture}
}

\newcommand{\xaa}{
\begin{tikzpicture}[scale=0.7,baseline=9mm]
\draw [thick] (0,0) -- (3,0);
\draw (1.5,0) circle (3pt);
\fill[black] (0.5,0) circle (3pt);
\draw (2.5,0) circle (3pt);
\draw (0.5,-0.5) node {$x-1$};
\draw [dashed] (0,0.5) -- (3,0.5);
\draw [thick] (3,1) -- (0,4);
\foreach \y in {0,1,2,3}{
\fill[black] (3-\y,1+\y) circle (3pt);
}
\foreach \y in {1,2,3}{
\fill[red] (3.5-\y,0.5+\y) circle (3pt);
}
\end{tikzpicture}
}

\newcommand{\axao}{
\begin{tikzpicture}[scale=0.7,baseline=9mm]
\draw [thick] (0,0) -- (3,0);
\draw (0.5,0) circle (3pt);
\fill[black] (1.5,0) circle (3pt);
\draw (2.5,0) circle (3pt);
\draw (1.5,-0.5) node {$x$};
\draw [dashed] (0,0.5) -- (3,0.5);
\draw [thick] (3,4) -- (0,1);
\foreach \y in {0,1,2,3}{
\draw (0+\y,1+\y) circle (3pt);
}
\foreach \y in {0,1,2}{
\fill[red] (0.5+\y,1.5+\y) circle (3pt);
}
\end{tikzpicture}
}

\newcommand{\aaxo}{
\begin{tikzpicture}[scale=0.7,baseline=9mm]
\draw [thick] (0,0) -- (3,0);
\draw (0.5,0) circle (3pt);
\fill[black] (2.5,0) circle (3pt);
\draw (1.5,0) circle (3pt);
\draw (2.5,-0.5) node {$x+1$};
\draw [dashed] (0,0.5) -- (3,0.5);
\draw [thick] (3,4) -- (0,1);
\foreach \y in {0,1,2,3}{
\draw (0+\y,1+\y) circle (3pt);
}
\foreach \y in {0,1,2}{
\fill[red] (0.5+\y,1.5+\y) circle (3pt);
}
\end{tikzpicture}
}

\newcommand{\xaao}{
\begin{tikzpicture}[scale=0.7,baseline=9mm]
\draw [thick] (0,0) -- (3,0);
\draw (1.5,0) circle (3pt);
\fill[black] (0.5,0) circle (3pt);
\draw (2.5,0) circle (3pt);
\draw (0.5,-0.5) node {$x-1$};
\draw [dashed] (0,0.5) -- (3,0.5);
\draw [thick] (3,4) -- (0,1);
\foreach \y in {0,1,2,3}{
\draw (0+\y,1+\y) circle (3pt);
}
\foreach \y in {0,1,2}{
\fill[red] (0.5+\y,1.5+\y) circle (3pt);
}
\end{tikzpicture}
}

\newtheorem{theorem}{Theorem}[section]
\newtheorem{prop}[theorem]{Proposition}
\newtheorem{lemma}[theorem]{Lemma}
\newtheorem{corollary}[theorem]{Corollary}
\theoremstyle{remark}
\newtheorem{remark}{Example}
\newtheorem{reemark}{Remark}
\newtheorem{definition}[theorem]{Definition}

\DeclareMathOperator{\b|}{\boldsymbol{|}}

\begin{document}

\title{Stochastic Fusion of Interacting Particle Systems and Duality Functions}

\author{Jeffrey Kuan}

\date{}

\maketitle

\abstract{We introduce a new method, which we call stochastic fusion, which takes an exclusion process and constructs an interacting particle systems in which more than one particle may occupy a lattice site. The construction only requires the existence of stationary measures of the original exclusion process on a finite lattice. If the original exclusion process satisfies Markov duality on a finite lattice, then the construction produces Markov duality functions (for some initial conditions) for the fused exclusion process. The stochastic fusion construction is based off of the Rogers--Pitman intertwining.

In particular, we have results for three types of models:

1. For symmetric exclusion processes, the fused process and duality functions are inhomogeneous generalizations of those in \cite{GKRV}. The construction also allows a general class of open boundary conditions: as an application of the duality, we find the hydrodynamic limit and stationary measures of the generalized symmetric simple exclusion process SSEP$(m/2)$ on $\mathbb{Z}_+$ for open boundary conditions. 

2. For the asymmetric simple exclusion process, the fused process and duality functions are inhomogeneous generalizations of those found in \cite{CGRS} for the ASEP$(q,j)$. As a by-product of the construction, we show that the multi--species ASEP$(q,j)$ preserves $q$--exchangeable measures, and use this to find new duality functions for the ASEP, ASEP$(q,j)$ and $q$--Boson. Additionally, the construction leads to duality for ASEP on the half--line with open boundary conditions. As an application of the latter duality, we find the hydrodynamic limit.

3. For dynamic models, we fuse the dynamic ASEP from \cite{BorodinDyn}, and produce a dynamic and inhomogeneous version of ASEP$(q,j)$. We also apply stochastic fusion to IRF models and compare them to previously found models.

We include an appendix, co-authored with Amol Aggarwal, which explains 
the algebraic roots of the model for the interested reader, as well as 
the relationship between fusion for interacting particle systems and 
fusion for stochastic vertex models. We compare the fused dynamical 
vertex weights with previously found dynamical vertex weights.
}

\tableofcontents

\section{Introduction}
The asymmetric exclusion process (ASEP), introduced in \cite{Spit70} and \cite{MGP68}, is an interacting particle system on a one--dimensional lattice where at most one particle may occupy each lattice site. Due to the simplicity of the model, it is often viewed as the ``canonical model for transport phenomena'' (\cite {HTYauAnnals}). If one wishes to construct so--called ``higher spin'' exclusion processes, where more than one particle may occupy a site, there is not necessarily an obviously ``canonical'' construction. In this paper, we present a new method to construct ``fused'' exclusion processes, using only the stationary measures of the original exclusion process. We call this method ``stochastic fusion.'' 

First, let us describe the construction in general. The construction of the fused exclusion process is based on Rogers--Pitman intertwining \cite{lrjp1}. Roughly speaking, Rogers--Pitman give an intertwining relation for which a function $\phi:S\rightarrow \hat{S}$ of a Markov chain $X_t$ on $S$ is Markov (for some initial conditions $X_0)$. If $P_t$ is the semigroup of transition probabilities of $X_t$, then the semigroup of transition probabilities of $\phi(X_t)$ is given by $\Lambda P_t \Phi$ where $\Lambda$ is a stochastic $\hat{S} \times S$ matrix and $\Phi$ is the $S \times \hat{S}$ matrix induced by $\phi$. Here, we will construct $\Lambda$ from the stationary measures of $X_t$. See Figure \ref{LP} for an illustration.

One way to view this fused process as the ``canonical'' fused process is through Markov duality. One application of Markov duality is to show weak asymmetry convergence to stochastic partial differential equations (see \cite{CST},\cite{YLKPZ},\cite{CGMDyn}) that describe the universality class. (In the case of ASEP, the stochastic PDE is the KPZ equation). Another application is to find asymptotics \cite{BCSDuality}. We would expect the fused process to be in the same universality class with the same asymptotics, and as such, Markov duality is a natural property that we expect it to have. And indeed, Markov duality can be described as an intertwining between transition semigroups; we will show that the aforementioned Rogers--Pitman intertwining leads naturally to the Markov duality intertwining for $\phi(X_t)$, for some initial conditions, if the original process $X_t$ satisfies Markov duality on a finite lattice. The precise statements will be found in Theorem \ref{RPThm} for the construction in general, and Proposition \ref{RPInter} for the specific case of symmetric processes. 

With the generalities described, we move on to describe the specific models considered here. There are three classes of models investigated here: symmetric exclusion processes, the asymmetric simple exclusion processes, and dynamic models. For the symmetric exclusion processes, the construction actually begins with multi--species symmetric exclusion processes, where there is exactly one particle of each species and all lattice sites are occupied. Such a process is also known as an \textit{interchange process}. It may seem counterintuitive to begin with the interchange process; but in several earlier works (\cite{BWColor}, \cite{KuanAHP}, \cite{BorodinBufetovCP}) it was seen that such processes can be easier to work with. After projecting to single--species models, we obtain inhomogeneous versions of the so--called SEP($m/2$) found in \cite{GKRV}. (Here, inhomogeneous in the sense that the maximum number $m_x$ of particles occupying a site $x$ can differ from site to site). By taking some $m_x$ to infinity, we can obtain lattice sites with infinitely many particles that can jump into the remainder of the lattice. This can be viewed as a symmetric exclusion process with open boundary conditions. We then find duality functions for the symmetric exclusion processes, for both closed boundary and a range of open boundary conditions (Proposition \ref{aaaa} and Theorem \ref{bbbb}), generalizing results in \cite{GKRV}, which required more specific open boundary conditions; and also generalizing results in the very recent paper \cite{CGRConsistent}, which considered specific initial conditions. As an application, we have results about the hydrodynamic limit and stationary measures of the SEP$(m/2)$ with open boundary conditions: see Theorem \ref{Opened}.

For the asymmetric simple exclusion process, the construction leads to an inhomogeneous generalization of the so-called ASEP$(q,j)$ of \cite{CGRS}; see Theorem \ref{SFA} below. (It is worth noting the results of \cite{Matsui2015} and \cite{Nachtergaele2004}, which give different constructions of interacting particle systems allowing more than one particle per site). We also recover their duality function through this construction as well, in Proposition \ref{FusedDuality}. As a byproduct of the construction, we show that the multi--species ASEP$(q,j)$ satisfies $q$--exchangeability, and after projecting to single--species models, we produce several new duality functions for the ASEP, ASEP$(q,j)$ and $q$--Boson: see section \ref{MSM}. Additionally, for a specific open boundary condition, the ASEP satisfies a duality, which is used to find the hydrodynamic limit; see Theorems \ref{OD} and \ref{Addend}.

For the dynamic models, we construct a dynamic version of ASEP$(q,j)$, starting from the dynamic ASEP constructed in \cite{BorodinDyn}. While there is Markov duality for dynamic ASEP on the infinite lattice \cite{BorCorIMRN}, there does not appear to be a Markov duality on the finite lattice. Note that the non--dynamic models have an underlying multi--species (or interchange process) model, unlike the dynamic ASEP, and this is reflected in the stochastic fusion construction; see remarks \ref{Constant} and \ref{Constant2} for an elaboration. We also apply the stochastic fusion construction to dynamical stochastic vertex weights (also called interaction-round-a-face, or IRF weights) and compare the weights to previous models in \cite{AmolIRF},\cite{BorodinDyn}: see the appendix.

We conclude the introduction by briefly comparing this method to some other previously described methods. The general method of \cite{CGRS} requires an underlying algebra of symmetries of the original process $X_t$, and produces duality and stationary measures. Here, we only require stationary measures of $X_t$, \textit{without} requiring an underlying algebra, which is a strictly weaker requirement. While this does come at the cost of not automatically proving a duality (for all initial conditions), it does provide reasonable Ans\"{a}tze which seem to work. 

Previous works of \cite{CorPetCMP} and \cite{BP} used fusion in a different context of stochastic vertex models, whereas we consider continuous--time interacting particle systems here. While the stochastic six vertex model degenerates to ASEP, this degeneration does not hold for higher spin models (due to non--negativity not holding -- see Remark 5.2 of \cite{CorPetCMP}), so vertex models and particle systems require different constructions. Another distinction is that stochastic vertex models require weights that solve the Yang--Baxter equation, whereas interacting particles do not. Note that the connection between Rogers--Pitman intertwining and fusion for vertex models was briefly remarked upon in section 3.2 of \cite{KuanCMP}. See the appendix for the algebraic background behind fusion.

The paper is organized as follows. Section 2 gives the background and notations. Section 3 gives the generalities of the stochastic fusion construction. To avoid proving theorems twice, the results for the symmetric exclusion process and the asymmetric simple exclusion process are combined in section 4.   Section 5 covers the dynamic models. The appendix explains the algebraic roots of the model for the interested reader, as well as the relationship between fusion for interacting particle systems and fusion for stochastic vertex models.

\textbf{Acknowledgements.} The author was supported by NSF grant DMS--1502665 and the Minerva Foundation. The author thanks Alexei Borodin and Ivan Corwin for valuable insights, and would also like to especially thank Amol Aggarwal for comments throughout the paper, and particularly for section 4.1.3. Both J.K. and A.A. were supported by NSF grant DMS--1664650. The author also thanks Axel S\'{a}enz for pointing out an initial mistake in the hydrodynamic limit of open ASEP. 

\section{Background and Notations}

\subsection{Rogers--Pitman Intertwining}
Recall the setup of Rogers and Pitman \cite{lrjp1}. Below we will write the operators will be composed from left to right, rather than from right to left. 

Let $X_t$ be a Markov Process on state space $S$ with transition semigroup $P_t$. Let $\Lambda$ be a Markov kernel from $\hat{S}$ to $S$; i. e. $\Lambda$ is a stochastic matrix with rows indexed by $\hat{S}$ and columns indexed by $S$ (assuming $S, \hat{S}$ are countable and discrete). For any map $\phi:S \rightarrow \hat{S}$, let $\Phi$ be the Markov kernel from $S$ to $\hat{S}$ defined by
$$
\Phi f = f \circ \phi.
$$

\begin{figure}
\begin{center}
\begin{tikzpicture}
\draw (1,2) circle (3pt);
\draw (4,2) circle (3pt);
\fill[black] (5,2) circle (3pt);
\draw (7,2) circle (3pt);
\fill[black] (2,2) circle (3pt);
\fill[red] (3,2) circle (3pt);
\fill[black] (6,2) circle (3pt);
\draw [very thick](0.5,2) -- (7.5,2);
\draw (8,2) node (a) { $ S$};
\draw (8,0) node (b) { $ \hat{S}$};
\draw (0,2) node (c) { $$};
\draw (0,0) node (d) { $$};
\draw (a) edge[out=-45,in=45,->, line width=0.5pt] (b);
\draw (d) edge[out=135,in=-135,->, line width=0.5pt] (c);
\draw (0.7,2) node {\Huge $[$};
\draw (2.3,2) node {\Huge $]$};
\draw (2.7,2) node {\Huge $[$};
\draw (5.3,2) node {\Huge $]$};
\draw (5.7,2) node {\Huge $[$};
\draw (7.3,2) node {\Huge $]$};
\draw (0.7,0) node  {\Huge $[$};
\draw (2.3,0) node {\Huge $]$};
\draw (2.7,0) node {\Huge $[$};
\draw (5.3,0) node {\Huge $]$};
\draw (5.7,0) node {\Huge $[$};
\draw (7.3,0) node {\Huge $]$};
\draw [very thick](0.5,-0) -- (7.5,-0);
\draw (1.5,0.2) circle (3pt);
\fill[black] (4,0) circle (3pt);
\draw (4,0.4) circle (3pt);
\draw (6.5,0.2) circle (3pt);
\fill[black] (1.5,-0.2) circle (3pt);
\fill[red] (4,-0.4) circle (3pt);
\fill[black] (6.5,-0.2) circle (3pt);
\draw (9,1) node {$\Phi$};
\draw (-1,1) node {$\Lambda$};
\end{tikzpicture}
\end{center}
\caption{The fusion map $\Phi$ and the fission map $\Lambda$. Red particles are species 1 and black particles are species 2. If $\mathfrak{s}(x)$ represents the particle configuration on the top line, then $\mathfrak{s}(0)=\mathfrak{s}(3)=\mathfrak{s}(6)=0$, $\mathfrak{s}(2)=1$ and $\mathfrak{s}(1)=\mathfrak{s}(4)=\mathfrak{s}(5)=2$.}
\label{LP}
\end{figure}
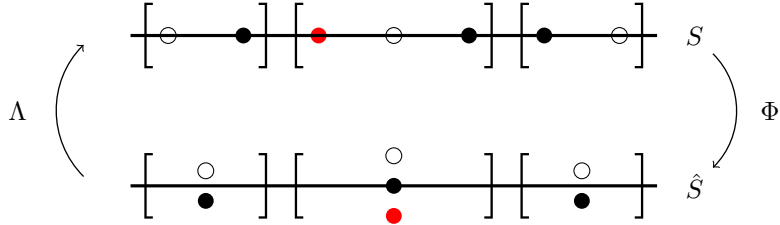

\begin{theorem}[\cite{lrjp1}]
Suppose 
\begin{itemize}
\item
$\Lambda\Phi$ is the identity on $\hat{S}$.
\item
$\Lambda P _ { t } = \Lambda P _ { t } \Phi \Lambda$ on $\hat{S}.$ 
\end{itemize}
Fix $y\in \hat{S}$. Let the initial condition of the Markov process $X_t$ be $X_{ 0} = \Lambda \left( y ,\cdot \right)$. Then $\phi(X_t)$ is a Markov process starting from $y$, with transition probabilities
$$
Q _ { t } := \Lambda P _ { t } \Phi.
$$
\end{theorem}

The Rogers--Pitman only allows one to conclude that $\phi(X_t)$ is Markov for some initial conditions $X_0$. To conclude that $\phi(X_t)$  is Markov for every initial condition $X_0$, stronger assumptions are needed.
\begin{corollary}
[\cite{KS60}] Suppose that $\Lambda P_t = \Lambda P_t \Phi \Lambda$ holds for every $\Lambda$ satisfying $\Lambda\Phi = \mathrm{Id}_{\hat{S}}$. Then $\phi(X_t)$ is a Markov process for every initial condition $X_0$, with transition probabilities $Q _ { t } := \Lambda P _ { t } \Phi$, which are independent of the choice of $\Lambda$.

\end{corollary}

\subsection{Notation}
 
 \subsubsection{$q$--Notation}

Define the $q$--deformed integers, factorials, binomials, and multinomials by
\begin{align*}
(n)_q &= 1 + q + \ldots + q^{n-1} = \frac{1-q^n}{1-q}, \\
(n)_q^! &= (1)_q \cdots (n)_q, \\
\binom{n}{k}_q &= \frac{(n)_q^!}{(k)_q^!(n-k)_q^!}, \\
\binom{n}{k_1,\ldots,k_r}_q &= \frac{ (n)_q^!}{ (k_1)_q^! \cdots (k_r)_q^! (n-k_1-\cdots-k_r)_q^!}.
\end{align*}
The $q$--binomials satisfy
$$
\binom{m-1}{k-1}_q + q^k \binom{m-1}{k}_q = \binom{m}{k}_q, \quad \quad q^{m-k}\binom{m-1}{k-1}_q +  \binom{m-1}{k}_q = \binom{m}{k}_q.
$$
Recall the $q$--binomial theorem \cite{AndrewsBook}
\begin{equation}\label{qBin}
\sum_{0 \leq x_k < \ldots < x_1 \leq m-1} q^{x_1+\ldots+x_k} = q^{(k-1)k/2}\binom{m}{k}_q.
\end{equation}

\subsubsection{State Space Notation}
There are two different ways of denoting a particle configuration: \textit{particle variables} or \textit{occupation variables}. Let us describe the former first.  

For each $1 \leq i \leq n$, let $\ldots, m_{-1}^{(i)},m_0^{(i)},m_1^{(i)},\ldots$ be a doubly infinite sequence of nonnegative integers. For brevity of notation let $\vec{m}^{(i)}=(m^{(i)}_x)_{x \in \mathbb{Z}}$ denote this sequence. Set
$$
\vert \vec{m}^{(i)} \vert = \sum_{x\in \mathbb{Z}} m^{(i)}_x, \quad \vert \vec{m} \vert = \sum_{i=1}^n \vert \vec{m}^{(i)} \vert,
$$
where $\infty$ is a permissible value for $\vert \vec{m}^{(i)}\vert$. Also set $m_x=\sum_{i=1}^n m_x^{(i)}$.

Here, let $S^{(n)}$ be the state space consisting of particle configurations on a lattice with $\vert \vec{m}\vert$ sites, when all $m_x=1$. In other words,
$$
S^{(n)} = \{0,1,\ldots,n\}^{\vert \vec{m}\vert}.
$$
We write $\mathfrak{s}(x)$ for an element of $S^{(n)}$. See Figure \ref{LP} for an example.

Let $\hat{S}^{(n)}$ be the state space of particle configurations on $\mathbb{Z}$, where up to $m_x$ particles may occupy lattice site $x$. In other words, 
$$
 \hat{S}^{(n)} =  \prod_{x \in \mathbb{Z} } \{(k_x^{(1)},\ldots,k_x^{(n)}): k_x^{(1)} + \ldots + k_x^{(n)} \leq m_x, \ \ k_x^{(i)} \in \mathbb{N}\}.
$$
We write $\hat{\mathfrak{s}}(x)$ for an element of $\hat{S}^{(n)}$, where each $\hat{\mathfrak{s}}(x)=(k_x^{(1)},\ldots,k_x^{(n)})$ is a sequence of integers, rather than a single integer. Note that if only finitely many $m_x$ are nonzero, then the particle configurations on $S^{(n)}$ and $\hat{S}^{(n)}$ live on a finite lattice.

For any sequence of nonnegative integers $\mathbf{N}=(N_1,\ldots,N_n)$, define 
$
S^{(n)}(\mathbf{N})\subseteq S^{(n)}
$
by 
$$
S^{(n)}(\mathbf{N}) = \{ \mathfrak{s} \in S^{(n)} :  \sum_{x \in \mathbb{Z}} 1_{\mathfrak{s}(x) = i } = N_i \text{ for } 1 \leq i \leq n \}.
$$
Similarly define $\hat{S}^{(n)}(\mathbf{N}) \subseteq \hat{S}^{(n)}$ by 
$$
\hat{S}^{(n)}(\mathbf{N}) = \{ \hat{\mathfrak{s}} \in \hat{S}^{(n)} :  \sum_{x \in \mathbb{Z}} k_x^{(i)} = N_i \text{ for } 1 \leq i \leq n \}.
$$

There are also natural projection maps between $S^{(n)}$ for different values of $n$. Let $\Pi$ be a partition of $\{0,1,\ldots,n\}$ into $p+1$ parts of consecutive integers. In other words, 
$$
\Pi = \{ \{0,\ldots,N_0\}, \{N_0+1,\ldots,N_1\}, \{N_{p-1}+1,\ldots,N_p\}\},
$$
where $N_p=n$. There is a corresponding map $S^{(n)}$ to $S^{(p)}$ defined by sending $i$ to the block to which it belongs. For example, if $n=9$ and $\Pi=\{\{0\},\{1,2\},\{3,4,5\},\{6,7,8,9\}\}$, then the projection from $S^{(9)}$ to $S^{(3)}$ sends $(8,6,7,5,3,0,9)$ to $(3,3,3,2,2,0,3)$. The partition $\Pi$ also maps $\hat{S}^{(n)}$ to $\hat{S}^{(p)}$ by sending
$$
(k_x^{(1)},\ldots,k_x^{(n)}) \rightarrow (k_x^{(N_0+1)}+ \ldots + k_x^{(N_1)},\ldots,k_x^{(N_{p-1}+1)} + \ldots + k_x^{(N_p)} ).
$$

The other convenient\footnote{This notation is common in the literature on Markov duality, where the original process is written in particle variables and the dual process is written in occupation variables.} notation for describing a particle configuration is with \textit{occupation variables}, assuming the configuration has only finitely many particles. For each $1\leq j\leq n$, if there are $N_j$ particles of species $j$, then write their locations as $\vec{x}^{(j)} = (x_{N_j}^{(j)} \leq \ldots \leq x_1^{(j)})$. Let $\vec{x}= (\vec{x}^{(1)},\ldots,\vec{x}^{(n)})$ denote the entire particle configuration.

A single--species (i.e. $n=1$) particle configuration can also be viewed as a function on $s:\mathbb{Z} \rightarrow \mathbb{Z}$. If $s(x+1)=s(x)-1$, then place a particle at lattice site $x$. Conversely, if $s(x+1)=s(x)+1$, then there is no particle at lattice site $x$. See Figure \ref{sf} for an example. Particle configurations do not determine a function $s(x)$ uniquely: observe that two functions $s,s'$ define the same particle configuration if and only if $s-s'$ is a constant. This results in an additional parameter, which we will be called the dynamic parameter.

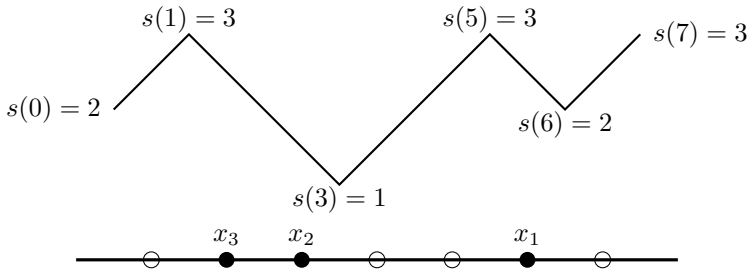
\begin{figure}
\caption{In this example, $x_3=1,x_2=2,x_1=5$.}
\begin{center}
\begin{tikzpicture}
\draw (-0.3,2) node {$s(0)=2$};
\draw (1.5,3.2) node {$s(1)=3$};
\draw (3.5,0.8) node {$s(3)=1$};
\draw (5.5,3.2) node {$s(5)=3$};
\draw (6.5,1.8) node {$s(6)=2$};
\draw (8.3,3) node {$s(7)=3$};
\draw (2,0.3) node {$x_3$};
\draw (3,0.3) node {$x_2$};
\draw (6,0.3) node {$x_1$};
\draw [very thick](0,0) -- (8,0);
\draw (1,0) circle (3pt);
\draw (4,0) circle (3pt);
\draw (5,0) circle (3pt);
\draw (7,0) circle (3pt);
\fill[black] (2,0) circle (3pt);
\fill[black] (3,0) circle (3pt);
\fill[black] (6,0) circle (3pt);
\draw [thick](0.5,2) -- (1.5,3) -- (3.5,1) -- (5.5,3) -- (6.5,2) -- (7.5,3);
\end{tikzpicture}
\end{center}
\label{sf}
\end{figure}

\subsection{Single--species interacting particle systems} 

\subsubsection{Dynamic ASEP}
The dynamic ASEP was first defined in \cite{BorodinDyn}. The state space of dynamic ASEP on $\mathbb{Z}$ is the set of functions $s: \mathbb{Z}\rightarrow \mathbb{Z}$ such that $s(x+1) = s(x) \pm 1$ for all $x\in \mathbb{Z}$. If $s,s'$ are two states such that $s(y)=s'(y)$ for all $y\neq x$, the jump rates from $s$ to $s'$ are
\begin{align*}
\frac{q(1+\alpha q^{-s(x)})}{1 + \alpha q^{-s(x)+1}}, &  \text{ if } s'(x) = s_x - 2, \\
\frac{1+\alpha q^{-s(x)}}{1 + \alpha q^{-s(x)-1}}, &  \text{ if } s'(x) = s_x + 2.
\end{align*}
All other jump rates are $0$. For the jump rates from $s$ to $s'$ to be nonzero, we must have that $s(x-1)=s(x+1)$, the value of which we write as $s(x\pm 1)$. It will be more convenient to write the jump rates as 
\begin{align*}
q\cdot \frac{(1+\alpha q^{-s(x)})}{1 + \alpha q^{-s(x \pm 1)}}, &  \text{ if } s'(x) = s_x - 2, \\
\frac{1+\alpha q^{-s(x)}}{1 + \alpha q^{-s(x \pm 1)}}, &  \text{ if } s'(x) = s_x + 2.
\end{align*}

\begin{figure}
\begin{center}
\begin{tikzpicture}
\draw [thick] (-1.5,0) -- (0,-1.5) -- (1.5,0);
\draw [dashed] (-1.5,0) -- (0,1.5) -- (1.5,0); 
\draw (0,-1.5) node (bottom)  {};
\draw (0,1.5) node (top) {};
\draw (bottom) edge[out=120,in=-120,->,dashed] (top);
\draw (0,-1.75) node {$s(x)$};
\draw (-2.25,0) node {$s(x-1)$};
\draw (2.25,0) node {$s(x+1)$};
\draw [dashed] (6,0) -- (7.5,-1.5) -- (9,0);
\draw [thick] (6,0) -- (7.5,1.5) -- (9,0); 
\draw (7.5,-1.5) node (bottom9)  {};
\draw (7.5,1.5) node (top9) {};
\draw (top9) edge[in=60,out=-60,->,dashed] (bottom9);
\draw (0.4,0) node {$\tfrac{1+\alpha q^{-s(x)}}{1 + \alpha q^{-s(x \pm 1)}}$};
\draw (7.1,0) node {$ \tfrac{q(1+\alpha q^{-s(x)})}{1 + \alpha q^{-s(x \pm 1)}}$};
\draw (0+7.5,1.75) node {$s(x)$};
\draw (-2.25+7.5,0) node {$s(x-1)$};
\draw (2.25+7.5,0) node {$s(x+1)$};
\end{tikzpicture}
\end{center}
\end{figure}

A stationary measure was found in \cite{BorCorIMRN}. This measure is defined by
\begin{align*}
\mathbb{P}(s(0)=2n) &= \frac{\alpha^{-2n}q^{n(2n-1)}(1+\alpha^{-1}q^{2n})}{(-\alpha^{-1},-q\alpha,q;q)_{\infty}},\\
\mathbb{P}(s(x-1)=s(x) + 1) &= \frac{q^{s(x)}}{\alpha + q^{s(x)}} \text{ for all } x\in \mathbb{Z},\\
\mathbb{P}(s(x-1)=s(x) - 1) &= \frac{\alpha}{\alpha + q^{s(x)}} \text{ for all } x\in \mathbb{Z}.
\end{align*}
Note that this is not the only stationary measure, since another stationary measure is the deterministic measure defined by $s(x)=x$ for all $x\in \mathbb{Z}$.

If $\alpha\rightarrow 0$, then dynamic ASEP becomes a usual ASEP with right jump rates $1$ and left jump rates $q$. If $\alpha\rightarrow \infty$, then dynamic ASEP becomes (after rescaling in time) a usual ASEP with right jump rates $q$ and left jump rates $1$. In this way, dynamic ASEP interpolates between two ASEPs. To keep the notation consistent with ASEP$(q,m/2)$, defined in the next subsection, we will call ASEP$_{1,q}$ the ASEP with left jump rates $1$ and right jump rates $q$. Similarly, ASEP$_{q,1}$ is the ASEP with left jump rates $q$ and right jump rates $1$.

\subsubsection{Dynamic SSEP}\label{dynSym}
Define the shifted sequences $\tilde{s}(x) := s(x) - \log_q \vert \alpha\vert$, which still satisfy $\tilde{s}(x+1) - \tilde{s}(x) \in \{-1,1\}$ but do not need to be integer valued. The jump rates them become 
$$
\tilde{s}(x) \mapsto \tilde{s}(x)-2 \text{ with rate } \frac{ q(1 \pm q^{-\tilde{s}(x)})}{1 \pm q^{-\tilde{s}(x\pm 1)}}, \quad \quad \tilde{s}(x) \mapsto \tilde{s}(x)+2 \text{ with rate } \frac{ 1 \pm q^{-\tilde{s}(x)}}{1 \pm q^{-\tilde{s}(x\pm 1)}}, 
$$
where the $\pm$ in the ``$1 \pm$'' is chosen to be $+$ if $\alpha>0$ and chosen to be $-$ if $\alpha <0$. If $\alpha<0$ and the $q\rightarrow 1$ limit is taken, the jump rates become 
$$
\tilde{s}(x) \mapsto \tilde{s}(x)-2 \text{ with rate } \frac{ \tilde{s}(x) }{ \tilde{s}(x\pm 1) }, \quad \quad \tilde{s}(x) \mapsto \tilde{s}(x)+2 \text{ with rate } \frac{ \tilde{s}(x) }{ \tilde{s}(x\pm 1) }.
$$
This can be equivalently written in terms of the unshifted sequences $s(x) := \tilde{s}(x) + \lambda$ as  
$$
{s}(x) \mapsto {s}(x)-2 \text{ with rate } \frac{ {s}(x) - \lambda }{ {s}(x\pm 1) - \lambda}, \quad \quad {s}(x) \mapsto {s}(x)+2 \text{ with rate } \frac{ {s}(x) - \lambda }{ {s}(x\pm 1)- \lambda  }.
$$
The parameter $\lambda$ is chosen so that the jump rates are always non--negative. This happens, for example, if $\lambda <c$ where $c$ is a constant such that the initial conditions satisfy $s(x) \geq \vert x \vert + c$ for all $x \in \mathbb{Z}$. 

When $\lambda$ is taken to $-\infty$, the jump rates simplify to the usual SSEP (symmetric simple exclusion process). However, note that the while dynamic SSEP satisfies spatial symmetry, it is not symmetric in the sense that $L_{dSSEP} \neq L_{dSSEP}^*$. Indeed, for a fixed value of $s(x)$,
$$
 \{ s(x)\mapsto s(x)-2\ \  \text{jump rates}\}  > 1 >  \{ s(x)\mapsto s(x)+2\ \  \text{jump rates}\}.
$$

\subsubsection{ASEP$(q,\vec{m})$ and $q$--Boson}

We define the following continuous--time Markov process on $\hat{S}$, which we call ASEP$(q,\vec{m})$. At each lattice site $x \in \mathbb{Z}$, suppose that there are $k_x$ particles, where $0 \leq k_x \leq m_x$. At most one particle may leave a site at a time, and when it does, it either jumps one step to the right or one step to the left. The jump rates for a particle to leave $x-1$ and enter $x$ are
$$
q\cdot q ^ { m_{x-1} - k _ { x - 1} } \cdot q ^ { k _ { x} } \left( \frac { 1- q ^ { k _ { x - 1} } } { 1- q ^ { m_{x-1} } } \right) \left( \frac { 1- q ^ { m_{x} - k _ { x} } } { 1- q ^ { m_{x} } } \right),
$$
and the jump rates for a particle to leave $x$ and enter $x-1$ are
$$
\left( \frac { 1- q ^ { m _{x-1}- k _ { x-1} } } { 1- q ^ { m_{x-1} } } \right) \left( \frac { 1- q ^ { k _ { x} } } { 1- q ^ { m_x } } \right).
$$
When all $m_x$ are the same integer $m$, then the process reduces to one introduced in \cite{CGRS}\footnote{In \cite{CGRS}, the jump rates are written using the $q$--integers $[n]_q=(q^n-q^{-n})/(q-q^{-1}) = q^{-n+1}(n)_{q^2}$. The right jump rates were defined as
$$
q^{k_x-k_{x+1}-(m+1)} [k_x]_q [m-k_{x+1}]_q = q^{-2m+1} (k_x)_{q^2} (m-k_{x+1})_{q^2} 
$$ 
and the left jump rates were defined as
$$
q^{k_{x-1}-k_x+(m+1)} [m-k_{x-1}]_q [k_x]_q = q^{2k_{x-1}-2k_x+3} (m-k_{x-1})_{q^2} (k_x)_{q^2} .
$$ 
Rescaling the time by $q^{2m-1}/(m)_{q^2}^2$, reversing the left--right directions, and replacing $q^2$ by $q$ yields the jump rates defined in the present paper.
}.

We note a few degenerations. In the $q\rightarrow 1$ limit, the jump rates become
$$
\frac{k_{x-1}}{m_{x-1}} \frac{ m_{x} - k_{x}}{m_{x}} \quad \text{  and  } \quad \frac{m_{x-1}-k_{x-1}}{m_{x-1}} \frac{k_x}{m_x}.
$$

If all $m_x=1$, then ASEP$(q,\vec{m})$ is the usual ASEP. 

In the $m_x\rightarrow \infty$ limit, and assuming $0<q<1$, the jump rates for right jumps converge to $0$, and the jump rates for left jumps converge to $1-q^{k_x}$. These are the jump rates for the $q$--Boson, introduced in \cite{SW98}.

\subsubsection{SEP($\vec{m}$)}\label{alm}
Here, we define the symmetric exclusion process; note that the word ``simple'' is not included, so the SEP($\vec{m}$) is not merely a special case of ASEP$(q,\vec{m})$. Suppose that $\mathcal{G}$ is a complete graph, and $p$ is a symmetric stochastic matrix on $\mathcal{G}$. At each vertex $x\in \mathcal{G}$, up to $m_x$ particles may occupy the site; i.e. $0 \leq \hat{\mathfrak{s}}(x) \leq m_x$, where $n=1$. Given a state $\mathfrak{s}(x)$, the jump rate for a particle from vertex $x$ to vertex $y$ is given by
$$
p(x,y) \hat{\mathfrak{s}}(x) (m_y - \hat{\mathfrak{s}}(y)). 
$$
This processes was introduced in \cite{GKRV}. When all $m_x=1$, this is just the symmetric exclusion process defined in \cite{Spit70}. 

There are also multi--species versions of the ASEP, ASEP$(q,j)$ and $q$--Boson, introduced in \cite{KuanJPhysA,BS,BS2}, \cite{KIMRN} and \cite{Take15}, respectively. The two--species ASEP was introduced in \cite{Ligg76}. Let us now describe the multi--species ASEP in the next section.

\subsection{Multi--species particle systems}
We define the meaning of the word ``multi--species'' in this paper. Given any partition $\Pi$ of $\{0,1,\ldots,n\}$ into $p+1$ blocks, the corresponding map from $\hat{S}^{(n)}$ to $\hat{S}^{(p)}$ defines a $\hat{S}^{(n)} \times \hat{S}^{(p)}$ matrix, which by abuse of notation is also denoted $\Pi$. If $Q_t$ is a semigroup of transition probabilities on $\hat{S}^{(1)}$, then a \textit{multii--species} version of $Q_t$ is a family of transition probabilities on $Q_t^{(1)},\ldots,Q_t^{(n)}$ on $\hat{S}^{(1)},\ldots,\hat{S}^{(n)}$ such that
$$
\Pi Q_t^{(p)} = Q_t^{(n)} \Pi
$$
for every partition $\Pi$ of $\{0,1,\ldots,n\}$ into $p+1$ blocks.

Similarly, given a function $D$ on $\hat{S}^{(1)} \times \hat{S}^{(1)}$, a \textit{multi--species} version of $D$ is a family of functions $D^{(p,n)}$ on $\hat{S}^{(p)} \times \hat{S}^{(n)}$ for $p,n \in \mathbb{Z}$ such that
$$
D^{(p,p)} \Pi^* = D^{(p,n)},
$$
where $\Pi = \{0,\{1,\ldots,n-p+1\},2,\ldots,n\}$ and the $^*$ denotes transposition.

Note that if $Q_t^{(p)} D^{(p,p)} = D^{(p,p)} Q_t^{(p)*}$, then
\begin{align*}
Q_t^{(p)} D^{(p,n)}&= Q_t^{(p)} D^{(p,p)} \Pi^*  = D^{(p,p)} Q_t^{(p)*}\Pi^* \\
&= D^{(p,p)} (\Pi Q_t^{(p)})^* = D^{(p,p)} (Q_t^{(n)} \Pi)^* = D^{(p,p)}\Pi^* Q_t^{(n)*} = D^{(p,n)} (   Q_t^{(n)}    )^*.
\end{align*}

Additionally, if $Q_t^{(p)}$ and $Q_t^{(n)}$ are symmetric, and $\Pi,\bar{\Pi}$ are two partitions of $\{0,1,\ldots,n\}$ into $p+1$ blocks, then
\begin{align*}
Q_t^{(n)} D^{(n,n)} = D^{(n,n)}Q_t^{(n)^*} &\Longrightarrow \bar{\Pi}^* Q_t^{(n)} D^{(n,n)} \Pi = \bar{\Pi}^*D^{(n,n)}Q_t^{(n)}\Pi \\
& \Longrightarrow Q_t^{(p)} \cdot \bar{\Pi}^* D^{(n,n)} \Pi = \bar{\Pi}^* D^{(n,n)} \Pi \cdot Q_t^{(p)*}.
\end{align*}
In general, multi--species versions of processes and dualities do not need to be unique.
\subsubsection{ASEP} 

The state space of $n$--species ASEP$_{1,q}$ consists of particle configurations where at most one particle may occupy a site. There are $n$ different species of particles, labeled $1,\ldots,n$. For $i<j$, we think of particles of species $i$ as lighter than particles of species $j$. Each particle independently tries to jump left at rate $1$ and right at rate $q$. If the particle attempts to jump to a site occupied by a heavier particle, the jump is blocked. If the particle attempts to jump to a site occupied by a lighter particle, the two particles switch places. 

The $n$--species ASEP$_{1,q}$ can also be viewed as $n$ coupled copies of ASEP$_{1,q}$ (this is how it was initially described in \cite{Ligg76}). For each $j$, consider the projection that sends particles of species $\geq j$ to particles and particles of species $<j$ to holes. (See Figure \ref{MS} for an example). Then each projection is Markov and evolves as an ($1$--species) ASEP$_{1,q}$, for a total of $n$ copies of ASEP$_{1,q}$.

Reversible measures of $n$--species ASEP$_{1,q}$ on a finite interval were found in \cite{BS3,KIMRN}. (See also \cite{Martin18} for stationary measures on an infinite line or on a ring). The expression is
$$
\pi(\mathfrak{s}) \propto \prod_{y<x} q^{1_{\{ \mathfrak{s}(y) < \mathfrak{s}(x)\}}}.
$$
Note that if $n=1$, then $\pi$ can be written in terms of particle variables as
$$
\pi(\vec{x}) \propto q^{x_1 + \ldots + x_N}.
$$

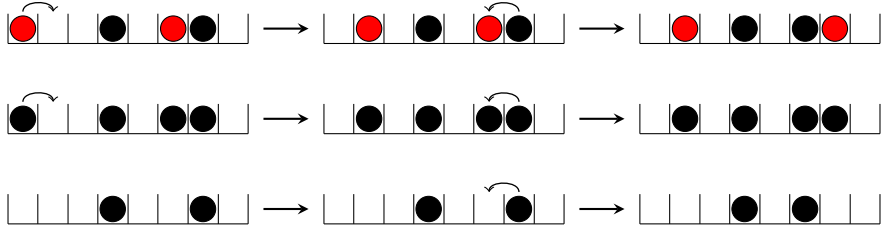
\begin{figure}
\caption{Red particles are species $1$ (lighter) and black particles are species $2$ (heavier). The first line shows a two--species ASEP, the second line shows the evolution of particles of species $\geq 1$, and the third line shows the evolution of particles of species $\geq 2$.}
\begin{center}
\begin{tikzpicture}[scale=0.6, every text node part/.style={align=center}]
\usetikzlibrary{arrows}
\usetikzlibrary{shapes}
\usetikzlibrary{shapes.multipart}
\tikzstyle{arrow}=[->,>=stealth,thick,rounded corners=4pt]

\draw (-1.66,0)--(3.66,0);
\draw (-1.66,0)--(-1.66,0.66);
\draw (-1,0)--(-1,0.66);
\draw (-0.33,0)--(-0.33,0.66);
\draw (0.33,0)--(0.33,0.66);
\draw (1,0)--(1,0.66);
\draw (1.66,0)--(1.66,0.66);
\draw (2.33,0)--(2.33,0.66);
\draw (3,0)--(3,0.66);
\draw (3.66,0)--(3.66,0.66);
\draw[fill=red] (-1.33,0.33) circle (8pt);
\draw[fill=black] (0.66,0.33) circle (8pt);
\draw[fill=red] (2,0.33) circle (8pt);
\draw[fill=black] (2.66,0.33) circle (8pt);

\node at (-1.33,0.5) (abc) {};
\node at (-.66,0.5) (abd) {};
\node at (2,0.5) {};
\node at (2.66,0.5) {};
\draw (abc) edge[out=90,in=90,->, line width=0.5pt] (abd);

\draw[arrow] (4,0.33)--(5,0.33);

\draw (7-1.66,0)--(7+3.66,0);
\draw (7+-1.66,0)--(7+-1.66,0.66);
\draw (7+-1,0)--(7+-1,0.66);
\draw (7+-0.33,0)--(7+-0.33,0.66);
\draw (7+0.33,0)--(7+0.33,0.66);
\draw (7+1,0)--(7+1,0.66);
\draw (7+1.66,0)--(7+1.66,0.66);
\draw (7+2.33,0)--(7+2.33,0.66);
\draw (7+3,0)--(7+3,0.66);
\draw (7+3.66,0)--(7+3.66,0.66);
\draw[fill=red] (7-0.66,0.33) circle (8pt);
\draw[fill=black] (7+0.66,0.33) circle (8pt);
\draw[fill=red] (7+2,0.33) circle (8pt);
\draw[fill=black] (7+2.66,0.33) circle (8pt);

\node at (7+2.66,0.5) (XYZ) {};
\node at (7+2,0.5) (xyw) {};
\node at (7+2,0.5) {};
\node at (7+2.66,0.5) {};
\draw (XYZ) edge[out=90,in=90,->, line width=0.5pt] (xyw);

\draw[arrow] (7+4,0.33)--(7+5,0.33);

\draw (14-1.66,0)--(14+3.66,0);
\draw (14+-1.66,0)--(14+-1.66,0.66);
\draw (14+-1,0)--(14+-1,0.66);
\draw (14+-0.33,0)--(14+-0.33,0.66);
\draw (14+0.33,0)--(14+0.33,0.66);
\draw (14+1,0)--(14+1,0.66);
\draw (14+1.66,0)--(14+1.66,0.66);
\draw (14+2.33,0)--(14+2.33,0.66);
\draw (14+3,0)--(14+3,0.66);
\draw (14+3.66,0)--(14+3.66,0.66);
\draw[fill=red] (14-0.66,0.33) circle (8pt);
\draw[fill=black] (14+0.66,0.33) circle (8pt);
\draw[fill=black] (14+2,0.33) circle (8pt);
\draw[fill=red] (14+2.66,0.33) circle (8pt);

\node at (14+2.66,0.5)  {};
\node at (14+2,0.5)  {};
\node at (14+2,0.5) {};
\node at (14+2.66,0.5) {};

\draw (-1.66,-2)--(3.66,-2);
\draw (-1.66,0-2)--(-1.66,0.66-2);
\draw (-1,0-2)--(-1,0.66-2);
\draw (-0.33,0-2)--(-0.33,0.66-2);
\draw (0.33,0-2)--(0.33,0.66-2);
\draw (1,0-2)--(1,0.66-2);
\draw (1.66,0-2)--(1.66,0.66-2);
\draw (2.33,0-2)--(2.33,0.66-2);
\draw (3,0-2)--(3,0.66-2);
\draw (3.66,0-2)--(3.66,0.66-2);
\draw[fill=black] (-1.33,0.33-2) circle (8pt);
\draw[fill=black] (0.66,0.33-2) circle (8pt);
\draw[fill=black] (2,0.33-2) circle (8pt);
\draw[fill=black] (2.66,0.33-2) circle (8pt);

\node at (-1.33,0.5-2) (abc) {};
\node at (-.66,0.5-2) (abd) {};
\node at (2,0.5-2) {};
\node at (2.66,0.5-2) {};
\draw (abc) edge[out=90,in=90,->, line width=0.5pt] (abd);

\draw[arrow] (4,0.33-2)--(5,0.33-2);

\draw (7-1.66,0-2)--(7+3.66,0-2);
\draw (7+-1.66,0-2)--(7+-1.66,0.66-2);
\draw (7+-1,0-2)--(7+-1,0.66-2);
\draw (7+-0.33,0-2)--(7+-0.33,0.66-2);
\draw (7+0.33,0-2)--(7+0.33,0.66-2);
\draw (7+1,0-2)--(7+1,0.66-2);
\draw (7+1.66,0-2)--(7+1.66,0.66-2);
\draw (7+2.33,0-2)--(7+2.33,0.66-2);
\draw (7+3,0-2)--(7+3,0.66-2);
\draw (7+3.66,0-2)--(7+3.66,0.66-2);
\draw[fill=black] (7-0.66,0.33-2) circle (8pt);
\draw[fill=black] (7+0.66,0.33-2) circle (8pt);
\draw[fill=black] (7+2,0.33-2) circle (8pt);
\draw[fill=black] (7+2.66,0.33-2) circle (8pt);

\node at (7+2.66,0.5-2) (XYZ) {};
\node at (7+2,0.5-2) (xyw) {};
\node at (7+2,0.5-2) {};
\node at (7+2.66,0.5-2) {};
\draw (XYZ) edge[out=90,in=90,->, line width=0.5pt] (xyw);

\draw[arrow] (7+4,0.33-2)--(7+5,0.33-2);

\draw (14-1.66,0-2)--(14+3.66,0-2);
\draw (14+-1.66,0-2)--(14+-1.66,0.66-2);
\draw (14+-1,0-2)--(14+-1,0.66-2);
\draw (14+-0.33,0-2)--(14+-0.33,0.66-2);
\draw (14+0.33,0-2)--(14+0.33,0.66-2);
\draw (14+1,0-2)--(14+1,0.66-2);
\draw (14+1.66,0-2)--(14+1.66,0.66-2);
\draw (14+2.33,0-2)--(14+2.33,0.66-2);
\draw (14+3,0-2)--(14+3,0.66-2);
\draw (14+3.66,0-2)--(14+3.66,0.66-2);
\draw[fill=black] (14-0.66,0.33-2) circle (8pt);
\draw[fill=black] (14+0.66,0.33-2) circle (8pt);
\draw[fill=black] (14+2,0.33-2) circle (8pt);
\draw[fill=black] (14+2.66,0.33-2) circle (8pt);

\node at (14+2.66,0.5-2)  {};
\node at (14+2,0.5-2)  {};
\node at (14+2,0.5-2) {};
\node at (14+2.66,0.5-2) {};

\draw (-1.66,-4)--(3.66,-4);
\draw (-1.66,0-4)--(-1.66,0.66-4);
\draw (-1,0-4)--(-1,0.66-4);
\draw (-0.33,0-4)--(-0.33,0.66-4);
\draw (0.33,0-4)--(0.33,0.66-4);
\draw (1,0-4)--(1,0.66-4);
\draw (1.66,0-4)--(1.66,0.66-4);
\draw (2.33,0-4)--(2.33,0.66-4);
\draw (3,0-4)--(3,0.66-4);
\draw (3.66,0-4)--(3.66,0.66-4);
\draw[fill=black] (0.66,0.33-4) circle (8pt);
\draw[fill=black] (2.66,0.33-4) circle (8pt);

\draw[arrow] (4,0.33-4)--(5,0.33-4);

\draw (7-1.66,0-4)--(7+3.66,0-4);
\draw (7+-1.66,0-4)--(7+-1.66,0.66-4);
\draw (7+-1,0-4)--(7+-1,0.66-4);
\draw (7+-0.33,0-4)--(7+-0.33,0.66-4);
\draw (7+0.33,0-4)--(7+0.33,0.66-4);
\draw (7+1,0-4)--(7+1,0.66-4);
\draw (7+1.66,0-4)--(7+1.66,0.66-4);
\draw (7+2.33,0-4)--(7+2.33,0.66-4);
\draw (7+3,0-4)--(7+3,0.66-4);
\draw (7+3.66,0-4)--(7+3.66,0.66-4);
\draw[fill=black] (7+0.66,0.33-4) circle (8pt);
\draw[fill=black] (7+2.66,0.33-4) circle (8pt);

\node at (7+2.66,0.5-4) (xyw) {};
\node at (7+2,0.5-4) (XYZ) {};
\node at (7+2,0.5-4) {};
\node at (7+2.66,0.5-4) {};
\draw (xyw) edge[out=90,in=90,->, line width=0.5pt] (XYZ);

\draw[arrow] (7+4,0.33-4)--(7+5,0.33-4);

\draw (14-1.66,0-4)--(14+3.66,0-4);
\draw (14+-1.66,0-4)--(14+-1.66,0.66-4);
\draw (14+-1,0-4)--(14+-1,0.66-4);
\draw (14+-0.33,0-4)--(14+-0.33,0.66-4);
\draw (14+0.33,0-4)--(14+0.33,0.66-4);
\draw (14+1,0-4)--(14+1,0.66-4);
\draw (14+1.66,0-4)--(14+1.66,0.66-4);
\draw (14+2.33,0-4)--(14+2.33,0.66-4);
\draw (14+3,0-4)--(14+3,0.66-4);
\draw (14+3.66,0-4)--(14+3.66,0.66-4);
\draw[fill=black] (14+0.66,0.33-4) circle (8pt);
\draw[fill=black] (14+2,0.33-4) circle (8pt);

\node at (14+2.66,0.5-4)  {};
\node at (14+2,0.5-4)  {};
\node at (14+2,0.5-4) {};
\node at (14+2.66,0.5-4) {};
\end{tikzpicture}

\end{center}
\label{MS}
\end{figure}

\subsubsection{$q$--Boson}

The $n$--species $q$--Boson can similarly be viewed as $n$ coupled single--species $q$--Bosons. At each lattice site, an arbitrary number of particles of species $j$ (for $1\leq j \leq n$) may occupy that site. If a lattice site has $k^{(j)}$ particles of species $j$, then the jump rate for a particle of species $j$ is given by $q^{ k^{(j+1)} + \cdots k^{(n)} } ( 1 - q^{k^{(j)}} )$. As before, the jump rates of the heavier particles do not depend on the positions of the lighter particles. To see that the Markov projection to particles of species $\geq j$ is itself a $q$--Boson, note the telescoping sum
$$
\left(1 - q^{k^{(n)}}\right) + q^{k^{(n)}}\left(1 - q^ { k^{ (n-1) }} \right) + q^{k^{(n-1)} + k^{(n)}}\left(1 - q^{k^{(n-2)}}\right)  + \ldots + q^{k^{(j+1)} + \ldots +  k^{(n)}}  \left(1 - q^{k^{(j)}}\right)  = 1 - q^{k^{(j)} + \ldots +  k^{(n)}} 
$$

\subsubsection{ASEP$(q,m/2)$}

Now we turn to the definition of the $n$--species ASEP$(q,m/2)$. Up to $m$ particles may occupy each lattice site. There are $n$ species of particles, and suppose that the lattice site $x$ has $k^{(j)}_x$ particles of species $j$. There are two types of jumps to consider: one is when a particle of species $j$ at lattice site $x$ jumps to the right and switches places with a particle of species $i<j$, and the second is when a particle of species $j$ at lattice site $x+1$ jumps to the left and switches places with a particle of species $i<j$. In the first case, the jump rates are 
$$
q\cdot q ^ { k_{x}^{(1)} + \cdots + k_x^{(j-1)} } \cdot q ^ { k^{(i+1)}_{ x+1} + \cdots + k^{(n+1)}_{ x+1} } \left( \frac { 1- q ^ { k^{(j)} _ { x} } } { 1- q ^ { m } } \right) \left( \frac { 1- q^{  k^{(i)}_{x+1} } } { 1- q^m  } \right),
$$
and in the second case the jump rates are
$$
q ^ { k_{x}^{(1)} + \cdots + k_x^{(i-1)} } \cdot q ^ { k^{(j+1)}_{ x+1} + \cdots + k^{(n+1)}_{ x+1} }  \left( \frac { 1- q ^{ k^{(i)}_{ x} } } { 1- q^m  } \right) \left( \frac { 1- q ^ { k^{(j)}_{x+1} } }{ 1- q^m } \right).
$$

\subsection{$q$--exchangeable and reversible measures}

To describe the reversible measures of these processes, we first describe a class of measures called $q$--exchangeable measures, which were introduced in \cite{GO,GO2}. Intuitively, this is a measure on particle configurations such that whenever a heavier particle moves to the right of a lighter particle, the probability of the particle configuration is multiplied by $q$. This is related to the jump rates of ASEP -- namely, the jump rates for a heavy particle to the right are $q$ times the jump rates for a heavy particle to the left.

Before giving the rigorous definition of $q$--exchangeable measures, first we define some notation and terminology. The basic intuition behind the notation comes from two symmetries on particle configurations. First, if there are two particles of the \textit{same} species at \textit{different} lattice sites, then switching the \textit{particle locations} of the two particles preserves the particle configuration. Second, if there are two particles of \textit{different} species at the \textit{same} lattice site, then switching the \textit{species numbers} of the two particles also preserves the particle configuration. This will be interpreted as the right action and the left action of two subgroups of $S(N)$, where $N$ is the number of particles. 

Given any permutation $\sigma \in S(N)$, let $\inv(\sigma)$ denote its number of inversions, which is the number of pairs $(i,j)$ such that $i<j$ and $\sigma(i) > \sigma(j)$. For any sequence of integers $ \mathbf{N}=(N_1,\ldots,N_n)$ such that $N_1 + \ldots + N_n=n$, define the Young subgroup $S(\mathbf{N}) = S(N_1) \times \cdots \times S(N_n)$ by letting $S(N_1)$ act on $\{1,\ldots,N_1\}$, and $S(N_2)$ act on $\{N_1 + 1,\ldots,N_1+N_2\}$, and so forth. Each left coset of a Young subgroup $H\leq S(N)$ has a unique representative with the fewest number of inversions (see e.g. Proposition 2.3.3 of \cite{Carter}); let $D_H$ denote this set of distinguished coset representatives. Every $\sigma \in S(N)$ has a unique decomposition $\sigma = \sigma^0 \bar{\sigma}$, where $\sigma^0 \in D_H$ and $\bar{\sigma} \in H$, which satisfies $\inv(\sigma) = \inv(\sigma^0) + \inv(\bar{\sigma})$ (see also Proposition 2.3.3 of \cite{Carter}.)

An equivalent definition of $\inv(\sigma)$ and Young subgroups will be useful. The symmetric group $S(N)$ is generated by transpositions $s_1,\ldots,s_{N-1}$, where $s_i$ is the transposition $(i \ i+1)$. The number of inversions of $\sigma$ is the minimal number $l$ such that $\sigma$ is the product of $l$ transpositions. A Young subgroup is any subgroup generated by $\{s_j: j\in J\}$ for some subset $J\subseteq \{1,\ldots,N-1\}$.

We will also need information about double cosets. Recall that there is a representation $\tau$ of $S_n$ acts on a $(N-1)$--dimensional vector space with basis $u_1,\ldots,u_{N-1}$ by mapping
$$
\tau_{s_i}(v) = v - 2\langle u_i,v\rangle u_i,
$$
where the bilinear form is defined by
$$
\langle u_i,u_j\rangle 
= 
\begin{cases}
1, &\text{ if } i=j, \\
-1/2, &\text{ if } i = j \pm 1,\\
0, &\text{ else.}
\end{cases}
$$
Given two Young subgroups $H'$ and $H$, let $D_{H',H}=D_{H'}^{-1} \cap D_H$.  Every double coset $H' \sigma H$ contains a unique element of $D_{H',H}$, and every $\sigma \in D_{H',H}$ is the unique element with minimal inversions in its double coset $H'\sigma H$. Conversely, given any $\sigma\in D_{H',H}$, let $K$ be the Young subgroup $K$ generated by  
$$
H' \cap \{s_k : \tau_{\sigma}(u_j)=u_k \text{ for } s_j \in H   \}.
$$
Then every element of the double set $H'\sigma H$ can be written uniquely as $a\sigma b$ for $a\in H' \cap D_K,\sigma \in D_{H',H},b\in H$, and this decomposition satisfies
\begin{equation}\label{Sum}
\inv(a \sigma b) = \inv(a) + \inv(\sigma) + \inv(b)
\end{equation}

\begin{remark}\label{AHPEx}
Let $H= S(1) \times S(2) \times S(2) \times S(3)$ and $H' = S(1) \times S(2) \times S(2) \times S(2) \times S(1)$ be Young subgroups of $S(8)$. The element $\sigma = s_5s_4s_3s_1s_6s_5 = 21467358$ is in $D_{H',H}$ and
\begin{align*}
\tau_{\sigma} (u_2) &= u_1 + u_2 + u_3 + u_4 + u_5, \\
\tau_{\sigma} (u_4) &= u_3 + u_4 + u_5 + u_6 ,\\
\tau_{\sigma} (u_6) &= u_4, \\
\tau_{\sigma} (u_7) &= u_5 + u_6 + u_7, 
\end{align*}
so therefore $K$ is generated by $s_4$ and $H'\cap D_K = \{s_2,s_6\}$.
\end{remark}

The $q$--exchangeable measure on $D_{H',H}$ is defined by 
$$
\mathrm{Prob}(\sigma) = \frac{q^{\inv(\sigma)}}{Z_{H',H}},
$$
where $Z_{H',H}$ is a normalizing factor. 

More generally, a $q$--exchangeable measure can be defined on the disjoint union across $D_{H',H}$. For
$$
\mathbf{x} \in \mathcal{W}_N := \{(x_1,\ldots,x_N): x_N \leq \cdots \leq x_1\} \subset \mathbb{Z}^N,
$$
define $\mathbf{m}(\mathbf{x})=(m_1,\ldots,m_r)$ so that
$$
x_1 = \cdots = x_{m_1} > x_{m_1+1} = \cdots = x_{m_1+m_2} > x_{m_1+m_2+1} = \cdots = \cdots 
> x_{m_1 + \cdots + m_{r-1}+1}  = \cdots = x_N,
$$
where $m_r$ is defined by $m_1 + \cdots + m_r=N$. Now, for a fixed Young subgroup $H' \subseteq S(N)$, a probability measure on 
$$
\coprod_{\mathbf{x} \in \mathcal{W}_N} \{\mathbf{x}\} \times D_{H', S(\mathbf{m}(\mathbf{x}))  },
$$
is $q$--exchangeable if it is of the pseudo--factorized form
$$
\mathrm{Prob}(\mathbf{x},\sigma) = p(\mathbf{x}) \frac{q^{\inv(\sigma)}}{Z_{H', S(\mathbf{m}(\mathbf{x}))}}
$$
for some probability measure $p(\cdot)$ on $\mathcal{W}_N$. An equivalent definition is that
\begin{equation}\label{edef}
q^{-\inv(\sigma)} \cdot \mathrm{Prob}(\mathbf{x},\sigma) = q^{-\inv(\sigma')}  \cdot \mathrm{Prob}(\mathbf{x},\sigma') 
\end{equation}
for all $\mathbf{x} \in \mathcal{W}_N$ and $\sigma,\sigma' \in D_{H',S(\mathbf{m}(\mathbf{x}))}$. 

Each $(\mathbf{x},\sigma)$ uniquely defines a particle configuration (by the uniqueness properties mentioned above); see Figure \ref{State} for an example. More precisely, using the $\mathbf{m}(\mathbf{x})$ notation, then the number of species $j$ particles at lattice site $x:=x_{m_1+\ldots+m_s+1} = \cdots = x_{m_1+\ldots+m_{s+1}}$ is given by 
$$
k_x^{(j)} = \left|  \{ \sigma(m_1+\ldots+m_s+1),\ldots,\sigma(  m_1+\ldots+m_{s+1} ) \} \cap \{ N_{ j-1 } +1,\ldots,N_j \}  \right|.
$$
Let $\mathfrak{M}(\mathbf{x},\sigma) \in \hat{S}^{(n)}$ denote the corresponding particle configuration.

\begin{remark}\label{ParticleExample} Consider the particle configuration shown in the top of Figure \ref{State}. There is more than one $\sigma\in S(N)$ which defines this particle configuration, and it is not hard to see that $\sigma = 21467358$ has the fewest inversions. In fact, this $\sigma$ is the element $s_5s_4s_3s_1s_6s_5 \in D_{H',H}$ from Example \ref{AHPEx}, where $H=S(\mathbf{m}(\mathbf{x}))=S(1)\times S(2) \times S(2) \times S(3)$ and $H'=S(\mathbf{N})=S(1) \times S(2) \times S(2) \times S(2) \times S(1)$.

It is straightforward to see that the permutations $21476358,21567438, 35178426$ also describe the same particle configuration as $\sigma$, but have more inversions. These turn out to be in the double coset $H'\sigma H$ of $\sigma$, because it can be seen through direct calculation that the decompositions $a \sigma b$ take the form
\begin{align*}
21476358 &= e\cdot s_5s_4s_3s_1s_6s_5 \cdot s_6,\\
21567438 &= s_6 \cdot s_5s_4s_3s_1s_6s_5 \cdot s_4,\\
35178426 &= s_6s_2 \cdot s_5s_4s_3s_1s_6s_5 \cdot s_7s_6s_4s_2.
\end{align*}
Note that $21476358$ is also equal to $s_4 \cdot s_5s_4s_3s_1s_6s_5 \cdot e$, demonstrating that $s_4 \notin D_K$. 
\end{remark}

\begin{remark}
Suppose that $H$ is replaced with $\hat{H}=S(1) \times S(2) \times S(5)$. Then $\hat{H}$ contains $s_5$ and $s_6$, and the particle configuration is represented by $\sigma = 21456378 = s_5s_4s_3s_1$. See the bottom of Figure \ref{State}.
\end{remark}

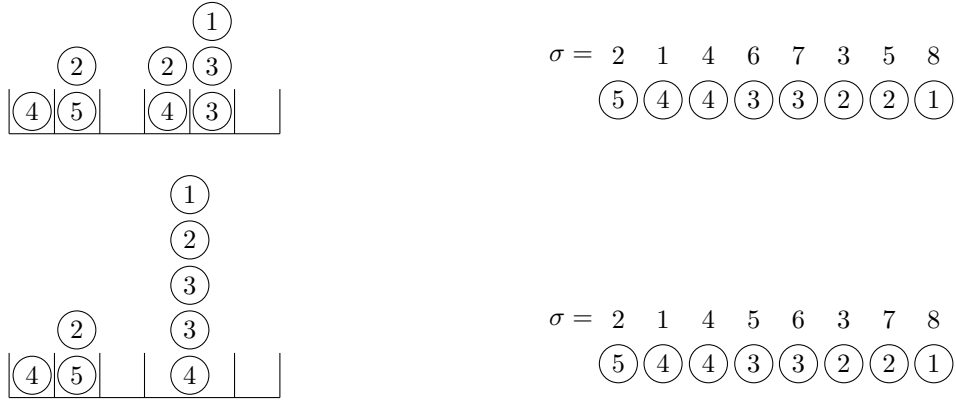
\begin{figure}
\begin{center}
\begin{tikzpicture}[scale=0.9, every text node part/.style={align=center}]
\usetikzlibrary{arrows}
\usetikzlibrary{shapes}
\usetikzlibrary{shapes.multipart}

\draw (--1,0)--(-3,0);
\draw (--1,0)--(--1,0.66);
\draw (--0.33,0)--(--0.33,0.66);
\draw (-0.33,0)--(-0.33,0.66);
\draw (-1,0)--(-1,0.66);
\draw (-1.66,0)--(-1.66,0.66);
\draw (-2.33,0)--(-2.33,0.66);
\draw (-3,0)--(-3,0.66);
\draw (-0.66,0.33) circle (8pt);
\draw (-0,0.33) circle (8pt);
\draw (-0,1) circle (8pt);
\draw (-0,1.66) circle (8pt);
\draw (-0.66,1) circle (8pt);
\draw (-2,1) circle (8pt);
\draw (-2,0.33) circle (8pt);
\draw (-2.66,0.33) circle (8pt);

\node at (-0,1.66) {$1$};
\node at (-0,1) {$3$};
\node at (-0,0.33) {$3$};
\node at (-0.66,0.33) {$4$};
\node at (-0.66,1)  {$2$};
\node at (-2,1) {$2$};
\node at (-2,0.33) {$5$};
\node at (-2.66,0.33) {$4$};

\draw  (6,0.5) circle (8pt);
\draw  (6.66,0.5) circle (8pt);
\draw  (7.33,0.5) circle (8pt);
\draw  (8,0.5)  circle (8pt);
\draw  (8.66,0.5) circle (8pt);
\draw  (9.33,0.5) circle (8pt);
\draw  (10,0.5) circle (8pt);
\draw  (10.66,0.5) circle (8pt);

\node at (5.3,1.16) {$\sigma=$};
\node at (6,1.16) {$2$};
\node at (6.66,1.16) {$1$};
\node at (7.33,1.16) {$4$};
\node at (8,1.16)  {$6$};
\node at (8.66,1.16) {$7$};
\node at (9.33,1.16) {$3$};
\node at (10,1.16) {$5$};
\node at (10.66,1.16) {$8$};

\node at (6,0.5) {$5$};
\node at (6.66,0.5) {$4$};
\node at (7.33,0.5) {$4$};
\node at (8,0.5)  {$3$};
\node at (8.66,0.5) {$3$};
\node at (9.33,0.5) {$2$};
\node at (10,0.5) {$2$};
\node at (10.66,0.5) {$1$};

\end{tikzpicture}

\vspace{0.2in}

\begin{tikzpicture}[scale=0.9, every text node part/.style={align=center}]
\usetikzlibrary{arrows}
\usetikzlibrary{shapes}
\usetikzlibrary{shapes.multipart}

\draw (--1,0)--(-3,0);
\draw (--1,0)--(--1,0.66);
\draw (--0.33,0)--(--0.33,0.66);
\draw (-1,0)--(-1,0.66);
\draw (-1.66,0)--(-1.66,0.66);
\draw (-2.33,0)--(-2.33,0.66);
\draw (-3,0)--(-3,0.66);
\draw (-0.33,0.33) circle (8pt);
\draw (-0.33,1) circle (8pt);
\draw (-0.33,3) circle (8pt);
\draw (-0.33,1.66) circle (8pt);
\draw (-0.33,2.33) circle (8pt);
\draw (-2,1) circle (8pt);
\draw (-2,0.33) circle (8pt);
\draw (-2.66,0.33) circle (8pt);

\node at (-0.33,1.66) {$3$};
\node at (-0.33,3) {$1$};
\node at (-0.33,2.33) {$2$};
\node at (-0.33,0.33) {$4$};
\node at (-0.33,1)  {$3$};
\node at (-2,1) {$2$};
\node at (-2,0.33) {$5$};
\node at (-2.66,0.33) {$4$};

\draw  (6,0.5) circle (8pt);
\draw  (6.66,0.5) circle (8pt);
\draw  (7.33,0.5) circle (8pt);
\draw  (8,0.5)  circle (8pt);
\draw  (8.66,0.5) circle (8pt);
\draw  (9.33,0.5) circle (8pt);
\draw  (10,0.5) circle (8pt);
\draw  (10.66,0.5) circle (8pt);

\node at (5.3,1.16) {$\sigma=$};
\node at (6,1.16) {$2$};
\node at (6.66,1.16) {$1$};
\node at (7.33,1.16) {$4$};
\node at (8,1.16)  {$5$};
\node at (8.66,1.16) {$6$};
\node at (9.33,1.16) {$3$};
\node at (10,1.16) {$7$};
\node at (10.66,1.16) {$8$};

\node at (6,0.5) {$5$};
\node at (6.66,0.5) {$4$};
\node at (7.33,0.5) {$4$};
\node at (8,0.5)  {$3$};
\node at (8.66,0.5) {$3$};
\node at (9.33,0.5) {$2$};
\node at (10,0.5) {$2$};
\node at (10.66,0.5) {$1$};

\end{tikzpicture}
\end{center}
\caption{The particle configuration referenced in Example \ref{ParticleExample}.
}
\label{State}
\end{figure}

The $q$--Binomial theorem can also be stated in terms of these cosets. Given a subgroup $G$ of $S(N)$, let
$$
\vert G\vert_q = \sum_{\sigma \in G} q^{\inv(\sigma)}.
$$
Then for any Young subgroup $H = S(N_1) \times \cdots \times S(N_r) \leq S(N)$, where $N_1 + \ldots + N_r=N$, we have (see e.g. Proposition 2.6 of \cite{KuanAHP})
\begin{equation}\label{qBin2}
\sum_{\sigma \in D_H} q^{\inv(\sigma)} = \binom{N}{N_1,\ldots,N_r}_q = \frac{ | S(N)|_q}{ \vert H\vert_q}
\end{equation}

In \cite{KuanAHP}, it is proved that the multi--species ASEP and $q$--Boson preserve $q$--exchangeability. In other words, if the initial conditions are $q$--exchangeable, then the probability measures at all future times are $q$--exchangeable. 

In \cite{KIMRN} and \cite{BS3} , it is shown that the multi--species ASEP on a finite lattice with closed boundary conditions has $q$--exchangeable reversible measures. The reversible measures of multi--species ASEP$(q,m/2)$ were also found in \cite{KIMRN}; this result will be generalized in Theorem \ref{SFA}(i) below.

We mention one more identity. Suppose that $\vec{m}=(m_x)_{x \in \mathbb{Z}}$ is of the form 
$$
m_x 
= \begin{cases}
1, \quad & 0 \leq m_x \leq L-1, \\
0, \quad & \text{else}.
\end{cases}
$$
Let $H=S(N_1) \times \cdots \times S(N_n) \leq S(N)$ where $N_1 + \ldots + N_n=N$. Then by \eqref{qBin} and \eqref{qBin2},
\begin{align}\label{qBin3}
q^{-(N-1)N/2} \sum_{0 \leq x_N < \ldots < x_1 \leq L-1 } q^{x_1 + \ldots + x_N} \sum_{\sigma \in D_H} q^{\inv(\sigma)} &= \binom{L}{N}_q \binom{N}{N_1,\ldots,N_n}_q  \notag \\
&= \binom{L}{L-N,N_1,\ldots,N_n}_q
\end{align}

\subsection{Definition and background on duality}

Recall the definition of stochastic duality:
\begin{definition} Two Markov processes (either discrete or continuous time) $X(t)$ and $Y(t)$ on state spaces $\mathfrak{X}$ and $\mathfrak{Y}$ are dual with respect to a function $D$ on $ \mathfrak{X}\times \mathfrak{Y}$ if
$$
\mathbb{E}_x\left[  D(X(t),y) \right] = \mathbb{E}_y \left[ D(x,Y(t))\right] \text{ for all } (x,y) \in \mathfrak{X}\times \mathfrak{Y} \text{ and all } t \geq 0.
$$
On the left--hand--side, the process $X(t)$ starts at $X(0)=x$, and on the right--hand--side the process $Y(t)$ starts at $Y(0)=y$.

An equivalent definition (for continuous--time processes and discrete state spaces) is that if the generator $\mathcal{L}_X$ of $X(t)$ is viewed as a $\mathfrak{X} \times \mathfrak{X}$ matrix, the generator $\mathcal{L}_Y$ of $Y(t)$ is viewed as a $\mathfrak{Y}\times \mathfrak{Y}$ matrix, and $D$ is viewed as a $\mathfrak{X} \times \mathfrak{Y}$ matrix, then $\mathcal{L}_XD = D\mathcal{L}_Y^*$, where the $^*$ denotes transpose.  For discrete--time chains with transition matrices $P_X$ and $P_Y$ also viewed as $\mathfrak{X}\times \mathfrak{X}$ and $\mathfrak{Y}\times \mathfrak{Y}$ matrices, an equivalent definition is
$$
P_X D = D P_Y^*.
$$
If $X(t)$ and $Y(t)$ are the same process, in the sense that $\mathfrak{X}=\mathfrak{Y}$ and $\mathcal{L}_{{X}} = \mathcal{L}_{{Y}}$ (for continuous time) or $P_X=P_Y$ (for discrete--time), then we say that $X(t)$ is self--dual with respect to the function $D$. 

\end{definition}

\begin{reemark}\label{StatDual}
The duality function can be related to stationary measures. In particular, take the $t\rightarrow \infty$ limit in the equality above. If there are unique stationary measures $X(\infty)$ and $Y(\infty)$, then we obtain
$$
\mathbb{E}[D(X(\infty),y)] = \mathbb{E}[D(x,Y(\infty))].
$$
Since the left--hand--side does not depend on $x$, then neither does the right--hand--side. Therefore, $\mathbb{E}[D(x,Y(\infty))]$ does not depend on $x\in \mathfrak{X}$, and similarly $\mathbb{E}[D(X(\infty),y)]$ does not depend on $y\in \mathfrak{Y}$.
\end{reemark}

\begin{reemark}
Note that if $c(x,y)$ is a function on $\mathfrak{X} \times \mathfrak{Y}$ which is constant under the dynamics of $X(t)$ and $Y(t)$, then $c(x,y)D(x,y)$ is also a duality function. This can be used to simplify the expression for $D(x,y)$. In previous works (such as \cite{KIMRN}) $c(x,y)$ was a function which only depended on the number of particles of each species, which is a constant assuming particle number conservation. In the present case, stochastic fusion does \textit{not} require particle number conservation, which allows for an easier analysis of open boundary conditions.
\end{reemark}

\begin{reemark}
Note that all Markov processes are dual to each other with respect to constant functions. More generally, if $\pi(y)$ is a stationary measure on $\mathfrak{Y}$, meaning that $\mathbb{E}_y[\pi(Y(t))]=\pi(y)$ for all $y\in \mathfrak{Y}$ and $t \geq 0$, then $D(x,y) = \pi(y)$ is a duality function between any Markov process $X(t)$ and $Y(t)$, because 
$$
\mathbb{E}_x[D(X(t),y)] = \pi(y) = \mathbb{E}_y[\pi(Y(t))] = \mathbb{E}_y[D(x,Y(t))].
$$
\end{reemark}

\begin{reemark}\label{ZCentral}
If $Z_X$ is a $\mathfrak{X}\times \mathfrak{X}$ matrix which commutes with $\mathcal{L}_X$, and $D$ is a duality function between $X(t)$ and $Y(t)$, then so is $Z_XD$, because
$$
\mathcal{L}_X Z_XD = Z_X \mathcal{L}_X D = Z_XD \mathcal{L}_Y^*.
$$
Similarly, if $Z_Y$ is a $\mathfrak{Y} \times \mathfrak{Y}$ matrix such that $Z_Y^*$ commutes with $\mathcal{L}_Y$, then
$$
\mathcal{L}_X DZ_Y = D \mathcal{L}_Y^* Z_Y = DZ_Y \mathcal{L}_Y^*,
$$
so $DZ_Y$ is also a duality function between $X(t)$ and $Y(t)$. More generally, $Z_X D Z_Y$ is a duality function as well. 
\end{reemark}

\begin{reemark}\label{QInter}
The previous remark can be generalized. Suppose that there is an additional process $W(t)$ on a state space $\mathfrak{W}$ with generator $L_W$. If $\mathcal{Q}$ is a $\mathfrak{W} \times \mathfrak{Y}$ matrix which intertwines with $\mathcal{L}_Y$ and $\mathcal{L}_W$, meaning that $\mathcal{Q} \mathcal{L}_Y = \mathcal{L}_W \mathcal{Q}$, then
\begin{align*}
\mathcal{L}_X D = D\mathcal{L}_Y^* &\Longrightarrow \mathcal{L}_X D \mathcal{Q}^* = D\mathcal{L}_Y^* \mathcal{Q}^*  = D (\mathcal{Q} \mathcal{L}_Y)^* = D ( \mathcal{L}_W \mathcal{Q})^* = D \mathcal{Q}^*\mathcal{L}_W^* \\
& \Longrightarrow  \mathcal{L}_X D \mathcal{Q}^* = D \mathcal{Q}^*\mathcal{L}_W^* 
\end{align*}

Similarly, if $\mathcal{Q}^*\mathcal{L}_X = \mathcal{L}_W \mathcal{Q}^*$, then
\begin{align*}
\mathcal{L}_X D = D\mathcal{L}_Y^* &\Longrightarrow  \mathcal{Q}^*\mathcal{L}_X D = \mathcal{Q}^*D\mathcal{L}_Y^*   \\
&\Longrightarrow \mathcal{L}_W \mathcal{Q}^* D = \mathcal{Q}^*D\mathcal{L}_Y^*  .
\end{align*}
\end{reemark}


\subsubsection{ASEP dualities}
There are a few previously known dualities for the ASEP, that we will mention here. (See also \cite{IS}). To describe the duality functional, we introduce some notation. For a particle configuration $\mathfrak{s} \in S^{(1)}$ and $x\in \mathbb{Z}$, let $N_x(\mathfrak{s})$ denote
$$
N_x(\mathfrak{s}):= \left| \{ z \geq x: \mathfrak{s}(z) = 1   \} \right|.
$$ 
If $\mathfrak{s}_t$ denotes an ASEP$_{1,q}$ written in particle variables, and $\vec{x}(t)$ denotes another ASEP$_{1,q}$ written in occupation variables with $N$ particles, then equation (3.11) of \cite{Sch97} shows that $\mathfrak{s}_t$ and $\vec{x}(t)$ are dual with respect to the functional
$$
D_{\mathrm{Sch}}(\mathfrak{s},\vec{x}) = \prod_{k=1}^N 1_{\{\mathfrak{s}(x_k)=1\}} q^{-x_{k} -N_{x_k}(\mathfrak{s})}.
$$
This duality holds on either a finite lattice (with closed boundary conditions) or an an infinite lattice. The multi--species version of $D_{\mathrm{Sch}}$ was found in Theorem 3.5 of \cite{BS3} and Theorem 2.5(a) \cite{KIMRN}. To state it, introduce the notation
$$
N_x^{(j)}( \mathfrak{s}) =  \left| \{ z \geq x: \mathfrak{s}(z) \geq j   \} \right| .
$$
Then
$$
D_{\mathrm{Sch}}^{(n,p)}( \mathfrak{s}, \mathfrak{s}') = \prod_x 1_{\{\mathfrak{s}(x) \geq \mathfrak{s}'(x) \geq 1\}} q^{-x - N_x^{\mathfrak{s}'(x)}(\mathfrak{s})}.
$$

In Theorem 4.2 of \cite{BCSDuality}, it is shown that if $\mathfrak{s}_t$ evolves as an ASEP$_{1,q}$, and $\vec{x}(t)$ evolves as a ASEP$_{q,1}$ with $N$ particles, then \cite{BCSDuality} shows that $\mathfrak{s}_t$ and $\vec{x}(t)$ are dual with respect to the functional
$$
D_{\mathrm{BCS}}(\mathfrak{s},\vec{x}) = \prod_{k=1}^N q^{ -N_{x_k}(\mathfrak{s})}.
$$
However, this duality only holds on an \textit{infinite lattice}. 

Another duality was found in Theorem 2.5(b) of \cite{KIMRN}: if $\mathfrak{s}_t$ evolves as an ASEP$_{1,q}$, and $\vec{x}(t)$ evolves as a ASEP$_{q,1}$ with $N$ particles, then \cite{BCSDuality} shows that $\mathfrak{s}_t$ and $\vec{x}(t)$ are dual with respect to the functional
$$
D_{\mathrm{Kua}}(\mathfrak{s},\vec{x}) = \prod_{k=1}^N 1_{\{\mathsf{s}(x_k)=0\}}q^{ -N_{x_k}(\mathfrak{s})}.
$$

\subsubsection{ASEP$(q,m/2)$ and $q$--Boson duality}
From equation (34) of Theorem 3.2 in \cite{CGRS}, we a duality between ASEP$(q,m/2)$ and itself. If $\mathfrak{s}_{t}$ and $\mathfrak{s}'_{t}$ each evolve as an ASEP$(q,m/2)$, then 
$$
\prod_{x \in \mathbb{Z}} \frac{ q^{\mathfrak{s'}(x)(\mathfrak{s'}(x)-\mathfrak{s}(x))/2} \binom{\mathfrak{s}(x)}{\mathfrak{s}'(x)}_q }{ q^{\mathfrak{s}'(x)(\mathfrak{s'}(x)-m)/2} \binom{m}{\mathfrak{s}'(x)}_q } q^{ \mathfrak{s}(x) \mathfrak{s}'(x)/2 + \mathfrak{s}'(x) \sum_{z <x} \mathfrak{s}(z)  - m \mathfrak{s}'(x) x} 1_{\{  \mathfrak{s}(x) \geq \mathfrak{s}'(x)  \}}
$$
is a duality function \footnote{Note that this paper uses a slightly different definition of $q$--binomials than \cite{CGRS}.}. Because the quantity $q^{m\sum_{x \in \mathbb{Z}} \mathfrak{s}'(x)/2}$ is constant with respect to the dynamics, an equivalent duality function is
\begin{equation}\label{CGRSd}
D_{\text{CGRS}}( \mathfrak{s}, \mathfrak{s}' ) := \prod_{x \in \mathbb{Z}} \frac{  \binom{\mathfrak{s}(x)}{\mathfrak{s}'(x)}_q }{  \binom{m}{\mathfrak{s}'(x)}_q } q^{ \mathfrak{s}'(x) \sum_{z <x} \mathfrak{s}(z)  - m \mathfrak{s}'(x) x} 1_{\{  \mathfrak{s}(x) \geq \mathfrak{s}'(x)  \} }.
\end{equation}
Another duality is when $\mathfrak{s}'_t$ evolves as an ASEP$(q,m/2)$ and $\mathfrak{s}_t$ evolves as a space--reversed ASEP$(q,m/2)$. Then the duality function is
\begin{equation}\label{Kuad}
D_{\text{Kua}}( \mathfrak{s}, \mathfrak{s}' ) := \prod_{x \in \mathbb{Z}} \frac{  \binom{m-\mathfrak{s}(x)}{\mathfrak{s}'(x)}_q }{  \binom{m}{\mathfrak{s}'(x)}_q } q^{ -\mathfrak{s}'(x) \sum_{z \geq x} \mathfrak{s}(z)} 1_{\{ m-  \mathfrak{s}(x) \geq \mathfrak{s}'(x)  \} }.
\end{equation}

Taking $m\rightarrow\infty$ yields a duality for the $q$--Boson, which was proven in \cite{KIMRN}. If $\vec{x}_t$ evolves as a $q$--Boson and $\mathfrak{s}_t$ evolves as a space--reversed $q$--Boson, then $\mathfrak{s}_t$ and $\mathfrak{s}_t'$ are dual with respect to the function
$$
\bar{D}_{\mathrm{BCS}}(\mathfrak{s},\vec{x}) := \prod_{k=1}^N q^{N_{x_k}(\mathfrak{s})}.
$$

\subsubsection{SEP($m/2$) duality}
In Theorem 4.2(b) of \cite{GKRV}, it is shown, using $SU(2)$ symmetry, that
$$
D_{\mathrm{GKRV}}(\mathfrak{s},\mathfrak{s}') = \prod_{x \in \mathcal{G}} \frac{  \binom{\mathfrak{s}(x)}{\mathfrak{s}'(x)} }{  \binom{m}{\mathfrak{s}'(x)} } 1_{\{\mathfrak{s}(x) \geq \mathfrak{s}'(x) \}}  .
$$
is a self-duality function for the SEP$(m/2)$. If $m=1$, this reduces to the \cite{Spit70} duality for the SEP:
$$
D_{\mathrm{Spi}}(\mathfrak{s},\mathfrak{s}') = \prod_{x \in \mathcal{G}}  1_{\{\mathfrak{s}(x) \geq \mathfrak{s}'(x) \}}
$$

\subsubsection{SSEP$(m/2)$ duality}
There are multiple ways to take the $q\rightarrow 1$ limit to obtain nontrivial duality functions in the symmetric case. Setting $q=1$ in $D_{\mathrm{CGRS}}$ yields the function $D_{\mathrm{GKRV}}$ in the case that $\mathcal{G}$ is an interval in $\mathbb{Z}$.

It is straightforward to see that
$$
\lim_{q \rightarrow 1} \frac{ D_{\text{BCS}}(\mathfrak{s},\vec{x}) -1}{q-1} = \lim_{q \rightarrow 1} \frac{ q^{-\sum_{k=1}^N N_{x_k}(\mathfrak{s})}- 1}{q-1} = -\sum_{k=1}^N N_{x_k}(\mathfrak{s}) := D_{\text{BCS}}'(\mathfrak{s},\vec{x}).
$$

\subsubsection{Dynamic ASEP duality}
In Theorem 2.3 of \cite{BorCorIMRN}, it is shown that if $\vec{x}(t)$ evolves as an ASEP$_{1,q}$, and $\mathfrak{s}(x)$ evolves as a dynamic ASEP with parameter $\alpha$, then the function
$$
 D_{\text{BC}}(\mathfrak{s},\vec{x}) = \prod _ { k = 1 } ^ { N } (  q^{\frac{-s({x_k})-x_k}{2}}  + \alpha^{-1} q^{k-1} )(   q^{\frac{s({x_k})-x_k}{2}} - q^{k-1})
$$
is a duality function. Using the $\mathfrak{s}$ notation, this can be written as 

$$
 \prod _ { k = 1 } ^ { N } ( q^{ -N_{x_k}(\mathfrak{s}) - {x_k}} + \alpha^{-1} q^{k-1})  ( q ^ { k - 1 } - q ^ { N _ { x _ { k } }(\mathfrak{s}) } ),\\
$$
if $N_x(\mathfrak{s})$ is always finite. We note that the function $N_x$ counts the number of particles to the right of $x$, while $-N_x-x$ counts the number of holes to the right of $x$.

If  $\alpha\rightarrow 0$ and we divide by constants, we obtain a duality function between a ASEP$_{q,1}$ and another ASEP$_{1,q}$:
\begin{align}\label{lim1}
\prod_{k=1}^N (q^{k-1} - q^{N_{x_k}(\mathfrak{s})}) &= q^{(N-1)N/2} \prod_{k=1}^N (1 - q^{N_{x_k}(\mathfrak{s}) - k +1  }) \notag \\
&= q^{(N-1)N/2} \sum_{I \subseteq \{1,\ldots,N\}} (-1)^{\vert I\vert} \prod_{i \in I} q^{N_{x_i}(\mathfrak{s})-i+1}.
\end{align}
If $\alpha \rightarrow \infty$, we obtain a duality function between ASEP$_{1,q}$ and ASEP$_{1,q}$:
\begin{align}\label{lim2}
\prod _ { k = 1 } ^ { N } ( q^{ -N_{x_k}(\mathfrak{s}) - {x_k}  } )  ( q ^ { k - 1 } - q ^ { N _ { x _ { k } }(\mathfrak{s}) } ) &= \prod_{k=1}^N (q^{ -N_{x_k}(\mathfrak{s}) - {x_k} +k-1 } - q^{-x_k})\notag\\
&= \sum_{I \subseteq \{1,\ldots,N\}} (-1)^{N - \vert I\vert} q^{-x_1 - \cdots - x_N} \prod_{i \in I} q^{-N_{x_i}(\mathfrak{s})+i-1}.
\end{align}
We will see in Corollary \ref{newdual} how to obtain these duality results directly using the interchange process. 

\subsubsection{Dynamic SSEP duality}
In the $q\rightarrow 1$ limit of $D_{BC}$, as done in Corollary 10.8 of \cite{BorodinDyn}, we obtain 
\begin{multline*}
\prod_{k=1}^N \left(  (k-1)^2 - (k-1)(\lambda - x_k) - \frac{s_{x_k}(t) - x_k}{2} \left( \frac{s_{x_k}(t) + x_k}{2} - \lambda \right) \right) \\
= \prod_{k=1}^N \left((k-1)-\lambda + \frac{s_{x_k}(t)+x_k}{2} \right)\left((k-1) + \frac{ x_k - s_{x_k}(t) }{2} \right).
\end{multline*}
This is a self--duality function for the dynamic SSEP. When $\lambda\rightarrow\infty$, the limit is (up to constants)
$$
\prod_{k=1}^N \left((k-1) - \frac{s_{x_k}(t) - x_k}{2}  \right) = \prod_{k=1}^N \left((k-1) - N_{x_k}(\mathfrak{s}) \right).
$$

\section{A modified Pitman--Rogers intertwining and duality}

We use the Pitman--Rogers theorem as a starting point for this section.

\begin{theorem}\label{RPThm}
(a) Assume that
\begin{itemize}
\item
$\Lambda\Phi$ is the identity on $\hat{S}$.
\item There exists a measure $\pi_{\hat{S}}$ on $\hat{S}$ such that 
$
\pi _ { \hat{S}   } \Lambda P _ { t } = \pi _ { \hat{S} } \Lambda P _ { t } \Phi \Lambda.
$
\end{itemize}

(i)  Let the initial condition of the Markov process $X_t$ be $X_0=\pi_S := \pi_{\hat{S}}\Lambda$.  Then $\phi(X_t)$ is Markov with initial condition $\pi_{\hat{S}}$ and transition probabilities
$$
Q _ { t } := \Lambda P _ { t } \Phi.
$$

(ii) Suppose that $(\tilde{P}_t: t\geq 0)$ is a semigroup of stochastic matrices on $S$, and $D$ is an operator on $S$ satisfying the intertwining
$$
P_tD = D\tilde{P}_t^* \text{ for all } t \geq 0.
$$
Then
$$
\pi_{\hat{S}} Q_t \hat{D} = \pi_{\hat{S}}\hat{D}  \tilde{P}_t^* \text{ for all } t \geq 0,
$$
where $\hat{D}:=\Lambda D$ is a $\hat{S} \times S$ matrix.

(b) Conversely, suppose that
$$
\pi_{\hat{S}} Q_t \hat{D} = \pi_{\hat{S}}\hat{D}  \tilde{P}_t^* 
$$ 
and 
$$
\pi _ { \hat{S}   } \Lambda P _ { t } = \pi _ { \hat{S} } \Lambda P _ { t } \Phi \Lambda,
$$
where as before $Q _ { t } = \Lambda P _ { t } \Phi$ and $\hat{D}=\Lambda D$. Then
$$
\pi _ { \hat{S} }  { \Lambda }  P _ { t } D = \pi _ { \hat{S} } \Lambda D \tilde{P} _ { t } ^ { * }
$$
\end{theorem}
\begin{proof}
(a) The proof of (i) is essentially identical to the proof in \cite{lrjp1}.

To see (ii), note that by the second assumption,
$$
\pi_{\hat{S}} Q_t \hat{D} = \pi_{\hat{S}} \Lambda P_t \Phi \Lambda {D} = \pi_{\hat{S}} \Lambda P_t  {D}.
$$
By the intertwining of $D$ between $P_t$ and $\tilde{P}_t$, we have that this equals
$$
\pi_{\hat{S}} \Lambda D\tilde{P}_t^* = \pi_{\hat{S}} \hat{ D }\tilde{P}_t^*,
$$
as needed.

Now turn to (b). By the intertwining relation,
$$
\pi _ { \hat{S} } \Lambda P _ { t } D = \pi _ { S } \Lambda D \tilde{P} _ { t } ^ {  * }.
$$
Therefore, by the second assumption,
$$
\pi _ { \hat{S} } \Lambda P _ { t } \Phi\Lambda D = \pi _ { S } \Lambda D \tilde{P} _ { t } ^ {  * }.
$$
Inserting the definitions of $Q_t$ and $\hat{D}$ finishes the proof.

\end{proof}

\begin{corollary}\label{MainCor}
Suppose
\begin{itemize}
\item
$\Lambda\Phi$ is the identity on $\hat{S}$.
\item There exists a probability measure $\pi_{\hat{S}}$ on $\hat{S}$ such that $\pi_{\hat{S}}\Lambda$ is a stationary measure for the Markov process $X_t$.
\end{itemize}
Then (i) and (ii) from the Theorem hold. Furthermore, 
$$
\pi_{\hat{S}}\hat{D}  \tilde{P}_{t_1}^* = \pi_{\hat{S}}\hat{D}  \tilde{P}_{t_2}^* \text{ for all } t_1,t_2 \geq 0.
$$
\end{corollary}
\begin{proof}
It suffices to show that
$
\pi _ { S ^ { \prime } } \Lambda P _ { t } = \pi _ { \hat{S} } \Lambda P _ { t } \Phi \Lambda.
$
Since $\pi_{\hat{S}}\Lambda$ is stationary, by definition $\pi_{\hat{S}}\Lambda P_t = \pi_{\hat{S}}\Lambda$. Therefore, using the fact that $\Lambda\Phi$ is the identity,
$$
 \pi _ { \hat{S} } \Lambda P _ { t } \Phi \Lambda =  \pi _ { \hat{S} } \Lambda \Phi \Lambda =  \pi _ { \hat{S} } \Lambda  =  \pi _ { \hat{S} } \Lambda P_t,
$$
showing that (i) and (ii) hold. 

By (ii),
$$
\pi_{\hat{S}} Q_t \hat{D} = \pi_{\hat{S}}\hat{D}  \tilde{P}_t^* \text{ for all } t \geq 0.
$$
But because $Q_t:=\Lambda P_t \Phi $ and $\pi_{\hat{S}}\Lambda$ is stationary and $\Lambda\Phi$ is the identity, then the left--hand--side equals
$$
\pi_{\hat{S}}Q_t \hat{D} =  \pi_{\hat{S}}\Lambda P_t \Phi \hat{D} =  \pi_{\hat{S}}\Lambda \Phi \hat{D} =  \pi_{\hat{S}}\hat{D}.
$$
Because this does not depend on $t$, then neither does the right--hand--side, showing the result. 
\end{proof}

For the remainder of this article, we will set $Q_t = \Lambda P_t \Phi$ and $\hat{D} = \Lambda D$.

In Theorem \ref{RPThm}, there are intertwining relations involving $\hat{D} = \Lambda D$. With some more assumptions, we can also show intertwining relations involving $\tilde{D} := \Lambda D \Phi$.
\begin{prop}\label{RPInter}
Suppose that the conditions of Theorem \ref{RPThm} hold. Suppose that $X(t)$ is reversible: i.e., there exists a diagonal $S\times S$ matrix $V$ such that 
$$
V^{-1} P_t V = P_t^*.
$$
Additionally, suppose that there exists a constant $\mathfrak{c}$ such that $\Lambda = \mathfrak{c}\Phi^*$. If $P_tD = D {P}_t^*$, and $\pi_S D\Phi \Lambda = \pi_S D$, then setting $\tilde{D} = \Lambda D \Phi$, we have 
$$
\pi_{\hat{S}} Q_t \tilde{D} = \pi_{\hat{S}}\tilde{D} Q_t^*.
$$


\end{prop}
\begin{proof}
Starting with $P_tD=DP_t^*$ and using $V^{-1}P_t = P_t^*V^{-1},$ we obtain
$$
\pi_{{S}} P_t D V^{-1} = \pi_{{S}}DV^{-1}P_t.
$$
By using $\pi_S D= \pi_S D\Phi \Lambda$ on the right--hand--side, and $\pi_S P_t = \pi_S P_t \Phi\Lambda$ (from Theorem \ref{RPThm}) on the left--hand--side, we obtain
$$
\pi_{{S}} P_t \Phi \Lambda  D V^{-1} = \pi_{{S}} D \Phi \Lambda V^{-1} P_t .
$$
Multiplying on the right by $V\Phi$, and using $V^{-1}P_t = P_t^* V^{-1}$, we have
$$
\pi_{\hat{S}} \Lambda P_t \Phi \Lambda  D \Phi = \pi_{\hat{S}}  \Lambda D \Phi \Lambda P_t^* \Phi.
$$
Replacing $\Lambda$ with $\mathfrak{c}\Phi^*$ and $\Phi^*$ with $\mathfrak{c}^{-1}\Lambda$ gives exactly the  condition that $\pi_{\hat{S}} Q_t \tilde{D} = \pi_{\hat{S}} \tilde{D} Q_t^*$.

\end{proof}

\begin{reemark}\label{Constant}
Because $\Phi$ and $\Lambda$ are both Markov kernels, if $\Lambda = \mathfrak{c}\Phi^*$, we must have that
$$
\mathfrak{c} = \left( \sum_{x} \Phi(x,y) \right)^{-1} = \sum_y \Lambda(y,x).
$$ 
Since $\Phi$ is deterministic, this implies that
$$
\mathfrak{c} = \vert \phi^{-1}(y) \vert^{-1}
$$
does not depend on $y$. If $S$ and $\hat{S}$ are finite, then this implies that $\mathfrak{c} = \vert \hat{S} \vert / \vert S \vert$. If $\Lambda$ is defined by
$$
\Lambda(y,x) = 1_{\{\phi(x) = y\}} \cdot \vert \phi^{-1}(y) \vert^{-1},
$$
then $\mathfrak{c}$ is well defined and $\Lambda\Phi$ is the identity.
\end{reemark}

\begin{reemark}\label{Constant2}
The condition $\vert \phi^{-1}(y) \vert \cdot \vert \hat{S} \vert= \vert S \vert $ for all $y\in \hat{S}$ holds for the interchange process, but not for single--species models. See Figure \ref{ic} for an example.

\begin{figure}
\begin{center}
\begin{tikzpicture}
\draw (1,2) circle (3pt);
\draw (3,2) circle (3pt);
\fill[black] (5,2) circle (3pt);
\fill[black] (2,2) circle (3pt);
\fill[red] (4,2) circle (3pt);
\draw [very thick](0.5,2) -- (5.5,2);
\draw (0,2) node (c) { $$};
\draw (0,0) node (d) { $$};
\draw (0.7,2) node {\Huge $[$};
\draw (2.3,2) node {\Huge $]$};
\draw (2.7,2) node {\Huge $[$};
\draw (5.3,2) node {\Huge $]$};
\draw (0.7,0) node  {\Huge $[$};
\draw (2.3,0) node {\Huge $]$};
\draw (2.7,0) node {\Huge $[$};
\draw (5.3,0) node {\Huge $]$};
\draw [very thick](0.5,-0) -- (5.5,-0);
\draw (1.5,0.2) circle (3pt);
\fill[black] (4,0) circle (3pt);
\draw (4,0.4) circle (3pt);
\fill[black] (1.5,-0.2) circle (3pt);
\fill[red] (4,-0.4) circle (3pt);
\end{tikzpicture}
\hspace{1in}
\begin{tikzpicture}
\draw (1,2) circle (3pt);
\fill[red] (4,2) circle (3pt);
\fill[black] (5,2) circle (3pt);
\draw (2,2) circle (3pt);
\fill[black] (3,2) circle (3pt);
\draw [very thick](0.5,2) -- (5.5,2);
\draw (0,2) node (c) { $$};
\draw (0,0) node (d) { $$};
\draw (0.7,2) node {\Huge $[$};
\draw (2.3,2) node {\Huge $]$};
\draw (2.7,2) node {\Huge $[$};
\draw (5.3,2) node {\Huge $]$};
\draw (0.7,0) node  {\Huge $[$};
\draw (2.3,0) node {\Huge $]$};
\draw (2.7,0) node {\Huge $[$};
\draw (5.3,0) node {\Huge $]$};
\draw [very thick](0.5,-0) -- (5.5,-0);
\draw (1.5,0.2) circle (3pt);
\fill[black] (4,0) circle (3pt);
\fill[black] (4,0.4) circle (3pt);
\draw (1.5,-0.2) circle (3pt);
\fill[red] (4,-0.4) circle (3pt);
\end{tikzpicture}
\end{center}
\caption{On the left, $\phi^{-1}(y)$ has $2!\cdot 3!$ elements, whereas on the right, $\phi^{-1}(y)$ only has $3$ elements. However, if there is only one particle of each species, then $\phi^{-1}(y)$ always has $2!\cdot 3!$ elements.}
\label{ic}
\end{figure}
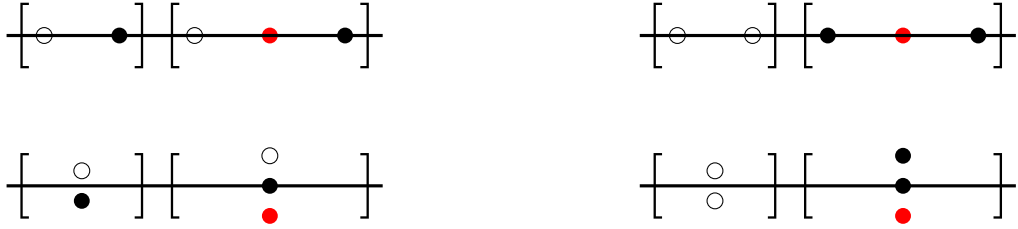

\end{reemark}

\section{Multi--species ASEP$(q,\vec{m})$ and SSEP$(\vec{m})$}
In this section, we will use the generalization of Rogers--Pitman intertwining from the previous section. This intertwining will be used to construct inhomogeneous multi--species ASEP$(q,\vec{m})$ and SSEP$(\vec{m})$, as well as to construct duality functionals. Section \ref{First} focuses on the ASEP$(q,\vec{m})$ and Section \ref{Second} will focus on SSEP$(\vec{m})$.

\subsection{ASEP$(q,\vec{m})$}\label{First}

\subsubsection{Stochastic Fusion Construction of Multi--species ASEP$(q,\vec{m})$}

Let $\Phi: S^{(n)} \rightarrow \hat{S}^{(n)}$ be the (deterministic) map defined by $\Phi := \bigotimes_{x\in \mathbb{Z}} \Phi_x$, where
\begin{align*}
\Phi_x: \{0,1,\ldots,n\}^{m_x} & \rightarrow \{(k_x^{(1)},\ldots,k_x^{(n)}): k_x^{(1)} + \ldots + k_x^{(n)} \leq m_x, \ \ k_x^{(i)} \in \mathbb{N}\} \\
\Phi_x(\epsilon_0,\ldots,\epsilon_{m_x-1}) &= (k_x^{(1)},\ldots,k_x^{(n)}),
\end{align*}
where each $\epsilon_*$ is in $\{ 0,1,\ldots,n\}$, and $k_x^{(i)}$ is the number of $\epsilon_*$ which are equal to $i$.

We will define a Markov kernel $\Lambda: \hat{S}^{(n)} \rightarrow S^{(n)}$ such that $\Lambda\Phi$ is the identity on $\hat{S}^{(n)}$. The map $\Lambda$ is defined by $\Lambda := \bigotimes_{x\in \mathbb{Z}} \Lambda_x$, where
\begin{align*}
\Lambda_x:  \{(k_x^{(1)},\ldots,k_x^{(n)}): k_x^{(1)} + \ldots + k_x^{(n)} \leq m_x, \ \ k_x^{(i)} \in \mathbb{N}\} &\rightarrow \{0,1,\ldots,n\}^{m_x} \\
\mathbb{P}( \Lambda_x(k_x^{(1)},\ldots,k_x^{(n)}) = \mathfrak{M}((x_1,\ldots,x_k),\sigma) ) &= \frac{q^{x_1+\ldots+x_k} q^{\inv(\sigma)} }{q^{(k_x-1)k_x/2}\binom{m_x}{k_x^{(1)},\ldots,k_x^{(n)}}_q}.
\end{align*}
Here, $k_x = k_x^{(1)} + \ldots + k_x^{(n)}$ and $\sigma \in S(k)$. It is straightforward that $\Lambda\Phi$ is the identity on $\hat{S}^{(n)}$. The map $\Lambda$ is a Markov operator by \eqref{qBin3}.

\begin{lemma}\label{qexc}
The maps $\Lambda$ and $\Phi$ preserve $q$--exchangeability.
\end{lemma}
\begin{proof}
First, let us show that the fission map $\Lambda$ preserves $q$--exchangeability. We do this by writing $\Lambda$ algebraically. If $(\mathbf{x},\sigma)$ is a particle configuration of ASEP, then 
$$\mathbf{x} \in \mathcal{W}_N^0 := \{(x_1,\ldots,x_N): x_N < \cdots < x_1\} \subseteq \mathbb{Z}^N.$$
Thus, $\mathbf{m}(\mathbf{x})=(1,\ldots,1)$, so $S(\mathbf{m}(\mathbf{x}))=S(N)$, meaning that $D_{H',S(\mathbf{m}(\mathbf{x}))}=D_{H'}$. Therefore, $\Lambda$ is a Markov operator   
$$
\coprod_{\mathbf{x} \in \mathcal{W}_N} \{\mathbf{x}\} \times D_{H', S(\mathbf{m}(\mathbf{x}))  } \rightarrow \mathcal{W}_N^0 \times D_{H'}.
$$
In particular,
$$
\mathrm{Prob}\left(\Lambda(\mathbf{x},\sigma)) = (\mathbf{x},\sigma b) \right) = \frac{q^{\inv(b)}}{Z}
$$
for some normalizing constant $Z$ (it equals $[N_1]_q^! \ldots [N_n]_q^!$, but we do not need to use this). By the uniqueness of the decomposition $a\sigma b$, any element of $D_{H'}$ can be uniquely written as $\sigma b$ for some $\sigma \in D_{H',S(\mathbf{m}(\mathbf{x}))  }$ and $b \in S(\mathbf{m}(\mathbf{x}))  $. Thus, to show that $\Lambda$ preserves $q$--exchangeability, it suffices to show that 
$$
q^{-\inv(\sigma b)} \cdot \operatorname{Prob}_{\Lambda} (\mathbf{y}, \sigma b) = q^{-\inv\left(\sigma^{\prime} b'\right)} \cdot \operatorname{Prob}_{\Lambda} \left(\mathbf{y}, \sigma^{\prime} b^{\prime}\right)
$$
for all $\sigma b,\sigma'b' \in D_{H'}$, where $\operatorname{Prob}_{\Lambda} $ is the pushforward of $\operatorname{Prob}$ under ${\Lambda}.$ To show this, it suffices to show that
$$
q^{-\inv(\sigma b)} q^{\inv(b)} \operatorname{Prob}_{\Lambda} (\mathbf{x},\sigma) = q^{-\inv(\sigma' b')} q^{\inv(b')} \operatorname{Prob}_{\Lambda} (\mathbf{x},\sigma'). 
$$
Using that $\inv(\sigma b) = \inv(\sigma) + \inv(b)$ and $\inv(\sigma' b') = \inv(\sigma') + \inv(b')$, this last equality reduces to the definition of $q$--exchangeability of $\mathrm{Prob}$. Therefore $\Lambda$ preserves $q$--exchangeability.

Now we show that the fusion map $\Phi$ also preserves $q$--exchangeability. If $\hat{H}$ is a Young subgroup containing $H$, then $D_{H',\hat{H}} \subseteq D_{H',H}$. Any $\sigma \in D_{H',H}$ has a unique decomposition $\sigma = \hat{\sigma} b$ where $\hat{\sigma} \in D_{H',\hat{H}}$ and $b\in \hat{H}$. But $b$ has its own unique decomposition $b= \hat{b}h$ where $\hat{b} \in D_H$ and $h\in H$. Therefore $\sigma = \hat{\sigma} \hat{b} h$, but since $\sigma \in D_{H',H}$, this must imply that $h=\mathrm{id}$. Thus, $\sigma$ has a unique decomposition $\sigma=\hat{\sigma}\hat{b}$ for $\hat{\sigma} \in D_{H',\hat{H}}$ and $\hat{b} \in D_H \cap \hat{H}$.

The fusion map is deterministic: if its image is some particle configuration $(\mathbf{x},\hat{\sigma})$ where $\hat{\sigma} \in D_{H',S(\mathbf{m}(\mathbf{x})) }$, then the fiber of this configuration can be described as $\{ (\mathbf{y}, \hat{\sigma} \hat{b}): \hat{b} \in S(\mathbf{m}(\mathbf{x})) \}$ for some $\mathbf{y} \in \mathcal{W}_N^0$. Thus, our $q$--exchangeable measure $\mathrm{Prob}$ satisfies \eqref{edef}, which can now be stated as
$$
q^{-\inv(\hat{\sigma} \hat{ b})} \cdot \operatorname{Prob}_{} (\hat{\mathbf{y}}, \hat{\sigma} \hat{ b}) = q^{-\inv\left(\hat{\sigma}^{\prime} \hat{b}'\right)} \cdot \operatorname{Prob}_{} \left(\mathbf{y}, \hat{\sigma}^{\prime} \hat{b}^{\prime}\right).
$$
Using this, the goal is to prove that
$$
q^{-\inv(\hat{\sigma}) } \cdot \operatorname{Prob}_{\Phi} (\hat{\mathbf{x}}, \hat{\sigma} ) = q^{-\inv\left(\hat{\sigma}^{\prime}\right)} \cdot \operatorname{Prob}_{\Phi} \left(\mathbf{x}, \hat{\sigma}^{\prime}\right).
$$
From the decomposition $\inv(\hat{\sigma } \hat{b}) = \inv(\hat{\sigma}) + \inv(\hat{b})$, it suffices to show that 
$$
\operatorname{Prob}_{\Phi} (\hat{\mathbf{x}}, \hat{\sigma} )  = q^{-\inv(\hat{b})} \operatorname{Prob}_{} (\hat{\mathbf{y}}, \hat{\sigma} \hat{ b}) 
$$
Setting $\hat{\sigma} = \hat{\sigma}'$ above shows that
$$
q^{-\inv( \hat{ b} )} \cdot \operatorname{Prob}_{} (\hat{\mathbf{y}}, \hat{\sigma} \hat{ b}) = q^{-\inv\left(\hat{b}'\right)} \cdot \operatorname{Prob}_{} \left(\mathbf{y}, \hat{\sigma}  \hat{b}^{\prime}\right).
$$
Setting $\hat{b}'=\mathrm{id}$, we see that it suffices to show that
$$
\operatorname{Prob}_{\Phi} (\hat{\mathbf{x}}, \hat{\sigma} )  = \operatorname{Prob}_{} (\hat{\mathbf{y}}, \hat{\sigma} ).  
$$
But this follows immediately from the definition of $\Phi$.

(iv) This follows immediately from (ii) and Lemma \ref{qexc}.

\end{proof}

We define the following continuous--time Markov process on $\hat{S}^{(n)}$, which we call the $n$--species ASEP$(q,\vec{m})$. The jump rates are similar to the jump rates for the $n$--species ASEP$(q,m/2)$. The only difference is that the $m$ should be replaced with $m_x$.

For a given $N$ and $\mathbf{N}=(N_1,\ldots,N_n)$ such that $N_1+\ldots+ N_n=N$, define the probability measure $\pi$ by
$$
\pi(\mathfrak{M}(\mathbf{x},\sigma)) \propto \frac{q^{\inv(\sigma)} q^{x_1+\ldots+x_N}}{q^{(L-1)L/2} \binom{L}{N_1,\ldots,N_n}} ,
$$
where $0\leq x_1 \leq \ldots \leq x_N \leq L-1$.

\begin{theorem}\label{SFA}
(i) The probability measure $\pi\Lambda$ on $S^{(n)}$ is stationary under the dynamics of $n$--species ASEP of a finite lattice. 

(ii) Let $X_t$ evolve as $n$--species ASEP on a finite lattice with initial condition $X_0=\pi\Lambda.$ Then $\phi(X_t)$ is Markov on $\hat{S}^{(n)}$ and evolves as an $n$--species ASEP$(q,\vec{m})$ on a finite lattice. Furthermore, the transition probabilities of $n$--species ASEP$(q,\vec{m})$ are given by $\Lambda P_t \Phi$, where $P_t$ denotes the transition probabilities of $n$--species ASEP.

(iii) The probability measure $\pi$ on $\hat{S}^{(n)}$ is stationary under the dynamics of $n$--species ASEP$(q,\vec{m})$ on a finite lattice.

(iv) The dynamics of the multi--species ASEP$(q,\vec{m})$ preserve $q$--exchangeability, as does the dynamics of the multi--species $q$--Boson.

\end{theorem}
\begin{proof}
(i) By definition, $\pi$ is $q$--exchangeable. By Lemma \ref{qexc}, $\pi\Lambda$ is also $q$--exchangeable. It is already known (from \cite{KuanAHP}) that the $n$--species ASEP preserves $q$--exchangeable measures. Since the measure proportional to $q^{x_1+\ldots+x_N}$ is stationary for single species ASEP, this shows (i).

(ii) Consider the process with transition probabilities $\Lambda P_t \Phi$. We first find the probability of a particle of species $j$ at lattice site $x$ jumping to the right and switching places with a particle of species $i<j$. In order for this jump to happen, three events need to occur. First, when the fission map $\Lambda$ splits the lattice site $x$, it needs to send a particle of species $j$ to the right--most lattice site; second, when the fission map splits the lattice site $x+1$, it needs to send a particle of species $i$ to the left--most lattice site; and third, the particle of species $j$ needs to jump to the right. Thus, using \eqref{qBin3}, the jump rate is
\begin{align*}
 & q\cdot \dfrac{   q^{k_x^{(j+1)} + \ldots + k_x^{(n)}}  \binom{m_x-1}{k_x^{(1)},\ldots, k_x^{(j)} -1, \ldots, k_x^{(n)} }_q  }{   \binom{m_x}{k_x^{(1)},\ldots,k_x^{(n)} }_q  } \cdot  \dfrac{ q^{k_{x+1}^{(1)} + \ldots + k_{x+1}^{(i-1)} }   \binom{m_{x+1}-1}{k_{x+1}^{(1)},\ldots, k_{x+1}^{(i)} -1, \ldots, k_{x+1}^{(n)} }_q  }{   \binom{m_{x+1}}{k_{x+1}^{(1)},\ldots,k_{x+1}^{(n)} }_q  } \\
 &= q\cdot q^{k_x^{(j+1)} + \ldots + k_x^{(n)}} q^{k_{x+1}^{(1)} + \ldots + k_{x+1}^{(i-1)}}  \frac{ (k_x^{(j)} )_q  }{(m_x)_q} \frac{ (k_{x+1}^{(i)})_q }{(m_{x+1})_q},
\end{align*}
which is the right jump rate for the multi--species ASEP$(q,\vec{m})$.

Similarly, in order for a particle of species $j$ at lattice site $x+1$ to jump to the left and switch places with a particle of species $i<j$, we need three events, whose probabilities multiply to
\begin{align*}
&1 \cdot \dfrac{   q^{k_x^{(1)} + \ldots + k_x^{(i-1)}}  \binom{m_x-1}{k_x^{(1)},\ldots, k_x^{(i)} -1, \ldots, k_x^{(n)} }_q  }{   \binom{m_x}{k_x^{(1)},\ldots,k_x^{(n)} }_q  } \cdot  \dfrac{ q^{k_{x+1}^{(j+1)} + \ldots + k_{x+1}^{(n)} }   \binom{m_{x+1}-1}{k_{x+1}^{(1)},\ldots, k_{x+1}^{(j)} -1, \ldots, k_{x+1}^{(n)} }_q  }{   \binom{m_{x+1}}{k_{x+1}^{(1)},\ldots,k_{x+1}^{(n)} }_q  }  \\
&= q^{ k_x^{(1)} + \ldots + k_x^{(i-1)} }  q^{k_{x+1}^{(j+1)} + \ldots + k_{x+1}^{(n)}}  \frac{ (k_x^{(i)} )_q  }{(m_x)_q} \frac{ (k_{x+1}^{(j)})_q }{(m_{x+1})_q},
\end{align*}
which is the left jump rate for the multi--species ASEP$(q,\vec{m})$.

(iii) By part (ii), we need to show that $\pi\Lambda P_t \Phi = \pi$. By (i), $\pi\Lambda P_t = \pi\Lambda$. Since $\Lambda\Phi$ is the identity, the result follows.

\end{proof}

\begin{reemark}
The $q$--exchangeability of the multi--species $q$--Boson had been previously proved in \cite{KuanAHP}, using a direct calculation.
\end{reemark}

\subsubsection{Duality Functionals}

Set
$$
m^{(z)}
= 
\begin{cases}
m_0 + \ldots + m_{z-1}, & n > 0, \\
0, & n=0,\\
-m_{-1} - \ldots - m_z, & n < 0.
\end{cases}
$$
Note that if all $m_z$ are equal to the same value $m$, then $m^{(z)}=mz$ for any $z\in \mathbb{Z}$.

Consider an ASEP$_{1,q}$ on the finite lattice $\{  m^{(z)}, m^{(z)}  + 1, \ldots, m^{(z+1)} - 1 \}$. By Remark \ref{StatDual} above, we have that 
$$
\mathbb{E}[D_{\mathrm{Sch}}(X(\infty),\vec{y})] = \sum_{m^{(z)} \leq x_k < \ldots < x_1 \leq m^{(z+1)}-1} \frac{q^{x_1+ \ldots + x_k} }{q^{km^{(z)}+k(k-1)/2} \binom{m_z}{k}_q} \prod_{i=1}^l 1_{\{ \vec{y} \subseteq \vec{x} \}} q^{ - N_{y_i}(\vec{x}) - y_i}  
$$
does not depend on the value of $\vec{y}=(y_1 > \ldots > y_l)$. In particular, set $x_i=y_i=m^{(z+1)}-i$ for $1 \leq i \leq l$. Then $x_i-y_i=0$ and $-N_{y_i}(\vec{x}) =-i+1$, so this shows that
\begin{align}\label{NoDep}
&\sum_{m^{(z)} \leq x_k < \ldots < x_1 \leq m^{(z+1)}-1} \frac{q^{x_1+ \ldots + x_k} }{q^{km^{(z)}+k(k-1)/2} \binom{m_z}{k}_q} \prod_{i=1}^l 1_{\{ \vec{y} \subseteq \vec{x} \}} q^{ - N_{y_i}(\vec{x}) - y_i} \notag   \\
&=\frac{ q^{-l(l-1)/2}  }{ q^{km^{(z)}} q^{k(k-1)/2} \binom{m_z}{k}_q } \sum_{m^{(z)} \leq x_{k} < \ldots < x_{l+1} \leq m^{(z+1)}-l-1} q^{x_{l+1} + \ldots + x_{k}} \notag \\
 &= \frac{ q^{(k-l) m^{(z)}}}  { q^{k m^{(z)}}  } \cdot q^{-l(l-1)/2} \frac{   q^{(k-l-1)(k-l)/2}}{ q^{k(k-1)/2}   } \frac{ \binom{m_z-l}{k-l}_q   } { \binom{m_z}{k}_q }  \notag \\
&= q^{-lm^{(z)}} q^{-l(k-1)} \frac{ \binom{k}{l}_q}{ \binom{m_z}{l}_q},
\end{align}
where we have used \eqref{qBin}.

\begin{prop}\label{FusedDuality}
For any $\vec{m}$, the $\hat{S} \times \hat{S}$ matrix $\Lambda D_{\mathrm{Sch}}\Phi$ is an inhomogeneous generalization of $D_{\mathrm{CGRS}}$. In particular, setting $\hat{\mathfrak{s}}(x)=k_x$ and $\hat{\mathfrak{s}}'(x)=l_x$,
$$
[\Lambda D_{\mathrm{Sch}}\Phi](\hat{\mathfrak{s}},\hat{\mathfrak{s}}') = \mathrm{const} \cdot \left(\prod_{z < w} q^{-k_zl_w} \right) \left(\prod_{z\in \mathbb{Z}} q^{-l_zm^{(z)}} \frac{\binom{k_z}{l_z}_q}{ \binom{m_z}{l_z}_q} \right).
$$
\end{prop}
\begin{proof} By \eqref{NoDep}, $[\Lambda D_{\mathrm{Sch}}\Phi](\hat{\mathfrak{s}},\hat{\mathfrak{s}}')$ equals
\begin{align*}
&\prod_{z\in \mathbb{Z}}  \sum_{m^{(z)} \leq x_{k_z}^{(z)} \leq \cdots \leq x_{k_1}^{(z)} \leq m^{(z+1)}-1} \frac{  q^{x_{k_1}^{(z)} + \ldots x_{k_z}^{(z)} }  }{  q^{ k_zm^{(z+1)} + (k_z-1)k_z/2 } \binom{m_z}{k_z}_q  } \prod_{i=1}^{l_z} 1_{ \{ \vec{y} \subseteq \vec{x} \}} q^{-y_i^{(z)}- N_{y_i^{(z)}}(\vec{x}) }  \\
&= \left( \prod_{z>w} q^{-k_zl_w}\right) \left( \prod_{z\in \mathbb{Z}} q^{-l_zm^{(z)}} q^{-l_z(k_z-1)} \frac{\binom{k_z}{l_z}_q }{   \binom{m_z}{l_z}_q }  \right) \\
&=q^l  \left(\prod_{z \geq w} q^{-k_zl_w} \right) \left(\prod_{z\in \mathbb{Z}} q^{-l_zm^{(z)}} \frac{\binom{k_z}{l_z}_q}{ \binom{m_z}{l_z}_q} \right)\\
&= q^l q^{-kl} \left(\prod_{z < w} q^{-k_zl_w} \right) \left(\prod_{z\in \mathbb{Z}} q^{-l_zm^{(z)}} \frac{\binom{k_z}{l_z}_q}{ \binom{m_z}{l_z}_q} \right).
\end{align*}

\end{proof}







\begin{reemark}
We can check by direct computation that for generic values of $\vec{m}$, $\Lambda D_{\mathrm{Sch}}\Phi$ is not a duality function for ASEP$(q,\vec{m})$.
\end{reemark}

\begin{reemark}
Following the same argument for the $n$--species ASEP$(q,\vec{m})$ would use \eqref{qBin3} instead of \eqref{qBin}. By comparing the binomial terms, it is apparent that the resulting functional would not be the duality function of $n$--species ASEP$(q,m/2)$ found in \cite{KIMRN}.  
\end{reemark}

\subsubsection{The Sch\"{u}tz duality for ASEP}\label{SchASEP}
It was already proven in \cite{Sch97} that ASEP$_{1,q}$ is dual to another ASEP$_{1,q}$ with respect to $D_{\mathrm{Sch}}$. Here, we find a way to find the duality using a bijective argument. This argument will be useful when examining the dynamic models. 

In what follows, we fix integers $r, k, m \ge 1$ and an $r$-tuple $Y = (y_1, y_2, \ldots , y_r)$ of integers such that $0 \le y_r < y_{r - 1} < \cdots < y_1 \le m - 1$. Let $\mathfrak{X} (m)$ denote the set of $X = (x_1, x_2, \ldots , x_k)$ is such that $0 \le x_t < x_{t - 1} < \cdots < x_1 \le m - 1$; for any $X \in \mathfrak{X} (m)$, define $s(X) = \sum_{j = 1}^k x_j$ and $h_X (y_i)$ to be the number of indices $j$ for which $x_j > y_i$.  

Let 
\begin{flalign*}
F(X; Y) = q^{s(X)} \displaystyle\prod_{i = 1}^r q^{-h_X (y_i) - y_i} \textbf{1}_{y_i \in X}; \qquad \mathfrak{S} (Y) = \displaystyle\sum_{X \in \mathfrak{X} (m)} F(X; Y).
\end{flalign*}
The function $F$ is simply the Schutz duality function.

\begin{lemma} 
	
\label{syiyi1}

Assume that $i \in [1, r]$ is an integer such that $y_i > y_{i + 1} + 1$ (here, we set $y_{r + 1} = 0$), and define $Z = (z_1, z_2, \ldots , z_r) = (y_1, y_2, \ldots , y_{i - 1}, y_i - 1, y_{i + 1}, y_{i + 2}, \ldots , y_r)$. Then, $\mathfrak{S} (Y) = \mathfrak{S} (Z)$. 
	
\end{lemma}

\begin{proof} 

Let $\mathcal{B} = \mathcal{B} (y_i)$ denote the set of $X \in \mathfrak{X} (m)$ such that $y_i - 1 \in X$ and $y_i \notin X$. Similarly, let $\mathcal{C} = \mathcal{C} (y_i)$ denote the set of $X \in \mathfrak{X} (m)$ such that $y_i - 1 \notin X$ and $y_i \in X$, and let $\mathcal{D} = \mathcal{D} (y_i)$ denote the set of $X \in \mathfrak{X} (m)$ such that $y_i - 1, y_i \in X$. 

Now, observe that 
\begin{flalign}
\label{scdbd}
\mathfrak{S} (Y) = C (Y) + D (Y); \qquad \mathfrak{S} (Z) = B (Z) + D (Z),
\end{flalign}

\noindent where for any $W \in \{ Y, Z \}$ we have defined 
\begin{flalign*}
B (W) = \displaystyle\sum_{X \in \mathcal{B}} F(X; W); \qquad C (W) = \displaystyle\sum_{X \in \mathcal{C}} F(X; W); \qquad D (W) = \displaystyle\sum_{X \in \mathcal{D}} F(X; W).
\end{flalign*}

Observe that if $X \in \mathcal{D}$ then $F(X; Y) = F(X; Z)$. Indeed, this follows from the fact that $\sum_{j = 1}^r \big( y_j + h_X (y_j) \big) = \sum_{j = 1}^r \big( z_j + h_X (z_j) \big)$, which holds since $y_j = z_j$ and $h_X (y_j) = h_X (z_j)$ for $j \ne i$; $y_i = z_i + 1$; and $h_X (y_i) = h_X (z_i) - 1$. Thus, $D(Y) = D(Z)$. 

Next, to any $X \in \mathcal{C}$ we can associate an $\overline{X} \in \mathcal{B}$ that is obtained by replacing $y_i \in X$ with $y_i - 1$. Similarly, $X$ can be obtained from $\overline{X}$ by replacing $y_i - 1 \in \overline{X}$ with $y_i$. Now, observe that $F(X; Y) = F(\overline{X}; Z)$ for any $X \in \mathcal{C}$. Indeed, this follows from the fact that $s(X) - \sum_{j = 1}^r \big( y_j + h_X (y_j) \big) = s(\overline{X}) - \sum_{j = 1}^r \big( z_j + h_{\overline{X}} (z_j) \big)$, which holds since $s (X) = s(\overline{X}) + 1$; $y_j = z_j$ and $h_X (y_j) = h_{\overline{X}} (z_j)$ for $j \ne i$; $y_i = z_i + 1$; and $h_X (y_i) = h_{\overline{X}} (z_i)$. Summing over $X \in \mathcal{C}$ yields $C (Y) = B (Z)$.

Now the lemma follows from \eqref{scdbd}.  
\end{proof}

\begin{corollary}
	
\label{syindependent} 

The quantity $\mathfrak{S} (Y)$ is independent of $Y$. 

\end{corollary}

\subsubsection{A ``multi--species'' modification}\label{MSM}
For $p<n$, define a Markov operator $\mathcal{Q}$ on $S^{(p)} \times S^{(n)}$, where we assume the lattice is finite. Suppose that $\tilde{\mathbf{y}} = (\tilde{y}_1,\ldots,\tilde{y}_{\tilde{r}})$ and $\mathbf{y}=(y_1,\ldots,y_r)$ for some $\tilde{r} \leq r$.  Write $\tilde{\mathbf{y}} \subseteq \mathbf{y}$ if $\tilde{\mathbf{y}} =  (y_{i_1},\ldots,y_{i_{\tilde{r}}})$ for some $y_{i_{\tilde{r}}} < \ldots < y_{i_1}$. With this ordering, we must have that $i_1 < \ldots < i_{\tilde{r}}$. Now define
$$
\mathcal{Q}(Y,\tilde{Y}) = q^{ - i_1 - \ldots - i_r} 1_{\tilde{Y} \subseteq Y}.
$$
Also define the diagonal matrix $\mathcal{P}$ on $S^{(1)}\times S^{(1)}$ by 
$$
\mathcal{P}(\vec{x},\vec{x}) = q^{-x_1 - \ldots - x_N}
$$

The notation is described pictorially below. 

\begin{center}
\begin{tikzpicture}
\draw (1,2) circle (3pt);
\fill[black] (4,2) circle (3pt);
\draw (5,2) circle (3pt);
\fill[black] (7,2) circle (3pt);
\fill[black] (2,2) circle (3pt);
\fill[black] (3,2) circle (3pt);
\fill[black] (6,2) circle (3pt);
\draw [very thick](0.5,2) -- (7.5,2);
\draw (8,2) node (a) { $ X$};
\draw (8,1) node (b) { $ Y$};
\draw (8,0) node (d) {$ \tilde{Y}$};
\draw [very thick](0.5,1) -- (7.5,1);
\draw (1,1) circle (3pt);
\draw (4,1) circle (3pt);
\draw (5,1) circle (3pt);
\draw (7,1) circle (3pt);
\fill[black] (2,1) circle (3pt);
\fill[black] (3,1) circle (3pt);
\fill[black] (6,1) circle (3pt);
\draw [very thick](0.5,0) -- (7.5,0);
\draw (1,0) circle (3pt);
\draw (4,0) circle (3pt);
\draw (5,0) circle (3pt);
\draw (7,0) circle (3pt);
\fill[red] (2,0) circle (3pt);
\draw (3,0) circle (3pt);
\fill[blue] (6,0) circle (3pt);
\end{tikzpicture}
\end{center}

\begin{theorem}\label{MSD}
(a) Suppose that $Y(t)$ evolves as an $n$--species ASEP$_{1,q}$ (on either a finite lattice with closed boundary conditions, or on the infinite line) with finitely many particles, and $D$ is a duality function between $X(t)$ and $Y(t)$ for some Markov process $X(t)$. Then the function $D\mathcal{Q}^* $ is a duality function between $X(t)$ and a $p$--species ASEP$_{1,q}$.

(b) Now suppose that $Y(t)$ evolves as a single species ASEP$_{q,1}$ on the infinite line. If $X(t)$ is any Markov process which is dual to $Y(t)$ with respect to $D$, then $X(t)$ is dual to ASEP$_{1,q}$ with respect to $D\mathcal{P}$.

\end{theorem}
\begin{proof}
(a) By Remark \ref{QInter}, it suffices to show that $\mathcal{Q}\mathcal{L}_{\text{ASEP}}^{(n)} = \mathcal{L}_{\text{ASEP}}^{(p)} \mathcal{Q}$. But this is exactly the condition that multi--species ASEP preserves $q$--exchangeability, which had already been proven in Proposition 4.5 of \cite{KuanAHP}. 

(b) Again by Remark \ref{QInter}, we just need to show that for 
satisfies $\mathcal{P}L_{\text{ASEP}_{q,1}} = L_{\text{ASEP}_{1,q}}\mathcal{P}$. First, note that on the infinite line,
$$
L_{\text{ASEP}_{q,1}}(\vec{x},\vec{x}) = L_{\text{ASEP}_{1,q}}(\vec{x},\vec{x}). 
$$
Next, if $\vec{y}$ is a particle configuration obtained from $\vec{x}$ by moving one particle to the right, then
\begin{align*}
[\mathcal{P}L_{\text{ASEP}_{q,1}}] (\vec{x}, \vec{y}) &= \mathcal{P}(\vec{x},\vec{x}) \cdot 1,\\
[L_{\text{ASEP}_{1,q}}\mathcal{P}] (\vec{x}, \vec{y}) &= q \cdot \mathcal{P}(\vec{y},\vec{y}) ,
\end{align*}
which are equal. Similarly,  if $\vec{y}$ is a particle configuration obtained from $\vec{x}$ by moving one particle to the left, then the needed equality $ \mathcal{P}(\vec{x},\vec{x}) \cdot q = 1 \cdot \mathcal{P}(\vec{y},\vec{y}) $, which is true.

\end{proof}

As a corollary, we have new duality functions for ASEP. Note that (a) and (b) below yield the duality results in \eqref{lim1} and \eqref{lim2}, which arose from the limit of dynamic ASEP.

\begin{corollary}\label{newdual}
(a)
Suppose $\mathfrak{s}_t$ evolves as a ASEP$_{q,1}$ and $\vec{x}_t$ evolves as an ASEP$_{1,q}$ with $r$ particles. Then for any $\tilde{r} \leq r$, the function
$$
\sum_{\substack{ I \subseteq \{1,\ldots,r\} \\ \vert I \vert = \tilde{r}}   }  \prod_{i \in I} q^{N_{x_i}(\mathfrak{s})-i+1}
$$
is a duality function between $\mathfrak{s}_t$ and $\vec{x}_t$ on the infinite line.

(b) Suppose $\mathfrak{s}_t$ evolves as a ASEP$_{1,q}$ and $\vec{x}_t$ evolves as an ASEP$_{1,q}$ with $r$ particles. Then for any $\tilde{r} \leq r$, the function
$$
q^{-x_1-\ldots-x_r}\sum_{\substack{ I \subseteq \{1,\ldots,r\} \\ \vert I \vert = \tilde{r}}   }  \prod_{i \in I} q^{-N_{x_i}(\mathfrak{s})+i-1}
$$
is a duality function between $\mathfrak{s}_t$ and $\vec{x}_t$ on the infinite line.

(c) Suppose $\mathfrak{s}_t$ evolves as a ASEP$_{1,q}$ and $\vec{x}_t$ evolves as an ASEP$_{q,1}$ with $r$ particles. Then 
$$
\prod_{k=1}^r 1_{\{\mathsf{s}(x_k)=1\}}q^{ -N_{x_k}(\mathfrak{s})}.
$$ 
is a duality function between $\mathfrak{s}_t$ and $\vec{x}_t$ on the infinite line.

(d) 
Suppose $\mathfrak{s}_t$ evolves as a  ASEP$_{q,1}$ and $\vec{x}_t$ evolves as an ASEP$_{1,q}$ with $r$ particles. Then for any $\tilde{r} \leq r$, the function
$$
1_{\{\mathfrak{s}(x_1)=\cdots \mathfrak{s}(x_k)=1\}}\sum_{\substack{ I \subseteq \{1,\ldots,r\} \\ \vert I \vert = \tilde{r}}   }  \prod_{i \in I} q^{N_{x_i}(\mathfrak{s})-i+1}
$$
is a duality function between $\mathfrak{s}_t$ and $\vec{x}_t$ on the infinite line, and the function
$$
1_{\{\mathfrak{s}(x_1)=\cdots \mathfrak{s}(x_k)=0\}}\sum_{\substack{ I \subseteq \{1,\ldots,r\} \\ \vert I \vert = \tilde{r}}   }  \prod_{i \in I} q^{N_{x_i}(\mathfrak{s})-i+1}
$$
is a duality function between $\mathfrak{s}_t$ and $\vec{x}_t$ either on the infinite line, or on a finite interval with closed boundary conditions. 

\end{corollary}
\begin{proof}
(a) After replacing $q\rightarrow q^{-1}$ and applying Theorem \ref{MSD}(a), we need to show that this duality function equals $D_{\text{BCS}}\mathcal{Q}^*$. This follows from an immediate calculation.

(b) Take $q\rightarrow q^{-1}$ in part (a) and rescale time by $q$. Now applying Theorem \ref{MSD}(b)  yields the result. 

(c) This follows by applying Theorem \ref{MSD}(b) to $D_{\mathrm{Sch}}$. 

(d) After substituting $q\rightarrow q^{-1}$ and rescaling time by $q$, the first statement follows from applying Theorem \ref{MSD}(a) to the function in part (c). The second statement follows from applying Theorem \ref{MSD}(a) to $D_{\mathrm{Kua}}$. 

\end{proof}

\begin{reemark}
The result in part (c) follows from \cite{YL19}, where the duality is proven for the six-vertex model; Because the six-vertex model degenerates to ASEP, this implies (c).
\end{reemark}

The applications of Theorem \ref{MSD}(a) above only use that ASEP$_{1,q}$ preserves $q$--exchangeability. Because the same holds for ASEP$(q,m/2)$ and the $q$--Boson, the same arguments as above imply the following:

\begin{prop}
If $\vec{x}_t$ evolves as a $q$--Boson and $\mathfrak{s}_t$ evolves as a space--reversed $q$--Boson, then $\mathfrak{s}_t$ and $\mathfrak{s}_t'$ are dual with respect to the function
$$
\sum_{\substack{ I \subseteq \{1,\ldots,r\} \\ \vert I \vert = \tilde{r}}   }  \prod_{i \in I}  q^{N_{x_i}(\mathfrak{s}) - i + 1}.
$$
and to the function
$$
\prod_{k=1}^r (q^{k-1}-q^{N_{x_k}(\mathfrak{s})}).
$$

If $\vec{x}_t$ evolves as an ASEP$(q,m/2)$ and $\mathfrak{s}_t$ evolves as a space--reversed ASEP$(q,j)$, then $\mathfrak{s}_t$ and $\mathfrak{s}_t'$ are dual with respect to the function
$$
\sum_{\substack{ I \subseteq \{1,\ldots,r\} \\ \vert I \vert = \tilde{r}}   }  \prod_{i \in I}  \frac{  \binom{m-\mathfrak{s}(x_i)}{\mathfrak{s}'(x_i)}_q }{  \binom{m}{\mathfrak{s}'(x_i)}_q } q^{ - i + 1 -\mathfrak{s}'(x_i) \sum_{z \geq x} \mathfrak{s}(z)} 1_{\{ m-  \mathfrak{s}(x_i) \geq \mathfrak{s}'(x_i)  \} }.
$$
\end{prop}

\subsection{Symmetric limit}\label{Second}

\subsubsection{Multi--species SSEP$(\vec{m})$}
In the symmetric case, we obtain that the jump rates for a particle of species $j$ at site $x \in \mathcal{G}$ to switch places with a particles of species $i$ (where $i<j$) at site $y \in \mathcal{G}$ are
$$
 p(x,y) \frac{k_x^{(j)}}{m_x} \frac{k_y^{(i)}}{m_y},
$$
where $p$ is the symmetric stochastic matrix on $\mathcal{G}$ from section \ref{alm}.

Note that when all $m_x=m_y$, we obtain a multi--species version of a time--rescaled version of the model from \cite{GKRV}. Also note that all jump rates are bounded above by $1$. 

\subsubsection{Open boundary conditions when $m_x\rightarrow\infty$}\label{opeen}
Consider a subset of $\mathcal{G}$ which we denote $\partial \mathcal{G}$. Suppose that $m_x=M$ for all $x \in \partial \mathcal{G}$. 

Then the jump rate for a particle of species $j$ located at $x \in \partial \mathcal{G}$ to switch places with a particle of species $i$ (for $i<j$) located at $y \notin \partial\mathcal{G}$ is
$$
p(x,y) \frac{k_x^{(j)}}{M} \frac{k_y^{(i)}}{m_y}
$$
and the jump rate for a particle of species $i$ located at $x \in \partial \mathcal{G}$ to switch places with a particle of species $j$ (for $i<j$) located at $y \notin \partial \mathcal{G}$ is
$$
p(x,y) \frac{k_x^{(i)}}{M} \frac{k_y^{(j)}}{m_y} .
$$
The particles away from $\partial \mathcal{G}$ interact as a multi--species SEP$(\vec{m})$.

In particular, if there is only one species of particle, and there are $\alpha_x M$ particles located at $x \in \partial \mathcal{G}$, then in the $M\rightarrow \infty$ limit one obtains a process which we call the SEP$(\vec{m})$ on $\mathcal{G}-\partial\mathcal{G}$ with open boundary conditions. The jump rate for a particle from $y\notin \partial\mathcal{G}$ to $x\in \partial\mathcal{G}$ is given by $p(y,x)(1-\alpha_x)k_y/m_y$, and the jump rate for a particle from $x\in \partial\mathcal{G}$ to $y\notin \partial\mathcal{G}$ is given by $p(x,y)\alpha_x(m_y-k_y)/m_y$. (Recall that $p$ is symmetric, so $p(x,y)=p(y,x)$). See Figure \ref{open} for an example in the case of SSEP. Each $x\in \partial \mathcal{G}$ is viewed as a reservoir with infinitely many particles. More generally, if there are $n$ species of particles and $\alpha^{(j)}M$ particles of species $j$ located at $x$, then as $M\rightarrow\infty$ one obtains an $n$--species SEP$(\vec{m})$ with open boundary conditions.

Note that if we restrict to the open boundary SEP$(\vec{m})$ with only finitely many particles, then all $\alpha_x\rightarrow 0$, so particles can exit $\mathcal{G} - \partial \mathcal{G}$ into $\partial \mathcal{G}$, but cannot enter $\mathcal{G}$ from $\partial \mathcal{G}$.

The open boundary conditions described here are a generalization of the ``boundary--driven'' symmetric exclusion process described in section 4.4 of \cite{GKRV}. There, if a particle is located at $x \in \partial \mathcal{G}$, there is a unique $y\in \mathcal{G}$ such that $p(x,y)=1$. Here, we allow the full range of values for $p(x,y)$, which allows for a wider range of open boundary conditions. These open boundary conditions are the same as those of Theorem 4.1 of \cite{CGRConsistent}, which considered some initial conditions. 

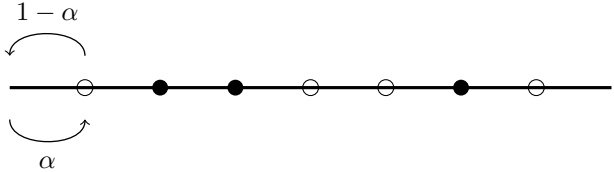
\begin{figure}
\caption{In this example, there are open boundary conditions. Particles enter at rate $\alpha$ and exit at rate $1-\alpha$. }
\begin{center}
\begin{tikzpicture}
\draw [very thick](0,0) -- (8,0);
\draw (1,0) circle (3pt);
\draw (4,0) circle (3pt);
\draw (5,0) circle (3pt);
\draw (7,0) circle (3pt);
\fill[black] (2,0) circle (3pt);
\fill[black] (3,0) circle (3pt);
\fill[black] (6,0) circle (3pt);
\draw (1,0.3) node (a) { };
\draw (0,-0.3) node (a') { };
\draw (0,0.3) node (b) { };
\draw (1,-0.3) node (c) {};
\draw (0.5,1) node  {$1-\alpha$};
\draw (0.5,-1) node {$\alpha$};
\draw (a) edge[out=90,in=90,->, line width=0.5pt] (b);
\draw (a') edge[out=-90,in=-90,->, line width=0.5pt] (c);
\end{tikzpicture}
\end{center}
\label{open}
\end{figure}

\subsubsection{Duality}

\begin{prop}\label{symm}
(a) For any $\vec{m}$, the $\hat{S}\times \hat{S}$ matrix $\Lambda D_{\mathrm{Spi}} \Phi$ is an inhomogeneous generalization of $D_{\mathrm{GKRV}}$, up to a constant factor. In particular,
$$
[\Lambda D_{\mathrm{Spi}} \Phi ](\mathfrak{s},\mathfrak{s}') = \prod_{x \in \mathcal{G}} \frac{  \binom{\mathfrak{s}(x)}{\mathfrak{s}'(x)} }{  \binom{m_x}{\mathfrak{s}'(x)} } 1_{\{\mathfrak{s}(x) \geq \mathfrak{s}'(x) \}}
$$
(b) Suppose that $m_y=M$ for $y \in \partial \mathcal{G}$. Then
$$
[\Lambda D_{\mathrm{Spi}} \Phi ](\mathfrak{s},\mathfrak{s}') = \prod_{ y \in \partial \mathcal{G} } \frac{  \mathfrak{s}(y) (  \mathfrak{s}(y)  -1) \cdots ( \mathfrak{s}(y)  -  \mathfrak{s}'(y)  + 1) }{M(M-1)\cdots ( M-\mathfrak{s}'(y) + 1 )}  \prod_{x  \in \mathcal{G}-\partial\mathcal{G}} \frac{  \binom{\mathfrak{s}(x)}{\mathfrak{s}'(x)} }{  \binom{m_x}{\mathfrak{s}'(x)} } 1_{\{\mathfrak{s}(x) \geq \mathfrak{s}'(x) \}} .
$$
In particular, if $\mathfrak{s}(y) = \alpha_y M$, then in the $M\rightarrow\infty$ limit the function becomes
$$
\prod_{ y \in \partial \mathcal{G} } \alpha_y^{\mathfrak{s}'(y)}  \prod_{x  \in \mathcal{G}-\partial\mathcal{G}} \frac{  \binom{\mathfrak{s}(x)}{\mathfrak{s}'(x)} }{  \binom{m_x}{\mathfrak{s}'(x)} } 1_{\{\mathfrak{s}(x) \geq \mathfrak{s}'(x) \}}.
$$

\end{prop}
\begin{proof}
(a) This follows immediately by taking $q\rightarrow 1$ in Proposition \ref{FusedDuality}.

(b) This follows from a direct calculation. 

\end{proof}

We next show how the SEP$(\vec{m})$ fits into the general framework of Proposition \ref{RPInter}.

\begin{prop}\label{aaaa} (a) Let $\tilde{D}^{(n)} = \Lambda D_{\mathrm{Spi}}^{(n)} \Phi$. Let $Q_t^{(n)}$ denote the semigroup of transition probabilities of $n$--species SEP$(\vec{m})$ on a finite lattice. Suppose that $\pi$ is uniform measure supported on particle configurations where there is at most particle one of each species. Similarly, assume that $\mathfrak{s}$ is a particle configuration with at most one particle of each species. Then
$$
\pi Q_t^{(n)} \tilde{D}^{(n)}  = \pi \tilde{D}^{(n)} Q_t^{(n)*}.
$$


(b) Suppose that $\hat{\mathfrak{s}}, \hat{\mathfrak{s}}'\in \hat{S}^{(3)}$, where $\hat{\mathfrak{s}}$ only contains particles of species $3$ and $1$, while $\hat{\mathfrak{s}}'$ only contains particles of species $2$ and $0$. Let $k$ denote the number of species $3$ particles, let $l$ denote the number of species $1$ particles, and let $m=\sum_x m_x$, which we assume to be finite. Let $\Pi$ be the partition of $\{0,1,\ldots,2m-1\}$ into the four blocks $\{0,1,\ldots,m-l-1\}, \{m-l,\ldots,2m-k-l\}, \{2m-k-l+1,\ldots,2m-k-1\}, \{2m-k,\ldots, 2m-1\}.$ Then
$$
[{\Pi}^* \tilde{D}^{(n)} \Pi](\hat{\mathfrak{s}},\hat{\mathfrak{s}}') = \mathrm{const} \cdot D_{\mathrm{GKRV}}(\hat{\mathfrak{s}},\hat{\mathfrak{s}}'),
$$
where on the right--hand--side $\hat{\mathfrak{s}}$ and $\hat{\mathfrak{s}}'$ are viewed as single species particle configurations.

\end{prop}

\begin{proof} (a) We just need to check that the conditions of Proposition \ref{RPInter} hold. First we need to check that the conditions of Theorem \ref{RPThm} hold. It is immediate that $\Lambda\Phi = \mathrm{id}$ and $n$--species SSEP is reversible, so it just remains to show that
$$
\pi \Lambda P_t = \pi \Lambda P_t \Phi \Lambda.
$$
By Theorem \ref{SFA}, $\pi\Lambda$ is stationary, so the left--hand--side is $\pi\Lambda P_t = \pi \Lambda$. The right--hand--side is $\pi \Lambda \Phi \Lambda$, which equals $\pi\Lambda$ because $\Lambda\Phi$ is the identity.

Next we just need to show that $\Lambda^* = \mathfrak{c}\Phi$ and $\pi D^{(n)}\Phi \Lambda = \pi D^{(n)}$. To see that $\Lambda^* = \mathfrak{c}\Phi$, we just need (by Remark \ref{Constant}) that $\mathfrak{c}^{-1} := \vert \phi^{-1}(\mathfrak{s})\vert$ is the same value for every $\mathfrak{s}$. This indeeds hold when there is exactly one particle of each species, where $\vert \phi^{-1}(\mathfrak{s}) \vert = \prod_x m_x!$. 

To see that $\pi D^{(n)}\Phi \Lambda = \pi D^{(n)}$, we first write it equivalently as
$$
\sum_{\mathfrak{s}'} [\pi D^{(n)}] ( \mathfrak{s}' )   [\Phi\Lambda] ( \mathfrak{s}' , \mathfrak{s}) = [\pi D^{(n)}](\mathfrak{s}).
$$
Because $\Lambda\Phi$ is the identity, the expression $[\Phi\Lambda] ( \mathfrak{s}' , \mathfrak{s}) = \Lambda(\phi(\mathfrak{s}'),\mathfrak{s})  $ can only be nonzero if $\phi(\mathfrak{s}') = \phi(\mathfrak{s})$. Therefore it suffices to prove that
$$
\sum_{\mathfrak{s}' \in \phi^{-1}(\phi(\mathfrak{s}))} [\pi D^{(n)}] ( \mathfrak{s}' )   \Lambda( \phi(\mathfrak{s}) , \mathfrak{s}) = [\pi D^{(n)}](\mathfrak{s}),
$$
which simplifies to 
$$
\mathfrak{c}\sum_{\mathfrak{s}' \in \phi^{-1}(\phi(\mathfrak{s}))}  [\pi D^{(n)}] ( \mathfrak{s}' )    =  [\pi D^{(n)}](\mathfrak{s}).
$$
Because the sum is over $\mathfrak{c}^{-1}$ terms, it suffices to show that
$$
[\pi D^{(n)}](\mathfrak{s}) = [\pi D^{(n)}](\mathfrak{s}'). 
$$
But this follows from Remark \ref{StatDual}.

(b) Because $\hat{\mathfrak{s}}$ only contains particles of species $3$ and $1$, while $\hat{\mathfrak{s}}'$ only contains particles of species $2$ and $0$, for all $\hat{\mathfrak{s}}^{(n)}$ and $\hat{\mathfrak{s}}'^{(n)}$ such that $\Pi^*(\hat{\mathfrak{s}}, \hat{\mathfrak{s}}^{(n)})$ and $\Pi( \hat{\mathfrak{s}}'^{(n)}, \hat{\mathfrak{s}}')$ are nonzero, the duality function $D(\hat{\mathfrak{s}}^{(n)}, \hat{\mathfrak{s}}'^{(n)})$ takes the same value. The number of such $\hat{\mathfrak{s}}^{(n)}$ is $k!(m-k)!$, and the number of such $\hat{\mathfrak{s}}'^{(n)}$ is $l!(m-l)!$, which are constant under the dynamics. The value of $D(\hat{\mathfrak{s}}^{(n)}, \hat{\mathfrak{s}}'^{(n)})$ is the same as $D_{\text{GKRV}}(\hat{\mathfrak{s}}^{(n)}, \hat{\mathfrak{s}}'^{(n)})$, by the same argument as in Proposition \ref{symm}(a).

\end{proof}

By comparing Theorem \ref{SFA} and Proposition \ref{symm} in light of Proposition \ref{RPInter}, it is natural to make the following duality Ansatz (note that this is a generalization of both Theorem 4.6 of \cite{GKRV} and Theorem 4.1 of \cite{CGRConsistent}):

\begin{theorem}\label{bbbb}
Set all $m_x=m$ for $x \in \mathcal{G} - \partial \mathcal{G}$. 
Let $\mathfrak{s}_t$ evolve as an open SEP($m/2)$ on $\mathcal{G} - \partial \mathcal{G}$, with particle reservoirs at $\partial \mathcal{G}$, as described in section \ref{opeen}. Let $\mathfrak{s}'_t$ evolve as SEP$(\vec{m})$ on $\mathcal{G}$ with finitely many particles, also described in section \ref{opeen}. Then $\mathfrak{s}_t$ and $\mathfrak{s}'_t$ are dual with respect to the function
$$
\prod_{ y \in \partial \mathcal{G} } \alpha_y^{\mathfrak{s}'(y)}  \prod_{x  \in \mathcal{G}-\partial\mathcal{G}} \frac{  \binom{\mathfrak{s}(x)}{\mathfrak{s}'(x)} }{  \binom{m}{\mathfrak{s}'(x)} } 1_{\{\mathfrak{s}(x) \geq \mathfrak{s}'(x) \}}.
$$
\end{theorem}
\begin{proof}
Let $L$ be the generator of $\mathfrak{s}_t$ and let $L'$ be the generator of $\mathfrak{s}'_t$. If $D(\mathfrak{s},\mathfrak{s}')$ denote the duality function, then the goal is to prove that
\begin{equation}\label{TwoSite}
LD (\mathfrak{s},\mathfrak{s}')= DL'^*(\mathfrak{s},\mathfrak{s}')
\end{equation}
for every $\mathfrak{s},\mathfrak{s}'$. Let $\mathfrak{s}'_{x\rightarrow y}$ indicate the particle configuration $\mathfrak{s}'$ after a particle has jumped from lattice site $x$ to $y$. Let $\mathfrak{s}'_{x-}$ denote the particle configuration after a particle from site $x$ has been removed, and similarly let $\mathfrak{s}'_{x+}$ denote the particle configuration after a particle has been added to site $x$. Then \eqref{TwoSite} is equivalent to 
\begin{multline*}
L(\mathfrak{s},\mathfrak{s})D(\mathfrak{s},\mathfrak{s}') 
+ \sum_{ x,y \in ( \mathcal{G}- \partial\mathcal{G})^2  } L(\mathfrak{s},{\mathfrak{s}}_{x\rightarrow y}  )D(\mathfrak{s}_{x\rightarrow y},\mathfrak{s}') + \sum_{x \in \mathcal{G} - \partial \mathcal{G}} \left( L(    \mathfrak{s},\mathfrak{s}_{x-}     )D(  \mathfrak{s}_{x-} , \mathfrak{s}'   ) + L(    \mathfrak{s},\mathfrak{s}_{x+}     )D(  \mathfrak{s}_{x+} , \mathfrak{s}'   ) \right) \\
= D(\mathfrak{s},\mathfrak{s}')L'( \mathfrak{s}' ,\mathfrak{s}') + \sum_{x,y \in \mathcal{G}} D(\mathfrak{s}, \mathfrak{s}'_{x\rightarrow y}) L'(\mathfrak{s}', \mathfrak{s}'_{x\rightarrow y}).
\end{multline*}
Let $L_{\mathrm{cl}}$ and $L'_{\mathrm{cl}}$ denote the generators on $\mathcal{G}- \partial \mathcal{G}$ with closed boundary conditions. Since $L_{\mathrm{cl}}D = DL_{\mathrm{cl}}^{'*}$, it suffices to prove that
\begin{multline*}
L(\mathfrak{s},\mathfrak{s})D(\mathfrak{s},\mathfrak{s}')  - L_{\mathrm{cl}}(\mathfrak{s} ,\mathfrak{s})D(\mathfrak{s},\mathfrak{s}') + \sum_{x \in \mathcal{G} - \partial \mathcal{G}} \left( L(    \mathfrak{s},\mathfrak{s}_{x-}     )D(  \mathfrak{s}_{x-} , \mathfrak{s}'   ) + L(    \mathfrak{s},\mathfrak{s}_{x+}     )D(  \mathfrak{s}_{x+} , \mathfrak{s}'   ) \right) \\
= D(\mathfrak{s},\mathfrak{s}')L'( \mathfrak{s}' ,\mathfrak{s}') - D(\mathfrak{s},\mathfrak{s}')L_{\mathrm{cl}}'( \mathfrak{s}' ,\mathfrak{s}')  + \sum_{ \substack{ x,y \in \mathcal{G}  \\ (x,y) \notin (\mathcal{G} - \partial \mathcal{G})^2 }  } D(\mathfrak{s}, \mathfrak{s}'_{x\rightarrow y}) L'(\mathfrak{s}', \mathfrak{s}'_{x\rightarrow y}) .
\end{multline*}
Now using that
\begin{align*}
- L(\mathfrak{s},\mathfrak{s}) &= -L_{\mathrm{cl}}(\mathfrak{s},\mathfrak{s}) + \sum_{x \in \mathcal{G} - \partial \mathcal{G}} \left( L(    \mathfrak{s},\mathfrak{s}_{x-}     ) + L(    \mathfrak{s},\mathfrak{s}_{x+}     ) \right)\\
- L'(\mathfrak{s}',\mathfrak{s}') &= -L'_{\mathrm{cl}}(\mathfrak{s}',\mathfrak{s}') + \sum_{ \substack{ x,y \in \mathcal{G}  \\ (x,y) \notin (\mathcal{G} - \partial \mathcal{G})^2 }  }  L'(\mathfrak{s}', \mathfrak{s}'_{x\rightarrow y}),
\end{align*}
it suffices to show that
\begin{multline*}
\sum_{x \in \mathcal{G} - \partial \mathcal{G}} \left( L(    \mathfrak{s},\mathfrak{s}_{x-}     )[ D(  \mathfrak{s}_{x-} , \mathfrak{s}'   ) -D(\mathfrak{s},\mathfrak{s}') ]+ L(    \mathfrak{s},\mathfrak{s}_{x+}     )[D(  \mathfrak{s}_{x+} , \mathfrak{s}'   ) -D(\mathfrak{s},\mathfrak{s}') ]\right) \\
=  \sum_{ \substack{ x,y \in \mathcal{G}  \\ (x,y) \notin (\mathcal{G} - \partial \mathcal{G})^2 }  } [D(\mathfrak{s}, \mathfrak{s}'_{x\rightarrow y}) - D(\mathfrak{s},\mathfrak{s}')] L'(\mathfrak{s}', \mathfrak{s}'_{x\rightarrow y})
\end{multline*}
Since particles can not exit the sink sites in the $\mathfrak{s}'$ process, the summation on the right--hand--side can be replaced with $x\in \mathcal{G}- \partial\mathcal{G}, y\in \partial\mathcal{G}$. Now using that
\begin{align*}
L(    \mathfrak{s},\mathfrak{s}_{x-}     )= \sum_{y \in \partial\mathcal{G}} p(x,y)(1-\alpha_y)& \frac{\mathfrak{s}(x)}{m}, \quad \quad  \quad \quad L(    \mathfrak{s},\mathfrak{s}_{x+}     )= \sum_{y \in \partial\mathcal{G}} p(y,x) \alpha_y \frac{m-\mathfrak{s}(x)}{m},\\
& L'(\mathfrak{s}', \mathfrak{s}'_{x\rightarrow y}) = p(x,y)\frac{\mathfrak{s}'(x)}{m},
\end{align*}
we see that it suffices to show that
\begin{multline*}
\sum_{ \substack{x \in \mathcal{G} - \partial \mathcal{G} \\ y \in \partial \mathcal{G}  } } p(x,y) \left( (1-\alpha_y)\mathfrak{s}(x) [ D(  \mathfrak{s}_{x-} , \mathfrak{s}'   ) -D(\mathfrak{s},\mathfrak{s}') ]+ \alpha_y(m-\mathfrak{s}(x))[D(  \mathfrak{s}_{x+} , \mathfrak{s}'   ) -D(\mathfrak{s},\mathfrak{s}') ]\right) \\
=  \sum_{ \substack{x \in \mathcal{G} - \partial \mathcal{G} \\ y \in \partial \mathcal{G}  } } p(x,y) [D(\mathfrak{s}, \mathfrak{s}'_{x\rightarrow y}) - D(\mathfrak{s},\mathfrak{s}')] \mathfrak{s}'(x),
\end{multline*}
where we used that $p(x,y)=p(y,x)$. Thus, it suffices to show that 
\begin{equation}\label{Need}
(1-\alpha_y)\mathfrak{s}(x) [ D(  \mathfrak{s}_{x-} , \mathfrak{s}'   ) -D(\mathfrak{s},\mathfrak{s}') ]+ \alpha_y(m-\mathfrak{s}(x))[D(  \mathfrak{s}_{x+} , \mathfrak{s}'   ) -D(\mathfrak{s},\mathfrak{s}') ] = [D(\mathfrak{s}, \mathfrak{s}'_{x\rightarrow y}) - D(\mathfrak{s},\mathfrak{s}')] \mathfrak{s}'(x)
\end{equation}
for all $x \in \mathcal{G} - \partial \mathcal{G} , y \in \partial \mathcal{G}$.

Note that we have demonstrated that the duality result does not depend on the actual values of $p(x,y)$, as long as $p$ is symmetric. Therefore, the situation reduces to that of Theorem 4.6 of \cite{GKRV}, since that theorem ultimately uses an identity equivalent to \eqref{Need}. For completeness, we finish the remainder of the proof here as well. 

So far we have not used the actual expression for the duality function $D$. To incorporate this information, we split up into various cases. Note that in the factorized form of $D$, we can cancel out all contributions from vertices other than $x$ and $y$.

\underline{Case 1: $\mathfrak{s}'(x) > \mathfrak{s}(x)+1$}

In this case, the value of $D$ is zero throughout \eqref{Need}, so it is immediate that the identity holds.

\underline{Case 2: $\mathfrak{s}'(x) = \mathfrak{s}(x)+1$}

In this case, $D(  \mathfrak{s}_{x-} , \mathfrak{s}'   )  = D(\mathfrak{s},\mathfrak{s}') = 0$, so it suffices to show that
$$
\alpha_y(m-\mathfrak{s}(x))\frac{ \binom{\mathfrak{s}(x) + 1}{\mathfrak{s}'(x)} }{ \binom{m}{\mathfrak{s}'(x)}} = \alpha_y  \frac{ \binom{\mathfrak{s}(x) }{\mathfrak{s}'(x) - 1} }{ \binom{m}{\mathfrak{s}'(x)-1}} \mathfrak{s}'(x).
$$
This is true by the definition of binomials.

\underline{Case 3: $\mathfrak{s}'(x) = \mathfrak{s}(x)$}

In this case, $D(  \mathfrak{s}_{x-} , \mathfrak{s}'   ) =0,  D(\mathfrak{s},\mathfrak{s}') \neq 0$, so it suffices to show that
$$
-(1-\alpha_y)\mathfrak{s}(x) D(\mathfrak{s},\mathfrak{s}') + \alpha_y(m-\mathfrak{s}(x))\left(\frac{D(  \mathfrak{s}_{x+} , \mathfrak{s}'   )}{D(\mathfrak{s},\mathfrak{s}')}  -1\right) D(\mathfrak{s},\mathfrak{s}')  = \left(\frac{D(\mathfrak{s}, \mathfrak{s}'_{x\rightarrow y})}{D(\mathfrak{s},\mathfrak{s}')}-1 \right)  D(\mathfrak{s},\mathfrak{s}') \mathfrak{s}'(x),
$$
which, by substituting the binomials, is equivalent to 
$$
-(1-\alpha_y)\mathfrak{s}(x) + \alpha_y(m-\mathfrak{s}(x))\left( \mathfrak{s}(x) + 1 - 1\right)   = \left( \alpha_y(m-\mathfrak{s}'(x)+1)-1 \right)   \mathfrak{s}'(x),
$$
which can be seen immediately to hold.

\underline{Case 4: $\mathfrak{s}'(x) < \mathfrak{s}(x)$}

Now, none of the $D$ terms in \eqref{Need} are zero, so we directly plug in 
$$
(1-\alpha_y)\mathfrak{s}(x) \left[ \frac{\mathfrak{s}(x) - \mathfrak{s}'(x)}{\mathfrak{s}(x)}-1 \right]+ \alpha_y(m-\mathfrak{s}(x))\left[ \frac{\mathfrak{s}(x)+1}{\mathfrak{s}(x)+1-\mathfrak{s}'(x)} - 1 \right] = \left( \alpha_y \frac{m-\mathfrak{s}'(x)+1}{\mathfrak{s}(x) - \mathfrak{s}'(x)+1} -1 \right)   \mathfrak{s}'(x),
$$
which simplifies to
$$
\mathfrak{s}'(x) + (m-\mathfrak{s}(x))\left[ \frac{\mathfrak{s}(x)+1}{\mathfrak{s}(x)+1-\mathfrak{s}'(x)} - 1 \right] = \mathfrak{s}'(x)\frac{m-\mathfrak{s}'(x)+1}{\mathfrak{s}(x) - \mathfrak{s}'(x)+1}.
$$
Multiplying through by $\mathfrak{s}(x) - \mathfrak{s}'(x)+1$ verifies the identity, completing the proof.

\end{proof}

As an application of the duality, we have the two results for the open SEP$(\vec{m})$: the first about hydrodynamic limit and the second about stationary measures. Here, we taken $\mathcal{G} = \{0,1,2,\ldots\}$ and $\partial \mathcal{G} = \{0\}$. Set $\alpha=\alpha_0 \in (0,1]$ and $p(x,x+1) = p(x+1,x) = \gamma \in (0,1/2]$ for all $x\geq 1$, while $p(0,1)=p(1,0)=1$. By rescaling time by $\gamma^{-1}$, which does not affect the duality result, so that the jump rates in the bulk are equal to $1$, we can view the entrance rates as $\gamma^{-1}\alpha$ and the exit rates as $\gamma^{-1}(1 - \alpha)$. Note that in the previous works of \cite{Ohk17}, there was a duality result which required extra sites called ``copying sites'' whenever the sum of the entrance and exit rates was not equal to $1$. We do not require such sites here.

Recall that $\mathrm{erfc}(z)$ is the complementary error function defined by 
$$
\mathrm{erfc}(z) = \frac{2}{\sqrt{\pi}}\int_z^{\infty} e^{-t^2}dt,
$$
and that its integral is
$$
\int_x^{\infty} \mathrm{erfc}(z)dz = \frac{e^{-x^2}}{\sqrt{\pi}} - x \cdot \mathrm{erfc}(x).
$$

\begin{theorem}\label{Opened}
(a) Set $\mathfrak{s}_0(x)=0$ for all $x$, and let $\rho_t(x)$ be the density profile of $\mathfrak{s}_t$, i.e. 
$$
\rho_t(x) = \frac{1}{m}\left( \mathbb{P}( \mathfrak{s}_t(x)=1) + 2 \cdot \mathbb{P}( \mathfrak{s}_t(x)=2) + \ldots + m \cdot \mathbb{P}(\mathfrak{s}_t(x)=m) \right).
$$

In the hydrodynamic limit,
\begin{align*}
\lim_{L\rightarrow \infty} \rho_ {\gamma^{-1} m\tau L}\left( \lfloor \chi L^{1/2} \rfloor \right) &= \alpha \cdot \mathrm{erfc}\left( \frac{\chi}{\sqrt{2\tau}}\right)\\
\mathcal{N}(\chi,\tau):=\lim_{L\rightarrow \infty} \mathbb{E}[m^{-1} N_{\chi L^{1/2}}(\mathfrak{s}_{\gamma^{-1} m \tau L})] &= \alpha \cdot \frac{\sqrt{2 \tau}}{\sqrt{\pi}} \exp\left( - \frac{\chi^2}{2\tau} \right) - \alpha \cdot \chi \cdot \mathrm{erfc}\left( \frac{\chi}{\sqrt{2\tau}}\right), 
\end{align*}
and $\mathcal{N}_{\alpha}(\chi,\tau)$ solves the $(1+1)$--dimensional heat equation 
$$
 \frac{\partial \mathcal{N}(\chi,\tau)}{ \partial \tau}=  \frac{1}{2}\frac{\partial^2 \mathcal{N}(\chi,\tau)}{\partial \chi^2} 
$$
on $[0,\infty)$ with initial condition $\mathcal{N}(\chi,0) = 0$ and Neumann boundary condition $\partial_{\chi}\mathcal{N}(\chi,\tau)\Big|_{\chi=0} = -\alpha $.

(b) The unique stationary measure of the process $\mathfrak{s}_{t}$, started from any initial condition, is the i.i.d. product of Bernoulli measures of parameter $\alpha$. In other words, 
$$
\mathbb{P}( \mathfrak{s}_{\infty}(x_1) =1, \ldots, \mathfrak{s}_{\infty}(x_d) =1)  = \alpha^d
$$
for every $(x_1,\ldots,x_d)$.

\end{theorem}
\begin{proof}

(a) Let $S_t$ denote a continuous--time simple random walk on $\mathbb{Z}>0$ with jump rates of exponential waiting $1$, and initial condition $S_t=x$. Suppose the jump rates from $1$ to the sink site $0$ are given by $\gamma^{-1}$. Applying the duality result (Theorem \ref{bbbb}) when the dual process $\mathfrak{s}_t'$ consists of a single particle, we have
$$
\rho_t(x) = \mathbb{E}_{\mathfrak{s}}[ D(\mathfrak{s}_t, \mathfrak{s}' )] = \mathbb{E}_{\mathfrak{s}'}[D(\mathfrak{s}_{}, \mathfrak{s}'_{t} )] =  \alpha  \cdot \mathbb{P}_x\left( \inf_{0 \leq s \leq \gamma t/m} S_s \leq 0\right) .
$$
The $\gamma /m$ coefficient occurs because the jump rates of SSEP$(m/2)$ are slower than for SSEP by a factor of $\gamma/m$. 

In the $L\rightarrow\infty$ limit, we can approximate $S_t$ by a Brownian motion, even with the different jump rates across $1$ to $0$; see \cite{SlowBond}. So by the reflection principle, in the limit we obtain
$$
\frac{\alpha}{\sqrt{ 2\tau \pi}} \int_{\chi}^{\infty} e^{-x^2/(2\tau)}dx = \frac{\alpha }{\sqrt{ \pi}} \int_{\chi/\sqrt{2 \tau}}^{\infty} e^{-t^2}dt = \alpha \cdot \mathrm{erfc}\left( \frac{\chi}{\sqrt{2\tau}}\right).
$$

Since 
$$
N_x(\mathfrak{s}_t) = \sum_{y \geq x} ( 1_{\{\mathfrak{s}_t(y)=1\}}  +2 \cdot 1_{\{\mathfrak{s}_t(y)=2\}} + \ldots + m \cdot 1_{\{\mathfrak{s}_t(y)=m\}} ),
$$
the expression for $\mathcal{N}(\chi,\tau)$ can be found by integrating $\mathrm{erfc}$. By direction computation, it can be checked that $\mathcal{N}(\chi,\tau)$ solves the heat equation with the specified initial and boundary conditions.

(b) Let $\mathfrak{s}'$ denote the particle configuration with particles located exactly at $x_1,\ldots,x_d$. By the duality result,
$$
\mathbb{E}_{\mathfrak{s}}[D(\mathfrak{s}_{\infty},\mathfrak{s}')] = \mathbb{E}_{\mathfrak{s}'}[D(\mathfrak{s}_{}, \mathfrak{s}'_{\infty} )].
$$
For the SSEP in one dimension, it is a classical result that with probability $1$, $\mathfrak{s}_t'(0)=d$ for sufficiently large $t$ (in other words, all $d$ particles have entered lattice site $0$). Therefore, 
$$
\mathbb{E}_{\mathfrak{s}}[D(\mathfrak{s}_{\infty},\mathfrak{s}')] = \mathbb{P}( \mathfrak{s}_{\infty}(x_1) =1, \ldots, \mathfrak{s}_{\infty}(x_d) =1)  = \alpha^d.
$$
\end{proof}

\begin{figure}
\caption{The left image shows $\mathcal{N}_{0.5}(\tau,\chi)$ for $\tau=0.1,0.2,0.3,\ldots,0.9$. The right image shows $\mathcal{N}_{\alpha}(0.5,\chi)$ for $\alpha=0.2,0.5,0.9$. In both images, the variable $\chi$ is plotted on the $x$--axis.}
\begin{center}
\includegraphics[height=5cm]{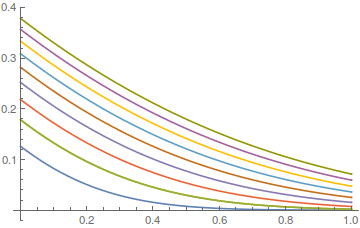}\includegraphics[height=5cm]{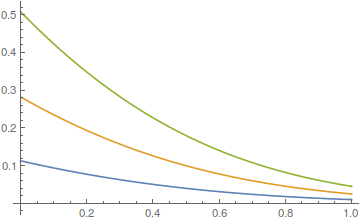}
\end{center}
\end{figure}

We can also consider the SSEP($\vec{m}$) on $\{1,\ldots,L\}$, with boundary sites at $\{0\}$ and $\{L+1\}$. If $p$ is chosen so that $p(x,x+1)=p(x+1,x)=1/2$ for $1 \leq x \leq L-1$, and $p(0,1)=p(1,0)=l, p(L,L+1)=p(L+1,L)=r$, then we obtain an open SEP$(m/2)$ on the finite lattice. Note that $l$ is not require to equal $r$, so the boundary conditions here generalize those in \cite{GKRV}. Also note that when all $m_x$ equal $1$, and the process reduces to SSEP, we obtain all possible boundary conditions. 

We conclude with a multi--species version of Proposition \ref{symm}.
\begin{prop}\label{symmM}
Suppose that $\mathcal{G}=\{0,1,2,\ldots\}$, with  $m_0=M$ and $m_x=1$ for $x\neq 0$. Also suppose that $\mathfrak{s}(0) = (\alpha^{(1)}M, \ldots, \alpha^{(n)}M)$. Then as $M\rightarrow \infty$,
$$
[\Lambda D_{\mathrm{Spi}}^{(n)}\Phi](\mathfrak{s},\mathfrak{s}') = \mathrm{const} \cdot (\alpha_1 + \ldots + \alpha_n)^{l_0^{(1)}} (\alpha_2 + \ldots + \alpha_n)^{l_0^{(2)}} \cdots (\alpha_n)^{l_0^{(n)}}\prod_{x > 0} 1_{\{\mathfrak{s}(x) \geq \mathfrak{s}'(x) \}},
$$
where $\mathfrak{s}'(0) = (l_0^{(1)},\ldots,l_0^{(n)})$.

\end{prop}
\begin{proof}
We only need to find the value of $\Lambda_0(\mathfrak{s}(0),\mathfrak{s}'(0))$. Set $k_0^{(j)} = \alpha^{(j)}M$. Given a configuration of the $k_0^{(j)}$ particles of species $j$ ($1 \leq j \leq n$) on the lattice $(-M+1,\ldots,-1,0)$, we need to count the number of possible configurations of the $l_0^{(j)}$ particles. The $l_0^{(1)}$ particles of species $1$ can be placed in any of the $k_0^{(1)} + \ldots +  k_0^{(n)}$ locations. Suppose that $o_2$ of those locations contains particles of species $\geq 2$ in the ``$k$'' configuration. Then there are $k_0^{(2)} + \ldots + k_0^{(n)} - o_2$ locations to place the $l_0^{(2)}$ particles of species $2$. Arguing analogously for higher species particles, we have a lower bound of
$$
\binom{k^{(1)}_0 + \ldots + k^{(n)}_0}{l^{(1)}_0} \binom{k^{(2)}_0 + \ldots + k^{(n)}_0-o_2 }{l^{(2)}_0} \binom{k^{(3)}_0 + \ldots + k^{(n)}_0-o_3 }{l^{(3)}_0} \cdots \binom{k^{(n)}_0-o_n}{l^{(n)}_0}, 
$$
where $o_j \leq l_0^{(1)} + \ldots + l_0^{(j-1)}$; and an upper bound of
$$
\binom{k^{(1)}_0 + \ldots + k^{(n)}_0 }{l^{(1)}_0} \binom{k^{(2)}_0 + \ldots + k^{(n)}_0 }{l^{(2)}_0 } \binom{k^{(3)}_0 + \ldots + k^{(n)}_0 }{l^{(3)}_0} \cdots \binom{k^{(n)}_0}{l^{(n)}_0}  .
$$
Since all $l_0^{(j)}$ and $o_j$ are finite, in the $M\rightarrow \infty$ limit we have
$$
(\alpha_1 + \ldots + \alpha_n)^{l_0^{(1)}} (\alpha_2 + \ldots + \alpha_n)^{l_0^{(2)}} \cdots (\alpha_n)^{l_0^{(n)}}.
$$

The denominator $\binom{M}{k_0^{(1)},\ldots,k_0^{(n)}}$ becomes a constant under the dynamics when $M\rightarrow\infty$, completing the proof.
\end{proof}

However, the function of the previous proposition is not a duality function for the multi--species SSEP with open boundary conditions.

\subsection{ASEP with open boundary conditions}
Having the SSEP with open boundary conditions as motivation, we revisit the ASEP with open boundary conditions. Suppose that $\ms$ and $\ms'$ are functions on $\{-1,-2,-3,\ldots\}$ such that $\ms(x),\ms'(x) \in \{0,1\}$ for $x < 0$. In this case, the Sch\"{u}tz duality function is 
$$
D_{\mathrm{Sch}}(\mathfrak{s}, \ms')=\prod_{x \leq -1} 1_{\{\mathfrak{s}(x) \geq \ms'(x) \}} q^{\ms'(x)(-N_{x}(\mathfrak{s})-x)},
$$
and similarly
$$
D_{\text{Kua}}(\mathfrak{s}, \ms')=\prod_{x \leq -1} 1_{\{1-\mathfrak{s}(x) \geq \ms'(x) \}} q^{-\ms'(x)N_{x}(\mathfrak{s})}.
$$
Defining $\Pi$ to be the charge reversal permutation $$\Pi(\ms,\ms')= 1_{\{\ms(x)= 1 -\ms'(x) \text{ for all } x\}},$$ 
we have that (see also Section 3.6 of \cite{KIMRN} )
$$
D_{\text{Kua}} = \Pi D_{\mathrm{Sch}}. 
$$
The duality function  in Corollary \ref{newdual}(c) is now
$$
[D_{\mathrm{Sch}}\mathcal{P}^{-1}](\mathfrak{s}, \ms')=\prod_{x \leq -1} 1_{\{\mathfrak{s}(x) \geq \ms'(x) \}} q^{-\ms'(x)N_{x}(\mathfrak{s})}.
$$
Note that  \begin{equation}\label{One} 1_{\{\mathfrak{s}(x) \geq \ms'(x) \}} q^{\ms'(x)(-N_x(\ms)-x)}=1_{\{\mathfrak{s}(x) \geq \ms'(x) \}}  \text{ for all } \ms,\ms' \text{ and for } x=-1. \end{equation}

Let $L_{l,r}$ denote the generator of the process with left jump rates $l$ and right jump rates $r$, where particles may enter lattice site $-1$ at rate $1$. Similarly, let $L'_{l,r}$ denote the generator of the process with left jump rates $l$ and right jump rates $r$, where particles may exit lattice site $-1$ at rate $1$.

\begin{lemma} For any $l,r$,
\begin{equation}\label{A}
\Pi L'_{l,r} \Pi = L_{r,l}.
\end{equation}
Furthermore,
\begin{equation}\label{B}
\mathcal{P} L'_{q,1} = L_{1,q}' \mathcal{P}.
\end{equation}
\end{lemma}
\begin{proof}
Both identities follow from a direct calculation. Note that for an infinite line, equation \eqref{B} is similar to the identity in the proof of Theorem \ref{MSD}(b).
\end{proof}


\begin{theorem}\label{OD}
We have the duality relations

\begin{align}
L_{1,q}D_{\mathrm{Sch}} &= D_{\mathrm{Sch}}(L_{1,q}')^* \label{C}, \\
L_{q,1}' \Pi D_{\mathrm{Sch}} &= \Pi D_{\mathrm{Sch}} (L_{1,q}')^*  \label{D} , \\
L_{q,1}' \Pi D_{\mathrm{Sch}} \mathcal{P}^{-1} &=  \Pi D_{\mathrm{Sch}}  \mathcal{P}^{-1}  (L_{q,1}')^* \label{E}, \\
L_{1,q} D_{\mathrm{Sch}} \mathcal{P}^{-1} &= D_{\mathrm{Sch}}  \mathcal{P}^{-1}  (L_{q,1}')^*. \label{F}
\end{align}
\end{theorem}
\begin{proof}
By \eqref{B}, the equations \eqref{C} and \eqref{F} are equivalent. Similarly by $\eqref{B}$, the equations \eqref{D} and \eqref{E} are equivalent. By \eqref{A}, the equations \eqref{C} and \eqref{D} are equivalent; similarly, \eqref{E} and \eqref{F} are equivalent. Thus, it suffices to show one of \eqref{C}--\eqref{F}; we will show \eqref{C}.

Let $\ms'_{-}$ denote the particle configuration $\ms'$ after a particle from site $-1$ has been removed, and let $\ms_{+}$ denote the particle configuration $\ms$ after a particle from site $-1$ has been added. By an identical argument as in the proof of Theorem \ref{bbbb}, it suffices to show that
\begin{equation}\label{Need}
 L(\ms,\ms_+)[D(\ms_+,\ms') - D(\ms,\ms')]  = [D(\ms,\ms'_{-}) - D(\ms,\ms')] L'(\ms' ,\ms'_{-}),
\end{equation}
where $L,L'$ are the generators of the processes $\ms_t,\ms'_t$. We consider four cases:

\underline{Case 1: $\ms(-1)=0, \ms'(-1)=1$}. 
Then $D(\ms,\ms')=0$, so it suffices to prove
$$
 L(\ms,\ms_+)D(\ms_+,\ms')   = D(\ms,\ms'_{-}) L'(\ms' ,\ms'_{-}).
 $$
By \eqref{One} and the definition of $L,L',D$, this equation is equivalent to 
$$
 \prod_{x \leq - 2} 1_{\{\mathfrak{s}(x) \geq \ms'(0) \}} q^{N_{x}(\mathfrak{s})-x} 
 =\prod_{x \leq -2} 1_{\{\mathfrak{s}(x) \geq \ms'(0) \}} q^{N_{x}(\mathfrak{s})-x}.
$$
which can be immediately seen to be true.

\underline{Case 2: $\ms(-1)=1, \ms'(-1)=1$}. In this case, $L(\ms,\ms_+)=0$, so it suffices that 
$$
D(\ms,\ms'_{-}) = D(\ms,\ms')
$$
This is again true by \eqref{One}.

\underline{Case 3: $\ms(-1)=1, \ms'(-1)=0$}. In this case, both sides of \eqref{Need} equal $0$, because the $L$ terms equal $0$.

\underline{Case 4: $\ms(-1)=0, \ms'(-1)=0$}. In this case,
$$
D(\ms_+,\ms') - D(\ms,\ms') = D(\ms,\ms'_{-}) - D(\ms,\ms') =0.
$$

\end{proof}
 
We now explain how the duality function arises fits into the larger framework of this paper. Consider an ASEP$(q,\vec{m})$ on $\{0,-1,-2,-3,\ldots\}$, where $m_0=M$ and $m_{x} = 1$ for all $x \leq -1$. Suppose there are $\alpha M$ particles at $0$, where $\alpha \in (0,1)$. Then as $M\rightarrow \infty$, the jump rates from $0$ to $-1$ converge to 
$$
\lim_{M\rightarrow \infty} \frac{ 1- q^{\alpha M}}{1- q^M}=1,
$$
and the jump rates from $-1$ to $0$ converge to 
$$
\lim_{M\rightarrow \infty} q \cdot q^{(1-\alpha)M}\frac{ 1- q^{\alpha M}}{1- q^M} =0.
$$
Thus, one obtains the process with generator $L_{1,q}$. Similarly, the $M\rightarrow \infty$ of a space--reversed ASEP$(q,\vec{m})$ with $\alpha M$ particles at $0$ is the process with generator $L_{q,1}'$.

Note that there is no dependence on $\alpha$ in the limit. This stands in contrast to the symmetric case, where the original process and entrance rates $1$ and the dual process had entrance rates $1-\alpha$ and exit rates $\alpha$. There, the contribution of the boundary site $0$ to the duality function appears through the term $\alpha^{\mathfrak{s}'(0)}$. When $\alpha=1$, this term can simply be removed.

If one takes the $M\rightarrow\infty$ limit of the ASEP$(q,\vec{m})$ duality function in $\Lambda D_{\textrm{Sch}}\Phi$ from Proposition \ref{FusedDuality}, the result is simply $\infty$, due to the term $q^{-\ms(0)\ms'(0)}$. However, with the symmetric case as motivation, one could simply remove this term from the duality. To account for the space reversal between $L_{1,q}$ and $L_{q,1}'$, we replace $D_{\mathrm{Sch}}$ with $D_{\text{Kua}}$ and multiply by $\mathcal{P}^{-1}$; this results in the duality  \eqref{F}.

We now proceed to the hydrodynamic limit. This requires the following lemma, which is essentially summation by parts:

\begin{lemma}
Define the function
$$
H(\mathfrak{s},x) = q^{-N_x(\mathfrak{s})}.
$$
Let $\mathfrak{s}_t$ evolve as an open ASEP with generator $L_{1,q}$, and let $x(t)$ evolve as an open ASEP with generator $L'_{q,1}$. Then
$$
\mathbb{E}_{\mathbf{0}}[H(\mathfrak{s}_t,x)] = 1 + (q^{-1}-1)\mathbb{P}_{\mathbf{0}}(\mathfrak{s}_t(-1)=1) + (1-q) \sum_{z=x}^{-2} \mathbb{E}_z[D_{\text{Sch}}\mathcal{P}^{-1}(\mathbf{0},x(t))].
$$
The initial condition $\mathbf{0}$ indicates a configuration with no particles. 
\end{lemma}
\begin{proof}
Using the identity
$$
H(\ms,x) - H(\ms,x+1) = 1_{\{\ms(x)=1\}} q^{-N_x(\ms)}(1-q) = D_{\mathrm{Sch}}\mathcal{P}^{-1}(\ms,x) (1-q)
$$
leads to a telescoping sum
$$
H(\ms,x) = H(\ms,-1) + (1-q) \sum_{z=x}^{-2} D_{\text{Sch}}\mathcal{P}^{-1}(\ms,z).
$$
Thus, by the duality \eqref{F},
$$
\mathbb{E}_{\mathbf{0}}[H(\ms_t,x)] = \mathbb{E}_{\mathbf{0}}[ H(\ms_t,-1)  ] + (1-q) \sum_{z=x}^{-2} \mathbb{E}_z[D_{\text{Sch}}\mathcal{P}^{-1}(\mathbf{0},x(t))].
$$
And finally,
\begin{align*}
\mathbb{E}_{\mathbf{0}}[ H(\ms_t,-1)  ]  &= q^{-1} \mathbb{P}_{\mathbf{0}}(\ms_t(-1)=1) + 1 \cdot \mathbb{P}_{\mathbf{0}}(\ms_t(-1)=0) \\ &= 1 + (q^{-1}-1)\mathbb{P}_{\mathbf{0}}(\mathfrak{s}_t(-1)=1).
\end{align*}
\end{proof}

\begin{theorem}\label{Addend}
In the hydrodynamic limit, setting $\sigma^2=3q-q^2$ and $c=1-q$,
$$
\mathfrak{h}(\zeta,\tau):= \lim_{L\rightarrow \infty} L^{-1} \mathbb{E}_{\mathbf{0}}[ q^{-N_{\zeta L}(\ms_{\tau L})}] = (1-q) \int_0^{\vert \zeta \vert} \int_0^{\tau} \frac{\xi}{\sigma\sqrt{2 \pi t^3}} e^{ - \frac{(\xi - ct)^2}{2\sigma^2 t} } dt d\xi,
$$
which solves the heat equation with convection
$$
\mathfrak{h}_t = \frac{\sigma^2}{2} \mathfrak{h}_{xx} - c \mathfrak{h}_x
$$
with Neumann boundary conditions $\partial_{\zeta} \mathfrak{h}(0,\tau)=-1$ and initial conditions $\mathfrak{h}(\zeta,0)=0$.
\end{theorem}
\begin{proof}
Note that
$$
\mathbb{E}_z[D_{\text{Sch}}\mathcal{P}^{-1}(\mathbf{0},x(t))] = \mathbb{P}_0 \left( \sup_{0 \leq s \leq t} S_s \geq \vert z \vert \right),
$$
where $S_t$ is a continuous--time random walk with right jump rates $1$ and left jump rates $q$. In the hydrodynamic limit $t = \tau L$ and $\vert z \vert= \xi L$, we have
$$
\mathbb{P}_0 \left( \sup_{0 \leq s \leq t} S_s \geq \vert z \vert \right) \rightarrow \mathbb{P}_0\left( \sup_{[0,\tau]} W_t \geq \xi \right),
$$
where $W_t = \sigma B_t + ct$ for $c=1-q$ and $\sigma^2 =  1 + q - (1-q)^2 = 3q-q^2 $. Now recall the fact that
$$
\mathbb{P}_0 \left(\sup_{[0,\tau]} W_t \geq \xi \right) =  \int_0^{\tau} \frac{  \xi}{\sigma \sqrt{2 \pi t^3}} e^{ - \frac{(\xi - (1-q)t)^2}{2\sigma^2 t} } dt 
$$
In applying the previous lemma, the summation becomes an integral, thus showing the double integral form of $\mathfrak{h}(\zeta,\tau)$. It can be checked directly that it satisfies the heat equation with convection.

To show the Neumann boundary conditions, it remains to show that
$$
\lim_{\zeta \rightarrow 0} \int_0^{\tau} \frac{\zeta}{\sigma\sqrt{2 \pi t^3}} e^{ - \frac{(\zeta - (1-q)t)^2}{2\sigma^2t} } dt  =1.
$$
Note that due to the singularity at $t=0$, the limit and the integral cannot be interchanged. However, since the integrand is the probability distribution function of a random variable, we can write
$$
\lim_{\zeta \rightarrow 0} \int_0^{\tau} \frac{\zeta}{\sigma\sqrt{2 \pi t^3}} e^{ - \frac{(\zeta - (1-q)t)^2}{2\sigma^2t} } dt  = 1- \lim_{\zeta \rightarrow 0} \int_{\tau}^{\infty} \frac{\zeta}{\sigma\sqrt{2 \pi t^3}} e^{ - \frac{(\zeta - (1-q)t)^2}{2\sigma^2 t} } dt .
$$
And now the integral and limit can be switched, to obtain $1-0=1$.

The initial conditions are immediate to verify. 
\end{proof}

\begin{reemark}
The heat equation (with or without convection) can be transformed into the Burger's equation via the Hopf--Cole transformation. See, for example, section 3 of \cite{SB02} for the full calculations. The transformation $N_x(\ms) \mapsto q^{-N_x(\ms)}$ is often called a discrete Hopf--Cole transformation.
\end{reemark}

\begin{reemark}
When $q=1$, the function $\mathfrak{h}$ becomes identically zero. This is due to the different scaling regimes of SSEP and ASEP in Theorems \ref{Opened} and \ref{Addend}.
\end{reemark}

\begin{reemark}
We mention a few other recent results concerning open ASEP. The paper \cite{barraquand2018} finds GOE fluctuations of the open ASEP, but with different boundary conditions; there, particles enter at rate $1/2$ and exit at rate $q/2$. In \cite{CSOpen}, the authors consider open ASEP in the \textit{weakly} asymmetric scaling regime and show convergence to the stochastic Heat equation with Robin boundary conditions.
\end{reemark}

\begin{figure}
\begin{center}
\includegraphics[height=3cm]{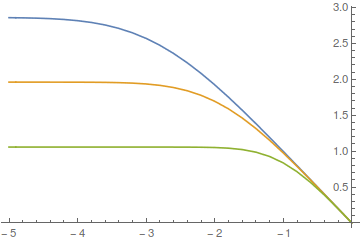}
\end{center}
\caption{For $q=0.1$, the image shows $\mathfrak{h}(\zeta,\tau)$ for $\tau=1,2,3$.}
\end{figure}

\section{Dynamic ASEP$(q,{m})$}

\subsection{Some new $q$--identities}
\begin{lemma}\label{Prev}
\begin{align*}
\binom{m}{k}_q(\alpha + q^{m+k}) + q^{m-k+1}\binom{m}{k-1}_q(\alpha+q^{k-1}) &= \binom{m+1}{k}_q(\alpha + q^m),\\
\binom{m-1}{k-1}_q(\alpha + q^{m+k-1}) + q^k\binom{m-1}{k}_q(\alpha + q^{k-1}) &= \binom{m}{k}_q(\alpha + q^{2k-1}).
\end{align*}
\end{lemma}
\begin{proof}
In the first identity, the $\alpha$ coefficient is
$$
\binom{m}{k}_q + q^{m-k+1}\binom{m}{k-1}_q =  \binom{m+1}{k}_q
$$
and the constant coefficient is
$$
q^{m+k}\binom{m}{k}_q  + q^m\binom{m}{k-1}_q = q^m \binom{m+1}{k}_q.
$$
Both of these follow from \eqref{qBin}. The proof of the second identity is identical.
\end{proof}

\begin{lemma}\label{Prev1}
For $m\geq k$,
$$
\sum_{y=k-1}^{m-1} \dfrac{q^y \binom{y}{k-1}_q}{(\alpha +q ^y) \cdots (\alpha + q^{y+k})} = \frac{q^{k-1}\binom{m}{k}_q}{(\alpha+q^{k-1})(\alpha+q^m)\cdots (\alpha + q^{m+k-1})}.
$$
\end{lemma}
\begin{proof}
Proceed by induction $m$. When $m=k$, the identity is clear.

Assuming the lemma holds for $m$, then Lemma \ref{Prev} implies 
\begin{multline*}
 \frac{q^{k-1}\binom{m}{k}_q}{(\alpha+q^{k-1})(\alpha+q^m)\cdots (\alpha + q^{m+k-1})} + \dfrac{q^m \binom{m}{k-1}_q}{(\alpha +q ^m) \cdots (\alpha + q^{m+k})} \\
 =  \frac{q^{k-1}}{(\alpha+q^{k-1})(\alpha+q^m)\cdots (\alpha + q^{m+k})} \left( \binom{m}{k}_q(\alpha + q^{m+k}) + q^{m-k+1}\binom{m}{k-1}_q(\alpha+q^{k-1}) \right) \\
 = \frac{q^{k-1}\binom{m+1}{k}_q}{(\alpha+q^{k-1})(\alpha+q^{m+1})\cdots (\alpha + q^{m+k})}.
\end{multline*}
as needed.
\end{proof}

\begin{lemma}\label{Prev2}
For $n \geq 0,$
$$
\sum _ { r = 0 } ^ { n - 1 } \frac {  { q ^r} } { \left( \alpha + q ^ { r } \right) \left( \alpha + q ^ { r + 1 } \right) } = \frac { 1 - q ^ { n } } { ( 1 - q ) ( \alpha + 1 ) \left( \alpha + q ^ { n } \right) }.
$$
\end{lemma}
\begin{proof}
This is a straightforward induction on $n$.
\end{proof}

\begin{theorem}\label{qaBin}
\begin{multline*}
\sum_{0 \leq x_k < \cdots < x_1 \leq m-1} \frac { q ^ { x _ { 1 } } } { \left( \alpha + q ^ { x _ { 1 } } \right) \left( \alpha + q ^ { x _ { 1 } + 1 } \right) } \cdots  \frac { q ^ { x _ { k } } } { \left( \alpha + q ^ { x _ { k } + 2k-2 } \right) \left( \alpha + q ^ { x _ { k } + 2k-1 } \right) }  \\
= \frac{q ^ { ( k - 1 ) k / 2 } \binom{m}{k}_ { q }}{\left( \alpha + q ^ { k - 1 } \right) \dots \left( \alpha + q ^ { 2 k - 2 } \right) \left( \alpha + q ^ { m } \right) \cdots \left( \alpha + q ^ { m + k-1 } \right)}.
\end{multline*}
\end{theorem}
\begin{proof}
Proceed by induction on $k$. The case when $k=1$ is exactly Lemma \ref{Prev2}.

Now assume the theorem holds for some value of $k-1$. Then by Lemma \ref{Prev1} and the induction hypothesis,
\begin{align*}
&\sum _ { x _ { 1 } = k - 1 } ^ { m - 1 } \frac { q ^ { x _ { 1 } } } { \left( \alpha + q ^ { x _ { 1 } } \right) \left( \alpha + q ^ { x_1 + 1 } \right) } \sum_{0 \leq x_k < \cdots < x_2 \leq x_1-1} \frac{q^{x_2}}{(\alpha + q^{x_2+2})(\alpha + q^{x_2+3})} \cdots \\
& \quad \quad = \sum _ { x _ { 1 } = k - 1 } ^ { m - 1 } \frac { q ^ { x _ { 1 } } } { \left( \alpha + q ^ { x _ { 1  } } \right) \left( \alpha + q ^ { x_1 + 1 } \right) } \cdot  \frac{q ^ { ( k - 2 ) (k-1) / 2 } \binom{x_1}{k-1}_ { q }}{\left( \alpha + q ^ { k  } \right) (\alpha + q^{k+1})\dots \left( \alpha + q ^ { 2 k - 2 } \right) \left( \alpha + q ^ { x_1+2 } \right) \cdots \left( \alpha + q ^ { x_1 + k } \right)}\\
& \quad \quad = \frac { q ^ { ( k - 2 ) ( k - 1 ) / 2 } } { \left( \alpha + q ^ { k } \right) \left( \alpha + q ^ { k + 1 } \right) \cdot \cdots \left( \alpha + q ^ { 2 k - 2 } \right) } \cdot  \frac{q^{k-1}\binom{m}{k}_q}{(\alpha+ q^{k-1})(\alpha + q^m)\cdots (\alpha + q^{m+k-1})} \\
& \quad \quad  = \frac{q ^ { ( k - 1 ) k / 2 } \binom{m}{k}_ { q }}{\left( \alpha + q ^ { k - 1 } \right) \dots \left( \alpha + q ^ { 2 k - 2 } \right) \left( \alpha + q ^ { m } \right) \cdots \left( \alpha + q ^ { m + k-1 } \right)},
\end{align*}
as needed.

\end{proof}

\subsection{Stationary measures of dynamic ASEP}

We now define the dynamic ASEP on a finite interval $\{0,\ldots,m\}$ with \textit{closed} boundary conditions. The state space is the set of functions $s:\{0,\ldots,m\} \rightarrow \mathbb{Z}$ such that $s(x+1) = s(x) \pm 1$ for all $0 \leq x \leq m$. The jump rates are the same as in the infinite lattice case, but we also require that $s(0)$ and $s(m)$ be fixed, unchanging values. 

For any $h\in \mathbb{Z}$, define the probability measure $\mathbb{P}_{q,\alpha,h,m}$ on the state space by its marginals
\begin{align*}
\mathbb{P}_{q,\alpha,h,m}(s(m)=h) &= 1 \\
\mathbb{P}_{q,\alpha,h,m}(s(x-1)=s(x) + 1) &= \frac{q^{s(x)}}{\alpha + q^{s(x)}} \text{ for all } 1 \leq x \leq m,\\
\mathbb{P}_{q,\alpha,h,m}(s(x-1)=s(x) - 1) &= \frac{\alpha}{\alpha + q^{s(x)}} \text{ for all } 1 \leq x \leq m.
\end{align*}
Letting $\mathbf{1}$ denote the constant function $\mathbf{1}(x)=1$ for all $x$, there is the shift
\begin{equation}\label{Shift}
\mathbb{P}_{q,\alpha,h,m}(s) = \mathbb{P}_{q,\alpha q^{-h},0,m}(s - h\mathbf{1}).
\end{equation}

Let $k$ be the integer satisfying $s_m+k - (m-k)=s_0$. Given $s(x)$, let $X_s\subseteq \{0,\ldots,m-1\}$ denote the set of values $x$ such that $s_{x+1}=s_x-1$. Order the elements of $X_s$ by $0 \leq x_k < \ldots < x_1 \leq m-1$. 

\begin{center}
\begin{tikzpicture}
\draw (-0.3,2) node {$s_0=2$};
\draw (1.5,3.2) node {$s_1=3$};
\draw (3.5,0.8) node {$s_3=1$};
\draw (5.5,3.2) node {$s_5=3$};
\draw (6.5,1.8) node {$s_6=2$};
\draw (8.3,3) node {$s_7=3$};
\draw (2,0.3) node {$x_3$};
\draw (3,0.3) node {$x_2$};
\draw (6,0.3) node {$x_1$};
\draw [very thick](0,0) -- (8,0);
\draw (1,0) circle (3pt);
\draw (4,0) circle (3pt);
\draw (5,0) circle (3pt);
\draw (7,0) circle (3pt);
\fill[black] (2,0) circle (3pt);
\fill[black] (3,0) circle (3pt);
\fill[black] (6,0) circle (3pt);
\draw [thick](0.5,2) -- (1.5,3) -- (3.5,1) -- (5.5,3) -- (6.5,2) -- (7.5,3);
\end{tikzpicture}
\end{center}

The pushforward of $\mathbb{P}_{q,\alpha,h,m}$ under the map $s\mapsto X_s$ is a probability measure on $2^{\{0,\ldots,m-1\}}$. By \eqref{Shift}, $\mathbb{P}_{q,\alpha,h,m}$ and $\mathbb{P}_{q,\alpha q^{-h},0,m}$ are pushed forward to the same measure. By a slight abuse of notation, let $\mathbb{P}_{q,\alpha,m}$ denote the probability measure on $2^{\{0,\ldots,m-1\}}$ which is the pushforward of $\mathbb{P}_{q,\alpha,0,m}$. The number of indices (whether three or four) will indicate which probability measure is being referenced. 

\begin{prop}\label{Factors}
Fix $q,\alpha$ and $m$. 

(a) For any fixed $0 \leq k \leq m$ and any sequence $0 \leq x_k < \ldots < x_1 \leq m-1$, the measure $\mathbb{P}_{q,\alpha,m}$ on $2^{\{0,\ldots,m-1\}}$ satisfies
$$
\mathbb{P}_{q,\alpha,m}(X = \{x_k,\ldots,x_1\} ) =  
(-\alpha^{-1}q;q^{-1})_{2k-m} \prod_{l=1}^k \frac{q^{x_l+2(l-1)} \cdot q^m \alpha }{(\alpha q^m + q^{x_l+2(l-1)})(\alpha q^m + q^{x_l+2l-1})} 
$$
and
$$
\mathbb{P}_{q,\alpha,m}(\vert X\vert=k) = \frac{\alpha}{1 + \alpha} \cdot \frac{q \alpha}{1 + q\alpha} \cdot \cdots \cdot  \frac{ q^{ m-2k-1} \alpha}{1 + q^{ m-2k-1} \alpha } \cdot Z^{(\alpha)}_{k,m},
$$
where
$$
Z^{(\alpha)}_{k,m} =  \frac{ \binom{m}{k}_ { q }}{ \left( 1+ \alpha q^{m-k+1} \right) \cdots \left( 1 + \alpha q^{ m - 2 k + 2 } \right) \left(1+ \alpha^{-1}  \right) \cdots \left( 1 + \alpha^{-1}q ^ { k-1 } \right)}.
$$
Here, $(a;q)_r$ is the $q$--Pochhamer symbol defined by
$$
(a;q)_r 
= \displaystyle
\begin{cases}
\prod_{j=0}^{r-1}(1-aq^{j}),  \quad r >0\\[10pt]
\prod_{j=1}^{-r} (1-a q^j)^{-1}, \quad r<0.
\end{cases}
$$

(b) The normalizing factors $Z^{(\alpha)}_{k,m}$ satisfy

$$
\frac{Z^{(\alpha)}_{k-l,m-l}}{Z^{(\alpha)}_{k,m}} = \frac{ \binom{m-l}{k-l}_q }{ \binom{m}{k}_q } (1 + \alpha q^{m-2(k-1)}) \cdots ( 1 + \alpha q^{m-2(k-l)}) (1 + \alpha^{-1}q^{k-1}) \cdots (1 + \alpha^{-1}q^{k-l}) 
$$

and
$$
\frac{Z^{(\alpha q^{-l})}_{k-l,m-l}}{Z^{(\alpha)}_{k,m}} = \frac{ \binom{m-l}{k-l}_q}{\binom{m}{k}_q}(1+\alpha^{-1})\cdots (1+\alpha^{-1} q^{l-1}) ( 1 + \alpha q^{m-k+1}) \cdots ( 1 + \alpha q^{m-k-l+2}) .
$$

(c) Additionally,
\begin{align*}
\mathbb{P}_{q,\alpha,m}(x_k=0 \big| \vert X\vert=k) = \frac{\binom{m-1}{k-1}_q(\alpha q^m + q^{m+k-1})}{\binom{m}{k}_q(\alpha q^m + q^{2k-1})}, \quad & \mathbb{P}_{q,\alpha,m}(x_k>0 \big| \vert X\vert=k)  = \frac{\binom{m-1}{k}_q q^k(\alpha q^m + q^{k-1})}{\binom{m}{k}_q(\alpha q^m + q^{2k-1})},\\
\mathbb{P}_{q,\alpha,m}(x_1=m-1 \big| \vert X\vert=k) = \frac{q^{m-k} \binom{m-1}{k-1}_q}{\binom{m}{k}_q} \frac{\alpha q^m + q^{k-1}}{\alpha q^m + q^{m-1}}, & \quad \mathbb{P}_{q,\alpha,m}(x_1< m-1 \big| \vert X\vert=k) = \frac{\binom{m-1}{k}_q}{\binom{m}{k}_q} \frac{\alpha q^m + q^{m+k-1}}{\alpha q^m+ q^{m-1}}.
\end{align*}
\end{prop}
\begin{proof}
(a) We consider only when $2k-m<0$; if $2k-m\geq0$ the proof is similar. By looking from right to left, we see that $\mathbb{P}_{q,\alpha,m}(X = \{x_k,\ldots,x_1\} )$ equals
\begin{align*}
& \frac{\alpha}{\alpha + 1} \cdot \frac{\alpha}{\alpha + q^{-1}} \cdot \cdots \cdot  \frac{\alpha}{\alpha + q^{x_1-m+1}}  \times \frac{q^{x_1-m}}{\alpha + q^{x_1-m}} \\
& \frac{\alpha}{\alpha + q^{x_1-m+1}} \cdot  \frac{\alpha}{\alpha + q^{x_1-m}} \cdot \cdots \cdot \frac{\alpha}{\alpha + q^{x_2-m+3}} \times \frac{q^{x_2-m+2}}{\alpha + q^{x_2-m+2}}\\
& \cdots \\
& \frac{\alpha}{\alpha + q^{x_k-m+2k-1}} \frac{\alpha}{\alpha + q^{x_k-m+2k-2}} \cdot \cdots \cdot \frac{\alpha}{\alpha + q^{k-(m-k)+1}},   
\end{align*}
which can be re--arranged as 
\begin{multline*}
\frac{\alpha}{\alpha +1} \cdot \frac{\alpha}{\alpha+q^{-1}} \cdot \cdots \cdot  \frac{\alpha}{\alpha + q^{k-(m-k)+1}}\\
\times \frac{q^{x_1} \cdot q^m \alpha}{(\alpha q^m + q^{x_1})(\alpha q^m + q^{x_1+1})} \cdots \frac{q^{x_k+2(k-1)} \cdot q^m\alpha }{(\alpha q^m + q^{x_k+2(k-1)})(\alpha q^m + q^{x_k+2k-1})} .
\end{multline*}
This shows the first part. By Theorem \ref{qaBin},
\begin{multline*}
 \sum_{0 \leq x_k < \ldots < x_1 \leq m-1}      \frac{q^{x_1}}{(\alpha q^m + q^{x_1})(\alpha q^m + q^{x_1+1})} \cdots \frac{q^{x_k}}{(\alpha q^m + q^{x_k+2(k-1)})(\alpha q^m + q^{x_k+2k-1})}            \\
= \frac{q ^ { ( k - 1 ) k / 2 } \binom{m}{k}_ { q }}{\left( \alpha q^m + q ^ { k - 1 } \right) \dots \left( \alpha q^m+ q ^ { 2 k - 2 } \right) \left( \alpha q^m+ q ^ { m } \right) \cdots \left( \alpha q^m+ q ^ { m + k-1 } \right)},
 \end{multline*}
which leads to the second part. 

(b) By part (a), it is immediate that
$$
\frac{Z^{(\alpha)}_{k-1,m-1}}{Z^{(\alpha)}_{k,m}} = \frac{ \binom{m-1}{k-1}_q }{ \binom{m}{k}_q } (1 + \alpha q^{m-2k+2})(1 + \alpha^{-1}q^{k-1}),
$$
By induction, the first statement of (b) follows. The proof of the second statement is similar.

(c) By part (a), 
\begin{align*}
\mathbb{P}_{q,\alpha,m}(x_k=0 \big| \vert X\vert=k) &= Z_{k,m}^{-1} \cdot \frac{q^{2k-2} q^m \alpha }{(\alpha q^m + q^{2k-2})(\alpha q^m + q^{2k-1})} \\
\times  & \sum_{1 \leq x_{k-1} < \cdots < x_1 \leq m-1}  \frac{q^{x_1} \cdot q^m \alpha }{(\alpha q^m + q^{x_1})(\alpha q^m + q^{x_1+1})} \cdots \frac{q^{x_{k-1}+2(k-2)}\cdot q^m \alpha}{(\alpha q^m + q^{x_{k-1}+2(k-2)})(\alpha q^m + q^{x_{k-1}+2k-3})} \\
&= Z_{k,m}^{-1} \cdot \frac{q^{2k-2}q^m\alpha}{(\alpha q^m + q^{2k-2})(\alpha q^m + q^{2k-1})} \\
\times & \sum_{0 \leq x_{k-1} < \cdots < x_1 \leq m-2}  \frac{q^{x_1+1} \cdot q^m \alpha }{(\alpha q^m + q^{x_1+1})(\alpha q^m + q^{x_1+2})} \cdots \frac{q^{x_{k-1}+2(k-2)+1} \cdot q^m \alpha}{(\alpha q^m + q^{x_{k-1}+2k-3})(\alpha q^m + q^{x_{k-1}+2k-2})}.
\end{align*}
By (a), this equals
$$
\frac{Z_{k-1,m-1}}{Z_{k,m}} \cdot \frac{q^{2k-2}q^m\alpha}{(\alpha q^m + q^{2k-2})(\alpha q^m + q^{2k-1})} \sum_{0 \leq x_{k-1} < \cdots < x_1 \leq m-2} \mathbb{P}_{q,\alpha,m-1}( X = \{x_{k-1},\ldots,x_1\} \big| \vert X\vert = k-1),
$$
Therefore,
\begin{align*}
\mathbb{P}_{q,\alpha,m}(x_k=0 \big| \vert X\vert=k)  = \frac{Z_{k-1,m-1}}{Z_{k,m}} \cdot  \frac{q^{2k-2}q^m\alpha}{(\alpha q^m + q^{2k-2})(\alpha q^m + q^{2k-1})} .
\end{align*}
The $Z$ terms simplify from part (b). Multiplying the numerator and denominator by $q^{2k-2}$ shows that most of the terms will cancel. What remains will yield
$$
\mathbb{P}_{q,\alpha,m}(x_k=0 \big| \vert X\vert=k)  = \frac{ \binom{m-1}{k-1}_q }{ \binom{m}{k}_q } \cdot \frac{ \alpha q^m  + q^{m+k-1}}{\alpha q^m + q^{2k-1}}  ,
$$
which equals the desired expression. Applying the second equality in Lemma \ref{Prev} to 
$$
\mathbb{P}_{q,\alpha,m}(x_k>0 \big| \vert X\vert=k) = 1 - \mathbb{P}_{q,\alpha,m}(x_k=0 \big| \vert X\vert=k) 
$$
yields the second equality of part (c). The proof for the third and fourth equalities are similar.

\end{proof}

\subsection{Stochastic fusion construction of dynamic ASEP$(q,\vec{m})$}
Fix positive integers $m(j)$. The state space of dynamic ASEP$(q,\vec{m})$ consists of maps $s:\mathbb{Z} \rightarrow \mathbb{Z}$ such that for all $x\in \mathbb{Z}$,
$$
\vert s(x+1)-s(x) \vert \leq m(x) \text{  and  } s(x+1)-s(x) \in m(x) + 2\mathbb{Z}.
$$
Given $s(x)$, let $k_s(x)$ be the solution to $s(x) - k_s(x) + (m({x})-k_s(x))=s(x+1)$. See the example in Figure \ref{bpp}.

\begin{figure}\label{bpp}
\caption{Here, $(m(0),m(1),\ldots,m(6)) = (1,3,1,2,2,3,1)$ and $(k(0),\ldots,k(6))=(0,2,1,1,0,3,0)$.}
\begin{center}
\begin{tikzpicture}
\draw (-0.3,2) node {$s(0)=2$};
\draw (1.5,3.2) node {$s(1)=3$};
\draw (3.5,0.8) node {$s(3)=1$};
\draw (5.5,3.2) node {$s(5)=3$};
\draw (6.5,-0.2) node {$s(6)=0$};
\draw (8.3,1) node {$s(7)=1$};
\draw [very thick](0,-3) -- (8,-3);
\draw (1,-3) circle (3pt);
\fill[black] (2,-3) circle (3pt);
\fill[black] (2,-2.5) circle (3pt);
\draw (2,-2) circle (3pt);
\fill[black] (3,-3) circle (3pt);
\fill[black] (4,-3) circle (3pt);
\draw (4,-2.5) circle (3pt);
\draw (5,-3) circle (3pt);
\draw (5,-2.5) circle (3pt);
\fill[black] (6,-3) circle (3pt);
\fill[black] (6,-2.5) circle (3pt);
\fill[black] (6,-2) circle (3pt);
\draw (7,-3) circle (3pt);
\draw [thick](0.5,2) -- (1.5,3) -- (2.5,2) -- (3.5,1) -- (4.5,1) -- (5.5,3) -- (6.5,0) -- (7.5,1);
\end{tikzpicture}
\end{center}
\end{figure}
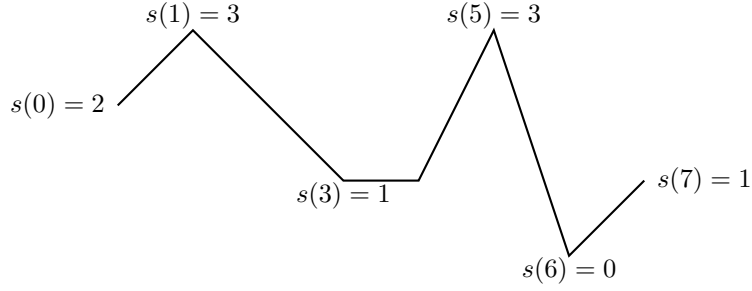

The jump rates will depend on the dynamic parameter $\alpha$. If $s,s'$ are two states such that $s(y)=s'(y)$ for all $y\neq x$, the jump rates from $s$ to $s'$ are
\begin{multline*}
=\frac{1 + {\alpha} q^{-s(x)} }{1 + {\alpha} q^{-s(x \pm 1)}  } \cdot  \mathbb{P}_{q,\alpha q^{-s(x)},m(x-1)}(  y_1 = m(x-1)-1 | \vert Y\vert = k_s(x-1) )  \cdot \mathbb{P}_{q,\alpha q^{-s(x+1)},m(x)}(y_{k_s(x)}>0 | \vert Y \vert = k_s(x)) \\
=\frac{1 + {\alpha} q^{-s(x)} }{1 + {\alpha} q^{-s(x \pm 1)}  }\left[ \frac{q^{m(x-1)-k_s(x-1)} \binom{m(x-1)-1}{k_s(x-1)-1}_q}{ \binom{m(x-1)}{k_s(x-1)}_q}  \cdot \frac{\alpha q^{m(x-1)-s(x)} + q^{k_s(x-1)-1}}{\alpha q^{m(x-1)-s(x)} + q^{m(x-1)-1}}\right] \left[ \frac{\binom{m(x)-1}{k_s(x)}_q}{ \binom{m(x)}{k_s(x)}_q} \frac{q^{k_s(x)} (\alpha q^{m(x)-s(x+1)} + q^{k_s(x)-1}  ) }{ \alpha q^{m(x)-s(x+1)} + q^{2k_s(x)-1}} \right] , 
\end{multline*}
if $s'(x) = s(x) + 2$ and 
\begin{multline*}
\frac{q(1+\alpha q^{-s(x)})}{1 + \alpha q^{-s(x \pm 1)}}  \cdot  \mathbb{P}_{q,\alpha q^{-s(x)},m(x-1)}( y_1 < m(x-1)-1 | \vert Y\vert = k_s(x-1) )   \cdot \mathbb{P}_{q,\alpha q^{-s(x+1)},m(x)}(y_{k_s(x)}=0 | \vert Y \vert = k_s(x)) \\
=\frac{q(1+\alpha q^{-s(x)})}{1 + \alpha q^{-s(x \pm 1)}} \left[ \frac{\binom{m(x-1)-1}{k_s(x-1)}_q}{\binom{m(x-1)}{k_s(x-1)}_q} \cdot \frac{\alpha q^{m(x-1)-s(x)} + q^{m(x-1)+k_s(x-1)-1}}{\alpha q^{m(x-1)-s(x)} + q^{m(x-1)-1}} \right] \left[ \frac{  \binom{m(x)-1}{k_s(x)-1}_q  }{ \binom{m(x)}{k_s(x)}_q } \frac{ \alpha q^{m(x)-s(x+1)} + q^{m(x)+k_s(x)-1}   }{ \alpha q^{m(x)-s(x+1)} + q^{2k_s(x)-1}} \right],  
\end{multline*}
if $s'(x)=s(x)-2$. Note that these simplify, respectively, to 

\begin{multline}\label{Line1}
q^{m(x-1)-k_s(x-1)+k_s(x)} \cdot \frac{1-q^{k_s(x-1)}}{ 1 - q^{m(x-1)}} \cdot \frac{1-q^{m(x)-k_s(x)}}{1-q^{m(x)}} \cdot \frac{1 + {\alpha} q^{-s(x)} }{1 + {\alpha} q^{-s(x \pm 1)}  } \\
\quad \quad \quad \quad \times  \frac{\alpha q^{m(x-1)-s(x)} + q^{k_s(x-1)-1}}{\alpha q^{m(x-1)-s(x)} + q^{m(x-1)-1}}    \frac{\alpha q^{m(x)-k_s(x)} + q^{k_s(x)-1}   }{ \alpha q^{m(x)-k_s(x)} + q^{2k_s(x)-1}}  \\
\end{multline}
and
\begin{multline}\label{Line2}
q\frac{1-q^{m(x-1)-k_s(x-1)}}{1-q^{m(x-1)}} \cdot \frac{1-q^{k_s(x)}}{1-q^{m(x)}} \cdot \frac{1+\alpha q^{-s(x)}}{1 + \alpha q^{-s(x \pm 1)}} \\
\times  \frac{\alpha q^{-s(x)} + q^{k_s(x-1)-1}}{\alpha q^{-s(x)} + q^{-1}}  \frac{ \alpha q^{m(x)-k_s(x)} + q^{m(x)+k_s(x)-1}   }{ \alpha q^{m(x)-k_s(x)} + q^{2k_s(x)-1}} .
\end{multline}

If $\alpha= 0$, then the jump rates become
\begin{align*}
&\frac{1-q^{k_s(x-1)}}{1-q^{m(x-1)}} \cdot \frac{1-q^{m(x)-k_s(x)}}{1-q^{m(x)}},\\
 q^{m(x) -k_s(x) + k_s(x-1)+1} & \cdot \frac{1-q^{m(x-1)-k_s(x-1)}}{1-q^{m(x-1)}} \cdot \frac{1-q^{k_s(x)}}{1-q^{m(x)}} ,
\end{align*}
which are the left and right jump rates of ASEP$(q,\vec{m})$. When $\alpha \rightarrow \infty$, then the jump rates become
\begin{align*}
&q^{m(x-1)-k_s(x-1)+k_s(x)+1} \cdot \frac{1-q^{k_s(x-1)}}{ 1 - q^{m(x-1)}} \cdot \frac{1-q^{m(x)-k_s(x)}}{1-q^{m(x)}}, \\
& \frac{1-q^{m(x-1)-k_s(x-1)}}{1-q^{m(x-1)}} \cdot \frac{1-q^{k_s(x)}}{1-q^{m(x)}} ,
\end{align*}
which are the right and left jump rates of ASEP$(q,\vec{m})$. In other words, the dynamical ASEP$(q,\vec{m})$ interpolates between the usual ASEP$(q,\vec{m})$ and its space reversal.

\subsubsection{Degeneration to dynamic SSEP$(\vec{m})$}
Now take the symmetric limit, as in Section \ref{dynSym}. Recall that we make the substitution $\tilde{s}(x) = s(x) - \log_q \vert \alpha\vert$. The first lines in \eqref{Line1} and \eqref{Line2} yield the dynamical SSEP jump rates. For the second lines, note that the shift in $s(x)$ does not change the values of $m(x)$ or $k_s(x)$. Therefore one of the two terms in the second lines will converges to $1$. Combining all the terms, the jump rates become
$$
\frac{k_s(x-1)}{m(x-1)} \frac{m(x) - k_s(x)}{m(x)} \frac{s(x) - \lambda}{s(x\pm 1) - \lambda} \frac{s(x)-\lambda -1 - m(x-1)+k_s(x-1)}{s(x)-\lambda -1}
$$
and
$$
\frac{m(x-1) - k_s(x-1)}{m(x-1)} \frac{k_s(x)}{m(x)}\frac{s(x) - \lambda}{s(x\pm 1) - \lambda} \frac{s(x)-\lambda -1 + k_s(x-1)}{s(x)-\lambda -1} .
$$
Note that substituting $k_s(y)\mapsto m(y)-k_s(y)$ for $y=x,x-1$ switches the two jump rates.

We note that another way to the symmetric limit is to set $\alpha = -q^{\lambda}$ and take the $q\rightarrow 1$ limit. In this limit, the jump rates become
$$
\frac{k_s(x-1)}{m(x-1)} \frac{m(x) - k_s(x)}{m(x)} \frac{s(x) - \lambda}{s(x\pm 1) - \lambda} \frac{s(x)-\lambda -1 - m(x-1)+k_s(x-1)}{s(x)-\lambda -1} \frac{-m(x)+2k_s(x)-\lambda - 1 }{-m(x)+3k_s(x)-\lambda - 1}
$$
and
$$
\frac{m(x-1) - k_s(x-1)}{m(x-1)} \frac{k_s(x)}{m(x)}\frac{s(x) - \lambda}{s(x\pm 1) - \lambda} \frac{s(x)-\lambda -1 + k_s(x-1)}{s(x)-\lambda -1} \frac{2k_s(x) -\lambda -1   }{3k_s(x) - m(x) - \lambda - 1} .
$$
The existence of two different symmetric limits highlights one of the differences between dynamic SSEP and dynamic SSEP$(\vec{m})$, which is that the former is preserved under the simultaneous transformations $\alpha\mapsto \alpha q, s(x) \mapsto s(x)+1$, while the latter is not.

If all $m(x)=1$, then the first jump rates are only nonzero if $k_s(x-1)=1$ and $k_s(x)=0$, and we obtain the dynamic SSEP jump rate. Similarly, the second jump rates are only nonzero if $k_s(x-1)=0$ and $k_s(x)=1$, and we again obtain the dynamic SSEP jump rate. 

If $\lambda \rightarrow -\infty$, we recover the usual SSEP$(\vec{m})$.

\subsubsection{Degeneration to dynamic $q$--Boson}\label{diffe}
Fix a state $s$ of dynamic ASEP$(q,\vec{m})$. Suppose that all $m(x)$ are taken to infinity, and all $k_s(x)$ are also taken to infinity in such a way that $m(x)-2k_s(x)$ is finite. This means that the function $s(x)$ is still well--defined in the limit.  In this limit, we have
\begin{multline*}
\mathbb{P}_{q,\alpha q^{-s(x)},m(x-1)}(  y_1 = m(x-1)-1 | \vert Y\vert = k_s(x-1) )  \rightarrow \frac{q^{-1}}{\alpha q^{-s(x)} + q^{-1}}, \\
\quad  \quad \mathbb{P}_{q,\alpha q^{-s(x+1)},m(x)}(y_{k_s(x)}>0 | \vert Y \vert = k_s(x)) \rightarrow  \frac{1}{1 + \alpha q^{m(x)-2k_s(x)+1-s(x+1)}} = \frac{1}{1 + \alpha q^{1-s(x)}},
\end{multline*}
so the jump rate from $s(x)$ to $s(x)+2$ is zero. Likewise,
$$
 \mathbb{P}_{q,\alpha,m(x-1)}( y_1 < m(x-1)-1 | \vert Y\vert = k_s(x-1) )  \rightarrow \frac{\alpha q^{-s(x)}  }{\alpha q^{-s(x)} + q^{-1}}, \quad\mathbb{P}_{q,\alpha,m(x)}(y_{k_s(x)}=0 | \vert Y \vert = k_s(x))   \rightarrow 1,
$$
so the jump rate from $s(x)$ to $s(x)-2$ is
$$
q \cdot \frac{1+\alpha q^{-s(x)}}{1 + \alpha q^{-s(x)+1}} \cdot \frac{\alpha q^{-s(x)}  }{\alpha q^{-s(x)} + q^{-1}}  .
$$





\subsection{A duality Ansatz for finite--lattice dynamic ASEP} \label{Dynamical} 

Here, we model a duality Ansatz using a similar argument from section \ref{SchASEP}.

Fix integers $t, r, m > 0$. Let $X = (x_1, x_2, \ldots,  x_t)$ and $Y = (y_1, y_2, \ldots , y_r)$ be sets of integers such $0 \le x_{i + 1} < x_i \le m -1$ and $0 \le y_{j + 1} < y_j \le m - 1$ for each $1 \le i \le t - 1$ and $0 \le j \le r - 1$. 

For any finite set $V$ of integers and integer $u$, define $h_B (u)$ to be the number of $v \in V$ greater than or equal to $u$; similarly, define $n_V (u)$ to be the number of $v \in V$ less than $u$. Observe that $h_V (u) + n_V (u) = |S|$ for any $u \in \mathbb{Z}$. Further denote 
\begin{flalign*}
h_V (u) = u - 2 |V| + 2 h_V (u), 
\end{flalign*}
and  for any set $V$ of nonnegative integers and for any integers $a \le b$, let $\omega_V (a, b)$ denote the number of $v \in Y$ such that $a < y \le b$. 
\noindent Now define 
\begin{flalign*}
F(X; Y) = \pi (X) \displaystyle\prod_{k = 1}^r T_k (X; Y) \displaystyle\prod_{x_j > y_k} P_{j, k} (X; Y), 
\end{flalign*}

\noindent where 
\begin{flalign*}
& \pi (X) = \prod_{j = 1}^t \displaystyle\frac{q^{x_j}}{(\alpha + q^{x_j + 2 (j - k) - 2}) (\alpha + q^{x_j + 2(j - k) - 1})}; \\
T_k (X; Y) = q^{-y_k} & (\alpha + q^{-s_X (y_k) }) (\alpha + q^{-s_X (y_k) + 1 }) \textbf{1}_{y_k \in X}; \qquad P_{j, k} (X; Y) = \displaystyle\frac{\alpha + q^{U + 1}}{q (\alpha + q^{U-1})}, 
\end{flalign*}

\noindent where $U = U_{j, k} (X; Y)$ is the function
\begin{equation}\label{ujk}
U_{j,k}(X;Y) = 
\begin{cases}
s_X(x_j) + \omega_Y(y_k,x_j), & \text{ if } x_j \notin Y,\\
\infty, & \text{ if }x_j \in Y
\end{cases}
\end{equation}

\begin{prop}
The value of $F(X;Y)$ does not depend on $Y$. 
\end{prop}
\begin{proof}
As above, suppose that $1 \le i \le r$ is some index such that $y_i \in Y$ but $y_i - 1 \notin Y$. Let $Z = (z_1, z_2, \ldots , z_r)$ denote the set obtained by setting $z_i = y_i - 1$ and $z_j = y_j$ for any $j \in [1, r] \setminus \{ j \}$. 

Now, assume that $y_i = x_a$ for some indices $i \in [1, r]$ and $a \in [1, t]$, and define $\overline{X} = (\overline{x}_1, \overline{x}_2, \ldots , \overline{x}_k)$ according to the following two cases. First, if $y_i - 1 \in X$, set $\overline{X} = X$. Second, if $y_i - 1 \notin X$, then $\overline{X}$ is obtained from $X$ by replacing $y_i \in X$ with $y_i - 1$. 

We would like to show that $F(X; Y) = F (\overline{X}; Z)$. If $y_i - 1 \in X$, then this is equivalent to 
\begin{flalign*}
\displaystyle\prod_{k = 1}^r q^{-y_k} & (\alpha + q^{-s_X (y_k) }) (\alpha + q^{-s_X (y_k) + 1}) \textbf{1}_{y_k \in X} \displaystyle\prod_{x_j > y_k} \displaystyle\frac{\alpha + q^{U_{j, k} (X; Y) + 1}}{q (\alpha + q^{U_{j, k} (X; Y) - 1})} \\
& = \displaystyle\prod_{k = 1}^r q^{- z_k} (\alpha + q^{-s_X (z_k) }) (\alpha + q^{-s_X (z_k) + 1}) \textbf{1}_{z_k \in X} \displaystyle\prod_{x_j > z_k} \displaystyle\frac{\alpha + q^{U_{j, k} (X; Z) + 1}}{q (\alpha + q^{U_{j, k} (X; Z) - 1})},
\end{flalign*}

\noindent since $\overline{X} = X$. Since $z_i = y_i - 1$ and $z_j = y_j$ for each $j \ne i$, the above is equivalent to 
\begin{flalign*}
\displaystyle\frac{\alpha + q^{-s_X (y_i) +1}}{q (\alpha + q^{-s_X (y_i)-1})} &  \displaystyle\prod_{x_j > y_k} \displaystyle\frac{\alpha + q^{U_{j, k} (X; Y) + 1}}{q (\alpha + q^{U_{j, k} (X; Y) - 1})} \\
& = \displaystyle\frac{\alpha + q^{U_{a, i} (X; Z) + 1}}{q (\alpha + q^{U_{a, i} (X; Z) - 1})} \displaystyle\prod_{\substack{x_j > z_k \\ (j, k) \ne (a, i)}} \displaystyle\frac{\alpha + q^{U_{j, k} (X; Z) + 1}}{q (\alpha + q^{U_{j, k} (X; Z) - 1})},
\end{flalign*}

\noindent where we have used the fact that $s_X (y_i) = s_X (y_i - 1) - 1$. Thus, the condition  $F(X;Y)=F(\bar{X};Z)$ (for $y_i-1\in X$) holds if
\begin{flalign}
\label{1conditionu}
\begin{aligned}
& U_{j, k} (X; Y) = - s_X (x_j), \qquad \quad	 \text{if $x_j = y_k + 1$, $y_k \in X$, and $y_k + 1 \notin Y$}; \\
& U_{j, k} (X; Y) = U_{j, k} (X, Z), \quad \text{if $(x_j, z_k) \ne (y_i, y_i - 1)$}. 
\end{aligned} 
\end{flalign}
We can immediately check that this holds.

Now, suppose that $y_i - 1 \notin X$. Then, the condition $F(X; Y) = F(\overline{X}; Z)$ is equivalent to 
\begin{flalign*}
\displaystyle\frac{\pi (X)}{\pi (\overline{X})} \displaystyle\frac{(\alpha + q^{-s_X (y_i) + 1}) (\alpha + q^{-s_X (y_i) })}{q (\alpha + q^{-s_{\overline{X}} (z_i) + 1}) (\alpha + q^{-s_{\overline{X}} (z_i) })} &  \displaystyle\prod_{x_j > y_k} \displaystyle\frac{\alpha + q^{U_{j, k} (X; Y) + 1}}{q (\alpha + q^{U_{j, k} (X; Y) - 1})}  =  \displaystyle\prod_{\overline{x}_j > z_k} \displaystyle\frac{\alpha + q^{U_{j, k} (X; Z) + 1}}{q (\alpha + q^{U_{j, k} (\overline{X}; Z) - 1})}.
\end{flalign*}
Using the fact that $x_a = y_i$ and $\overline{x}_a = y_i - 1 = z_i$, we deduce that 
\begin{flalign*} 
\displaystyle\frac{\pi (X)}{\pi (\overline{X})} & = \displaystyle\frac{q (\alpha + q^{\overline{x}_a + 2 (a - k) - 1}) (\alpha + q^{\overline{x}_a + 2 (a - k) - 2})}{(\alpha + q^{x_a + 2 (a - k) - 1}) (\alpha + q^{x_a + 2 (a - k) - 2})} = \displaystyle\frac{q (\alpha + q^{-s_{\overline{X}} (z_i) +1}) (\alpha + q^{-s_{\overline{X}} (z_i) })}{(\alpha + q^{- s_X (y_i) + 1}) (\alpha + q^{-s_X (y_i) })},
\end{flalign*}

\noindent so the above is equivalent to 
\begin{flalign*}
 \displaystyle\prod_{x_j > y_k} \displaystyle\frac{\alpha + q^{U_{j, k} (X; Y) + 1}}{q (\alpha + q^{U_{j, k} (X; Y) - 1})}  =  \displaystyle\prod_{x_j > z_k} \displaystyle\frac{\alpha + q^{U_{j, k} (X; Z) + 1}}{q (\alpha + q^{U_{j, k} (\overline{X}; Z) - 1})}.
\end{flalign*}

\noindent Thus, it suffices to show that
\begin{flalign}
\label{2conditionu}
U_{j, k} (X; Y) = U_{j, k} (\overline{X}; Z), \quad \text{for all $1 \le j, k \le r$.} 
\end{flalign}
But this follows immediately as well.
\end{proof}

It turns out, however, that this does not produce a duality. Despite this, it is still plausible that the dynamic ASEP$(q,m/2)$ has the same weak asymmetry limit as the dynamic ASEP. In \cite{CGMDyn}, it was proven that the dynamic ASEP converges to the solution of the space--time Ornstein--Uhlenbeck equation. Therefore, one can ask:

\textbf{Question:} Does the dynamic ASEP$(q,m/2)$ also converge to the solution of the space--time Ornstein--Uhlenbeck equation?

\section{Appendix co--authored with Amol Aggarwal}
{\color{black} In this appendix, we describe the algebraic symmetry underlying the interacting particle systems, as well as its relationship to stochastic vertex models.}

\subsection{Multi--species dynamic ASEP}

We make a few comments about the construction of a multi--species dynamic ASEP. If we wish to maintain the Markov projection property, it is natural to consider $n$--tuples of functions $s:\mathbb{Z}\rightarrow \mathbb{Z}$; that is, $(s^{(0)}, s^{(1)}, \ldots, s^{(n)})$ such that $s^{(0)}(x) \leq \ldots \leq s^{(n)}(x)$ for all $x\in \mathbb{Z}$. Furthermore, for every $x$, there is a unique $j \in  \{0,\ldots,n\}$ such that $s^{(i)}(x+1)-s^{(i)}(x)=1$ for $0\leq i \leq j$ and $s^{(i)}(x+1)-s^{(i)}(x)=1$ for $i>j$.  This $j$ corresponds to a species $j$ particle at lattice site $x$.  See Figure \ref{dASEP}.

If we wish to define jump rates between different states, then there are up to $n$ different possibilities, corresponding to the $n$ different species. See e.g. Figure \ref{PJ}. We also note an algebraic reason for the Markov projection property: the $n$--species ASEP can be constructed from the Drinfeld--Jimbo quantum group $\mathcal{U}_q(\mathfrak{gl}_{n+1})$ (see \cite{KIMRN}), and by standard representation the Lie algebra $\mathfrak{gl}_{n+1}$ is built out of copies of $\mathfrak{sl}_2$, and this fact corresponds to the Markov projection property. However, this no longer holds for the Felder--Varchenko elliptic quantum group.

Another construction of multi--species dynamic models comes from section 5 of \cite{ABBDynamic} (see, in particular Figure 17 of that paper). There, the dynamic parameter only changes based on the number of particles, not the species of the particle.  So, for instance, $\Pi = \{ \{0,1\},\{2,\ldots,n\}\}$ would not define a Markov projection. {\color{black} However, it is possible that the method of \cite{ABBDynamic}, but with the underlying algebra as $E_{\eta,\tau}(\mathfrak{sl}_n)$ rather than $\mathcal{U}_q(\mathfrak{gl}_{n+1})$, would produce a multi--species dynamic ASEP with the Markov projection property. We do not pursue this direction here.}

\begin{figure}
\caption{State space and jumps of a multi--species dynamic ASEP.}
\label{dASEP}
\begin{center}
\begin{tikzpicture}[scale=0.7]

\foreach \y in {0,1,2,3,4}{
\draw [dashed](\y+1.5,-7) -- (\y+1.5,7);
\draw [very thick](0 ,-6) -- (8 ,-6);
\fill[blue] (1 ,-6) circle (3pt);
\fill[yellow] (4 ,-6) circle (3pt);
\fill[violet] (5 ,-6) circle (3pt);
\fill[red] (2 ,-6) circle (3pt);
\fill[green] (3 ,-6) circle (3pt);
\fill[orange] (6 ,-6) circle (3pt);
\draw [purple, thick](0.5 ,-4) -- (6.5 ,2);
\draw [blue, thick](0.5 ,-2) -- (1.5 ,-1)  -- (2.5 ,0) -- (4.5,2) -- (5.5,1) -- (6.5 ,2);
\draw [green, thick](0.5 ,0) -- (1.5 ,-1) -- (2.5 ,0) -- (3.5 ,1) -- (4.5 ,2) -- (5.5 ,1) -- (6.5 ,2);
\draw [yellow, thick](0.5 ,2) -- (1.5 ,1) -- (2.5 ,2) --  (3.5,1) -- (4.5,2) -- (5.5 ,1) -- (6.5 ,2);
\draw [orange, thick](0.5 ,4) -- (1.5 ,3) -- (2.5 ,4) -- (5.5 ,1) -- (6.5 ,2);
\draw [red, thick](0.5 ,6) -- (1.5 ,5) -- (2.5 ,6) -- (6.5 ,2);

\draw (2.2,-5.5) node (a) { };
\draw (1.8,-5.5) node (a') { };
\draw (1,-5.5) node (b) { };
\draw (3,-5.5) node (c) {};
\draw (a') edge[out=90,in=90,->, line width=0.5pt] (b);
\draw (a) edge[out=90,in=90,->, line width=0.5pt] (c);

}
\end{tikzpicture}

\begin{tikzpicture}[scale=0.7]

\foreach \y in {0,1,2,3,4}{
\draw [dashed](\y+1.5,-7) -- (\y+1.5,7);
\draw [very thick](0 ,-6) -- (8 ,-6);
\fill[blue] (1 ,-6) circle (3pt);
\fill[yellow] (4 ,-6) circle (3pt);
\fill[violet] (5 ,-6) circle (3pt);
\fill[green] (2 ,-6) circle (3pt);
\fill[red] (3 ,-6) circle (3pt);
\fill[orange] (6 ,-6) circle (3pt);
\draw [purple, thick](0.5 ,-4) -- (6.5 ,2);
\draw [blue, thick](0.5 ,-2) -- (1.5 ,-1)  -- (2.5 ,0) -- (4.5,2) -- (5.5,1) -- (6.5 ,2);
\draw [green, thick](0.5 ,0) -- (1.5 ,-1) -- (2.5 ,0) -- (3.5 ,1) -- (4.5 ,2) -- (5.5 ,1) -- (6.5 ,2);
\draw [yellow, thick](0.5 ,2) -- (1.5 ,1) -- (2.5 ,0) --  (3.5,1) -- (4.5,2) -- (5.5 ,1) -- (6.5 ,2);
\draw [orange, thick](0.5 ,4) -- (1.5 ,3) -- (2.5 ,2) -- (3.5,3) -- (5.5 ,1) -- (6.5 ,2);
\draw [red, thick](0.5 ,6) -- (1.5 ,5) -- (2.5 ,4) -- (3.5,5) -- (6.5 ,2);
}

\foreach \y in {0,1,2,3,4}{
\draw [dashed](\y+10+1.5,-7) -- (\y+10+1.5,7);
\draw [very thick](0+10 ,-6) -- (8+10 ,-6);
\fill[red] (1+10 ,-6) circle (3pt);
\fill[yellow] (4+10 ,-6) circle (3pt);
\fill[violet] (5+10 ,-6) circle (3pt);
\fill[blue] (2+10 ,-6) circle (3pt);
\fill[green] (3+10 ,-6) circle (3pt);
\fill[orange] (6+10 ,-6) circle (3pt);
\draw [purple, thick](0.5+10 ,-4) -- (6.5+10 ,2);
\draw [blue, thick](0.5+10 ,-2) -- (1.5+10 ,-1)  -- (2.5+10 ,0) -- (4.5+10,2) -- (5.5+10,1) -- (6.5+10 ,2);
\draw [green, thick](0.5+10 ,0) -- (1.5+10 ,1) -- (2.5+10 ,0) -- (3.5+10 ,1) -- (4.5+10 ,2) -- (5.5+10 ,1) -- (6.5+10 ,2);
\draw [yellow, thick](0.5+10 ,2) -- (1.5+10 ,3) -- (2.5+10 ,2) --  (3.5+10,1) -- (4.5+10,2) -- (5.5+10 ,1) -- (6.5+10 ,2);
\draw [orange, thick](0.5+10 ,4) -- (1.5+10 ,5) -- (2.5+10 ,4) -- (5.5+10 ,1) -- (6.5+10 ,2);
\draw [red, thick](0.5+10 ,6) -- (1.5+10 ,7) -- (2.5+10 ,6) -- (6.5+10 ,2);
}

\end{tikzpicture}
\end{center}

\end{figure}
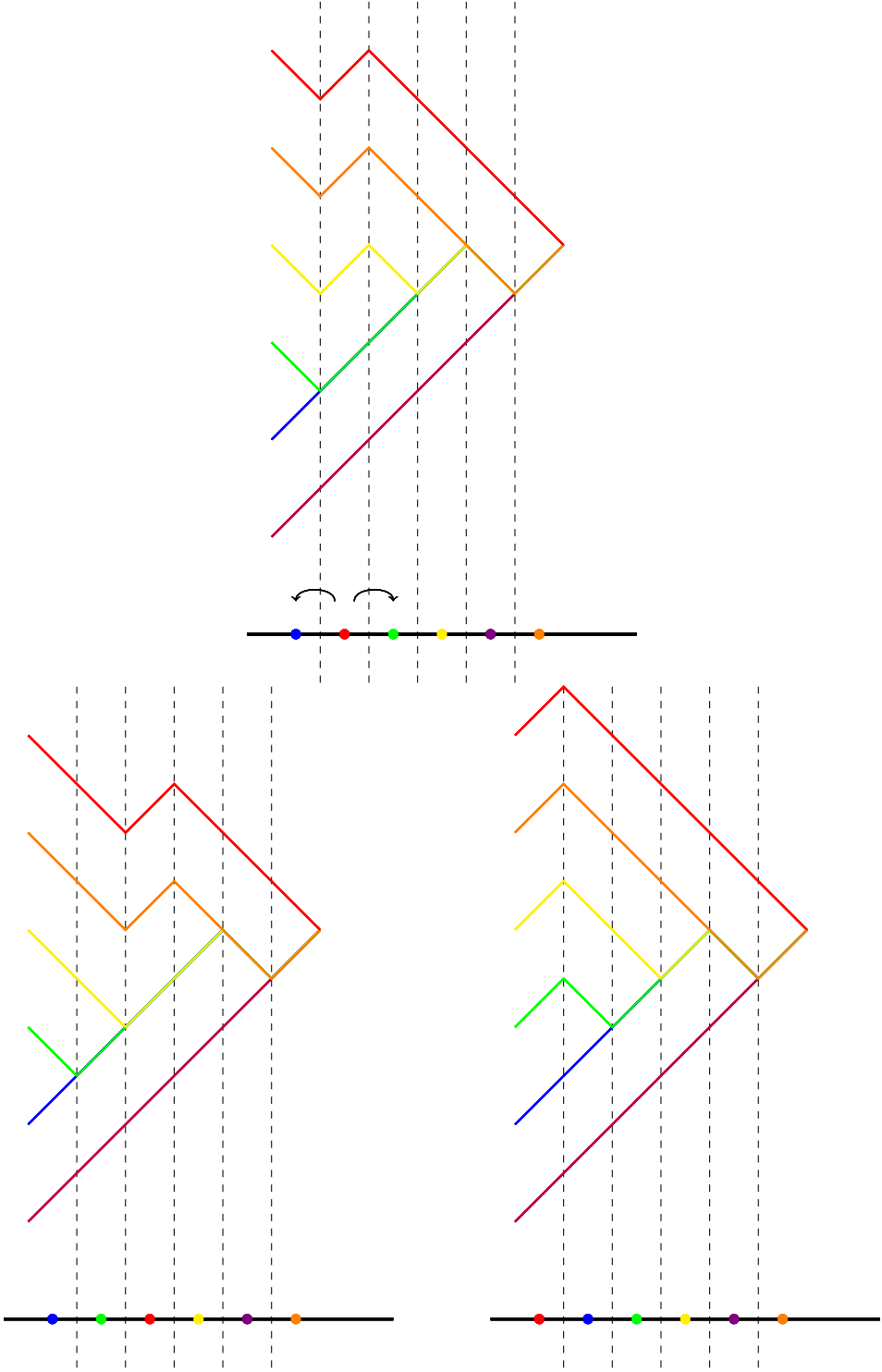

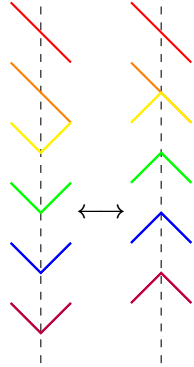
\begin{figure}
\caption{Possible Jumps}
\begin{center}
\begin{tikzpicture}[scale=0.4]
\foreach \x in {0,4}{
\draw [dashed](0+\x,-4) -- (0+\x,8);
\draw [red,thick] (-1+\x,8) --(1+\x,6);
\draw [orange,thick] (-1+\x,6) --(1+\x,4);
\draw [yellow,thick] (-1,4) -- (0,3) -- (1,4);
\draw [green,thick] (-1,2) -- (0,1) -- (1,2);
\draw [blue,thick] (-1,0) -- (0,-1) -- (1,0);
\draw [purple,thick] (-1,-2) -- (0,-3) -- (1,-2);
\draw (2,1) node {$\longleftrightarrow$};
\draw [yellow,thick] (3,4) -- (4,5) -- (5,4);
\draw [green,thick] (3,2) -- (4,3) -- (5,2);
\draw [blue,thick] (3,0) -- (4,1) -- (5,0);
\draw [purple,thick] (3,-2) -- (4,-1) -- (5,-2);
}
\end{tikzpicture}
\end{center}
\label{PJ}
\end{figure}

\subsection{Fusion}\label{alg}
Here, we sketch a description of fusion from a representation theoretic perspective as well as a probabilistic interpretation. Consider a two--dimensional vector space $V$, with basis vectors $\vert 0 \rangle$ and $\vert 1 \rangle$. The $m$--fold tensor product $V^{\otimes m}$ is a $2^m$--dimensional vector space with a vacuum vector 
$$
\Omega := \vert \underbrace{0,0,\ldots,0}_m\rangle.
$$
If $\vert 0 \rangle$ is viewed as a hole and $\vert 1\rangle$ as a particle, then $\Omega$ represents the particle configuration consisting of zero particles on a one--dimensional lattice with $m$ sites. 

For any interacting particle system on a finite lattice which satisfies particle number conservation, the vacuum state $\Omega$ is trivially a stationary measure. If the state space is restricted to particle configurations where the number of particles is a fixed number $k$, there is a unique stationary measure. For many examples of interacting particle systems, there is a creation operator which constructs the $k$--particle stationary measure from the $(k-1)$--particle stationary measure. In other words, the creation operator picks a lattice site at random and creates a particle at that site if possible; and this operator applied to the $(k-1)$--particle stationary measure yields the $k$--particle stationary measure. By repeatedly applying the creation operator to the vacuum vector, one finds all of the $k$--particle stationary measures for $0\leq k \leq m$.
One could also begin with the state $\vert 1,1,\ldots,1\rangle$ consisting entirely of particles, and apply annihilation operators to construct the $k$--particle stationary measure from the $(k+1)$--particle stationary measure. By the uniqueness of the $k$--particle stationary measure, the creation and annihilation operators must yield the same measures. In particular, there must be some consistency relations between the creation and annihilation operators that depend on the number of particles. 

These consistency relations can be formalized as the relations in an algebra generated by the creation, annihilation and number operators. For ASEP, the algebra is the Drinfeld--Jimbo quantum group $\mathcal{U}_q(\mathfrak{sl}_2)$, and for dynamic ASEP the algebra is the Felder--Varchenko elliptic quantum group $E_{\eta,\tau}(\mathfrak{sl}_2)$. (It is usually easier to start with the algebraic structure and construct the particle system, rather than constructing the algebra from the particle system). If $V_m$ is defined to be the submodule of $V^{\otimes m}$ generated by the vacuum vector $\Omega$, then $V_m$ must be $(m+1)$--dimensional (if the particle system preserves particle number), with a basis vector corresponding to each $k$--particle stationary measure for $0\leq k \leq m$. The module $V_m$ is also the spin $m/2$ representation of the algebra, which is used in the higher spin vertex models and the ASEP$(q,m/2)$. 

{\color{black} The embedding of $V_m$ into $V^{\otimes m}$ can be defined via $R$--matrices. We describe this for $\mathcal{U}_q(\mathfrak{sl}_2)$Let $s_m$ denote the symmetrizer of $V^{\otimes m}$ onto $V_m$, and let $R(z)$ denote the $4\times 4$ matrix acting on $V \otimes V$ defined by 
$$
R(z)=\left(\begin{array}{cccc}{1} & {0} & {0} & {0} \\ {0} & {\frac{1-z}{q^{2}-z}} & {\frac{q^{2}-1}{q^{2}-z}} & {0} \\ {0} & {\frac{z\left(q^{2}-1\right)}{q^{2}-z}} & {\frac{q^{2}(1-z)}{q^{2}-z}} & {0} \\ {0} & {0} & {0} & {1}\end{array}\right)
$$
Let $P$ be the operator on $V\otimes V$ that permutes the ordering, i.e. $P(v\otimes w) = w\otimes v$, and set $\check{R}(z) = R(z) \circ P.$ Then $s_m$ can be defined inductively by (see \cite{Jimbo1986})
$$
s_{m+1} =\frac{1}{q^{m+1}-q^{-m-1}} s_{m}^{} \check{R}_{m, m+1}\left(q^{ m}\right) s_{m}^{},
$$
where we have used the usual spin chain notation, in which the subscripts $\hat{R}_{ij}$ indicate that $\hat{R}$ is acting on the $i$th and $j$th components of the tensor product. Moreover, $R(z)$ satisfies 
$$ 
R_{0 m}(z) R_{0, m-1}\left(z q\right) \cdots R_{01}\left(z q^{m-1}\right) s_{m}= s_{m} R_{0 m}\left(z q^{m-1}\right) R_{0, m-1}\left(z q^{m-2}\right) \cdots R_{01}(z),
$$
which is often called \textit{fusion}. On the left (resp. right) hand side, the power of $q$ increases (resp. decreases) by $1$ at each step; similarly, in the blocking reversible measures for ASEP (resp. space--reversed ASEP), the power of $q$ increases (resp. decreases) by $1$ whenever a particle jumps one step to the right. This is no coincidence: note that $\check{R}(1 + \epsilon) = \mathrm{Id} + \epsilon L_{\text{ASEP}}$, where $L_{\text{ASEP}}$ is the generator of ASEP. The degeneration of the stochastic six vertex model to ASEP corresponds to taking $z\rightarrow 1$; note, however, that the corresponding degeneration for higher spin vertex models does not produce a matrix with non--negative off--diagonal entries \cite{CorPetCMP}.}

If one defines $\Lambda:V_m\rightarrow V^{\otimes m}$ and $\Phi: V^{\otimes m} \rightarrow V_m$ to be the inclusion and symmetrizer maps, then these are the fission and fusion maps of Figure \ref{LP}. Note that it is immediate that $ \Lambda \Phi$ is the identity on $V_m$. Additionally, the algebra elements define an intertwining between $\Lambda$ and $\Phi$. In particular, we claim that for any algebra element $a$, $\Lambda a = \Lambda a \Phi \Lambda$ on $V_m$. This is simply because $\Phi\Lambda$ is the identity on $\mathrm{im}(\Lambda)$ and $a$ sends $\mathrm{im}(\Lambda)$ to itself. The two equations $\Lambda\Phi = \mathrm{id}$ and $\Lambda a  = \Lambda a \Phi \Lambda$ are the two equations in Rogers--Pitman intertwining.  {\color{black} The equation defining fusion suggests that $\Lambda$ should be defined according to the reversible measures of the process, as was done in this paper. }

\subsection{Higher spin dynamical stochastic vertex weights}\label{hsdsvw}
As introduced in the previous section, fusion applies to vertex weights coming from the elliptic quantum group $E_{\tau, \eta} (\mathfrak{sl}_2)$. Here, we implement this procedure and compare the resulting weights to those introduced in Definition 1.1 of \cite{AmolIRF}. We will see that the weights here (see \eqref{rw} below) differ from those previous weights, but only up to a conjugating factor. We were also unable to realize our weights as direct special cases of those given by Theorem 4.6 of \cite{ABBDynamic}. Still, it should be possible to obtain our weights by applying the stochasticization procedure of \cite{ABBDynamic} to the transposed (as opposed to the original) $R$-matrix entries of $E_{\tau, \eta} (\mathfrak{sl}_2)$, but we will not pursue this here.

Set $\theta(z)=\sin(\pi z)$. Define the two functions 
$$
\alpha(\lambda,w,\eta,\tau) = \frac{\theta(w)\theta(\lambda+2 \eta)}{\theta(w-2\eta)\theta(\lambda)}, \quad \quad\beta(\lambda,w,\eta,\tau) =  \frac{\theta(-w-\lambda)\theta(2 \eta)}{\theta(w-2\eta)\theta(\lambda)} 
$$
More generally, define
\begin{align*}
\alpha^{(r)}(\lambda,w) & := \alpha(\lambda,w) \alpha(\lambda+2\eta,w+2\eta) \cdots \alpha(\lambda + 2(r-1)\eta, w+2(r-1)\eta) = \frac{\theta(w+2(r-1)\eta)\theta(\lambda+2r\eta)}{\theta(w-2\eta)\theta(\lambda)} \\
\beta^{(r)}(\lambda,w) & = \frac{\theta( - w - \lambda  -2(r-1)\eta )\theta(2r\eta)}{\theta(w-2\eta)\theta(\lambda)},
\end{align*}
and set $\alpha^{(0)}(\lambda,w)=1,\beta^{(0)}(\lambda,w)=0$ by convention.

Fix a positive integer $m$. For $0\leq k \leq m$, define functions $a^{(k)}_{\pm},b^{(k)}_{\pm}$ by 
\begin{align*}
a_+^{(k)}(\lambda,w) &= \alpha^{(k)}(\lambda,w)\\
a_-^{(k)}(\lambda,w) &= \alpha^{(m-k)}(-\lambda,w)\\
b_+^{(k)}(\lambda,w) &= \beta^{(k)}(\lambda,w)\\
b_-^{(k)}(\lambda,w) &= \beta^{(m-k)}(-\lambda,w)
\end{align*}

Now fix a positive integer $l$. Given $0 \leq j \leq l$, let $\mathcal{I}_j$ be the set
$$
\mathcal{I}_j := \{ (i_1,\ldots,i_l): i_1 + \ldots + i_l = j\} \subset \{0,1\}^l.
$$
Let $\vec{i}$ denote an element of $\mathcal{I} := \{0,1\}^l$. For $1 \leq r \leq l$, define the map $p_r^{(j')}$ on $\mathcal{I}$ by 
$$
p^{(j')}_r(\vec{i})=
\begin{cases}
 i_1 + \ldots + i_r, \quad & 1 \leq r \leq l-j' \\
 i_1 + \ldots + i_{l-j'} + (i_{l-j'+1} -1 ) + \ldots + (i_r-1) , \quad & l-j' < r \leq l
\end{cases}
$$
Also define
$$
h_r(\vec{i})= | \{ s: 1 \leq s \leq r \text{ and } i_s = 0 \} | -  | \{ s: 1 \leq s \leq r \text{ and } i_s = 1 \} |
$$
For each $\vec{i} \in \mathcal{I}$, and for $1\leq r\leq l$, define the symbol $\epsilon_{r,j'}(\vec{i})$ by 
$$
 \epsilon_{r,j'}(\vec{i}) = 
\begin{cases}
a_+, \quad & \text{ if } 1 \leq r \leq l-j' \text{ and } i_r = 0 \\
b_+, \quad & \text{ if } 1 \leq r \leq l-j' \text{ and } i_r = 1 \\
a_-, \quad & \text{ if } l-j' < r \leq l \text{ and } i_r = 1 \\
b_-, \quad & \text{ if } l-j' < r \leq l \text{ and } i_r = 0 
\end{cases}
$$
Given $0\leq j,j'\leq l$ and $0 \leq k,k'\leq m$ such that $j+k = j'+k'$, consider the sum
$$
R_{lm}(j',k';j,k) := \sum_{ \vec{i} \in \mathcal{I}_j} \ep_{1,j'}^{(k')}(\lambda,w) \ep_{2,j'}^{  (k'-p_1^{(j')}) }  (\lambda - 2\eta h_1,w-2\eta) \cdots \epsilon_{l,j'}^{(k'-p_{l-1}^{  (j')  } )  }(\lambda - 2\eta h_{l-1}, w - 2(l-1)\eta).
$$
In the summation, the dependence of $\epsilon_r, p_r^{(j')}$ and $h_r$ on $\vec{i}$ has been suppressed for convenience of notation. This sum is an equivalent expression of the fusion identity of the previous section. Also note that for $\vec{i} \in \mathcal{I}_j$,
$$
k' - p_l^{(j')}(\vec{i}) = k' - ( j - j') = k
$$
For example,
\begin{align*}
R_{22}(1,1;1,1)&=a_+^{(1)}(\lambda,w)a_-^{(1)}(\lambda -2\eta,w-2\eta) + b_+^{(1)}(\lambda,w) b_-^{(0)}(\lambda+2\eta,w-2\eta)\\
&= \alpha^{(1)}(\lambda,w)\alpha^{(1)}(-(\lambda - 2\eta),w-2\eta) + \beta^{(1)}(\lambda,w) \beta^{(2)}(-(\lambda + 2\eta),w-2\eta) , \\ 
R_{22}(1,1;2,0) &= b_+^{(1)}(\lambda,w) a_-^{(0)}(\lambda + 2\eta, w-2\eta) \\
&= \beta^{(1)}(\lambda,w) \alpha^{(2)}( -(\lambda +2\eta)  , w - 2\eta ) \\
R_{22}(1,1;0,2) &= a_+^{(1)}(\lambda,w) b_-^{(1)}(\lambda - 2\eta, w-2\eta)\\
&= \alpha^{(1)}(\lambda,w) \beta^{(1)}( -(\lambda -2\eta)  , w - 2\eta ) 
\end{align*}
One can check that
$$
R_{22}(1,1;1,1) + R_{22}(1,1;2,0) + R_{22}(1,1;0,2) = 1
$$

More generally, $R_{lm}(j',k';j,k)$ corresponds to the plaquette given by 

\begin{center}
\begin{tikzpicture}
\draw (2,0) node {$j$};
\draw (-2,0) node {$j'$};
\draw (0,2) node {$k$};
\draw (0,-2) node {$k'$};
\draw (0.75,-0.75) node {$\lambda$};
\draw [very thick](-1.5,0) -- (1.5,0) ;
\draw [very thick](0,-1.5) -- (0,1.5) ;
\end{tikzpicture}
\end{center}

Let us now explicitly evaluate these weights, which we will call the \textit{$q$-Jackson specialization} (similar to the \textit{$q$-Hahn specialization} giving rise to the $q$-Hahn TASEP in Section 6.6 of \cite{BP}) of the fused R-matrix elements $R_{lm} (j', k'; j, k)$. 

In what follows, we fix a parameter $\eta \in \mathbb{C}$ with $\Im \eta > 0$. For any $z \in \mathbb{C}$, we set $f(z) = \sin (\pi z)$. Furthermore, for any integer $k \ge 0$, we denote $[z]_k = \prod_{j = 0}^{k - 1} f(z - 2 \eta j)$ and $g (k) = [2 \eta k]_k f (2 \eta)^{-k}$.

Recall that there is an explicit relationship between the $R$-matrix elements $R_{lm} (j', k'; j, k)$ (implicitly dependent on an additional parameter $w \in \mathbb{C}$) and the fused weights $W_J \big( i_1, j_1; i_2, j_2 \b| v, \lambda, \Lambda \big)$ given by Definition 3.6 of \cite{AmolIRF}. Specifically, we have that 
\begin{flalign}
\label{rw}
R_{lm} (j', k'; j, k) = \displaystyle\frac{g (k')}{g(k)} W_J (k, j; k', j'  \b| v, \lambda, \Lambda),	
\end{flalign}

\noindent where 
\begin{flalign} 
\label{lambdam}
\Lambda = m; \qquad J = l; \qquad v = w + 2 \eta (m - 2l). 
\end{flalign}

\noindent We are interested in the case when the weights on the right side of \eqref{rw} simplify (which is called the $q$-Jackson specialization); this occurs when $v = - \eta \Lambda$ or, equivalently, when $w + 3 \eta m = 4 \eta l$. In that case, Theorem 3.10 of \cite{AmolIRF} yields  
\begin{flalign*}
W_J \big( i_1, j_1; i_2, j_2 \b| v, \lambda\big) & = f (2 \eta)^{i_2 - i_1} \displaystyle\frac{\textbf{1}_{i_1 \ge j_2} \big[ 2 \eta J \big]_J}{\big[ 2 \eta j_1 \big]_{j_1} \big[ 2 \eta (J - j_1) \big]_{J - j_1}} \displaystyle\frac{\big[ 2 \eta \Lambda \big]_{i_1}}{\big[ 2 \eta \Lambda \big]_{i_2}} \displaystyle\frac{ \big[ 2 \eta i_1 \big]_{j_2} \big[ 2 \eta (\Lambda - i_1) \big]_{J - j_2} }{\big[ 2 \eta \Lambda \big]_J} \\
& \qquad \times \displaystyle\frac{\big[ \lambda + 2 \eta i_2 \big]_{J - j_1} \big[ \lambda + 2 \eta (i_2 + j_1 - \Lambda - 1) \big]_{j_1}}{\big[ \lambda + 2 \eta j_1 \big]_{J - j_1} \big[ \lambda + 2 \eta (2j_1 - J - 1) \big]_{j_1} }, 
\end{flalign*}

\noindent if $i_1 + j_1 = i_2 + j_2$, and is equal to $0$ otherwise. Applying \eqref{rw} (and replacing the quadruple $(j', k'; j, k)$ with $(j_2, i_2; j_1, i_1)$, which we always assume to satisfy $i_1 + j_1 = i_2 + j_2$ in the below), we find that
\begin{flalign}
\label{ridentity}
\begin{aligned}
 R_J (j_2, i_2; j_1, j_1) & = \displaystyle\frac{g (i_2)}{g(i_1)} W_J (i_1, j_1; i_2, j_2),\\
& = \displaystyle\frac{[2 \eta i_2]_{i_2}}{[2 \eta i_1]_{i_1}} \displaystyle\frac{\textbf{1}_{i_1 \ge j_2} \big[ 2 \eta J \big]_J}{\big[ 2 \eta j_1 \big]_{j_1} \big[ 2 \eta (J - j_1) \big]_{J - j_1}} \displaystyle\frac{\big[ 2 \eta \Lambda \big]_{i_1}}{\big[ 2 \eta \Lambda \big]_{i_2}} \displaystyle\frac{ \big[ 2 \eta i_1 \big]_{j_2} \big[ 2 \eta (\Lambda - i_1) \big]_{J - j_2} }{\big[ 2 \eta \Lambda \big]_J} \\
& \qquad \times \displaystyle\frac{\big[ \lambda + 2 \eta i_2 \big]_{J - j_1} \big[ \lambda + 2 \eta (i_2 + j_1 - \Lambda - 1) \big]_{j_1}}{\big[ \lambda + 2 \eta j_1 \big]_{J - j_1} \big[ \lambda + 2 \eta (2j_1 - J - 1) \big]_{j_1} }. 
\end{aligned}
\end{flalign}

\noindent We now reparameterize 
\begin{flalign}
\label{qskappa1} 
q = e^{- 4 \pi \textbf{i} \eta}; \quad s = e^{2 \pi \textbf{i} \eta \Lambda}; \quad \kappa = e^{2 \pi \textbf{i} \lambda}.
\end{flalign} 

\noindent Observe that
\begin{flalign*}
f (a) = \displaystyle\frac{e^{\pi \textbf{i} a} - e^{- \pi \textbf{i} a}}{2 \textbf{i}} = \displaystyle\frac{\textbf{i}}{2 A^{1 / 2}} \big( 1 - A \big) = \displaystyle\frac{A^{1 / 2}}{2 \textbf{i}} \big( 1 - A^{-1} \big), \quad \text{where $A = e^{2 \pi \textbf{i} a}$,}
\end{flalign*}

\noindent for any $a \in \mathbb{C}$. This gives 
\begin{flalign}
\label{akhypergeometric}
[a]_k = \left( \displaystyle\frac{\textbf{i}}{2 A^{1 / 2}} \right)^k q^{- \binom{k}{2} / 2} \big( A; q \big)_k = \left( \displaystyle\frac{A^{1 / 2}}{2 \textbf{i}} \right)^k q^{\binom{k}{2} / 2} \big( q^{1 - k} A^{-1}; q \big)_k.
\end{flalign}

\noindent Repeatedly applying \eqref{akhypergeometric} to \eqref{ridentity} yields 
\begin{flalign*}
R_J (j_2, i_2; j_1, j_1) & = (-1)^{j_1} q^{\binom{j_1}{2} - i_2 J} \textbf{1}_{i_1 \ge j_2} \displaystyle\frac{(q; q)_{i_2}}{(q; q)_{i_1}} \displaystyle\frac{(q; q)_J}{(q; q)_{j_1} (q; q)_{J - j_1}} \displaystyle\frac{(s^2; q)_{i_1}}{(s^2; q)_{i_2}} \\ 
& \quad \times \displaystyle\frac{(q^{i_1 - j_2 + 1}; q)_{j_2} (s^2 q^{i_1}; q)_{J - j_2}}{(s^2; q)_J} \displaystyle\frac{(q^{i_2 + j_1 + 1 - J} \kappa; q)_{J - j_1} (q^{i_2} s^2 \kappa; q)_{j_1}}{(q^{2j_1 - J + 1} \kappa; q)_{J - j_1} (q^{j_1 - J} \kappa; q)_{j_1}} \\
& = q^{J (j_1 - i_2) }\textbf{1}_{i_1 \ge j_2}  \displaystyle\frac{(q; q)_{i_2}}{(q; q)_{i_2 - j_1}} \displaystyle\frac{(q^{-J}; q)_{j_1}}{(q; q)_{j_1}} \displaystyle\frac{(s^2 q^J; q)_{i_2 - j_1}}{(s^2; q)_{i_2}}  \displaystyle\frac{(q^{i_2 + j_1 + 1 - J} \kappa; q)_{J - j_1} (q^{i_2} s^2 \kappa; q)_{j_1}}{(q^{2j_1 - J + 1} \kappa; q)_{J - j_1} (q^{j_1 - J} \kappa; q)_{j_1}}.
\end{flalign*}

\noindent Denoting 
\begin{flalign}
\label{qskappa2} 
Q = q^{-1}; \qquad S = s^{-1}; \qquad K = q^{2 i_2} s^2 \kappa,
\end{flalign}

\noindent we obtain 
\begin{flalign*}
R_{lm} (j_2, i_2; j_1, j_1) & = (S^2 Q^J)^{j_1} \displaystyle\frac{(Q; Q)_{i_2}}{(Q; Q)_{i_2 - j_1} (Q; Q)_{j_1}} \displaystyle\frac{(Q^{-J}; Q)_{j_1} (S^2 Q^J; Q)_{i_2 - j_1}}{(S^2; Q)_{i_2}} \\
& \qquad \times  \displaystyle\frac{(Q^{i_2} S^2 K; q)_{J - j_1} (Q^{i_2 - j_1 + 1} K; q)_{j_1}}{(Q^{2i_2 -j_1} S^2 K; q)_{J - j_1} (Q^{2i_2 - 2j_1 + J + 1} S^2 K; q)_{j_1}},
\end{flalign*}

\noindent which is a reparameterization of the function $\varphi$ given by Definition 1.2 of \cite{AmolIRF}.  

To see this reparametrization, set $A = S^2 Q^J$ and $B = S^2$. Replacing $(j_2, i_2; i_1, j_2)$ with $(j', k'; j, k)$, \eqref{lambdam}, \eqref{qskappa1}, and \eqref{qskappa2} imply that 
\begin{flalign*}
Q & = e^{4 \pi \textbf{i} \eta}; \quad S = e^{- 2 \pi \textbf{i} \eta \Lambda} = e^{- 2 \pi \textbf{i} \eta m}; \quad K = e^{2 \pi \textbf{i} (\lambda + 2 \eta \Lambda - 4 i_2 \eta)} = e^{2 \pi \textbf{i} (\lambda + 2 m \eta - 4 k' \eta)}; \\
 & \qquad \qquad A  = e^{4 \pi \textbf{i} \eta ( J - \Lambda)} = e^{4 \pi \textbf{i} \eta ( l - m)}; \qquad B = e^{- 4 \pi \textbf{i} \eta \Lambda} = e^{- 4 \pi \textbf{i} \eta m}.
\end{flalign*}

\noindent Then, 
\begin{flalign*}
R_{lm} (j', k'; j, k) = \varphi_{Q, A, B, K} (j \b| k'),  
\end{flalign*}

\subsubsection{The dynamical $q$--Boson weights}
Here, we will analyze a certain continuous time limit of the dynamical $q$-Hahn Boson model. We will see that the situation differs from the non--dynamical case, where the $q$--Boson process can be obtained both as a limit of ASEP$(q,j)$ and the $q$--Hahn Boson. 

Let us fix parameters $q, a, b, \kappa \in \mathbb{C}$. For all nonnegative integers $i, j$ (including $i = \infty$ if $|q| < 1$), define 
	\begin{flalign*}
	\varphi \big( j \b| i \big) = \varphi_{q, a, b, \kappa} \big( j \b| i \big) = a^j \displaystyle\frac{\textbf{1}_{j \le i} (q; q)_{i}}{(q; q)_{j} (q; q)_{i - j}} \displaystyle\frac{(b / a; q)_{j} (a; q)_{i - j}}{ (b; q)_{i}} \displaystyle\frac{(q^{i} b \kappa; q)_{i - j} (q^{i - j + 1} \kappa; q)_{j}}{(q^{i - j} a \kappa; q)_{i - j} (q^{2i - 2j + 1} a \kappa; q)_{j}}. 
	\end{flalign*}

\noindent Following Section 2.2 of \cite{KuanCMP} or Section 1 of \cite{KMMO}, we change variables $\lambda = \frac{b}{a}$ and $\mu = b$ to find that 
\begin{flalign*}
\varphi \big( j \b| i \big) = \left( \displaystyle\frac{\mu}{\lambda} \right)^j \displaystyle\frac{\textbf{1}_{j \le i} (q; q)_{i}}{(q; q)_{j} (q; q)_{i - j}} \displaystyle\frac{(\lambda; q)_{j} (\mu / \lambda; q)_{i - j}}{ (\mu; q)_{i}} \displaystyle\frac{(q^{i} \mu \kappa; q)_{i - j} (q^{i - j + 1} \kappa; q)_{j}}{(q^{i - j} \mu \kappa / \lambda; q)_{i - j} (q^{2i - 2j + 1} \mu \kappa / \lambda; q)_{j}}. 
\end{flalign*}	

Of interest in both Section 2.2 of \cite{KuanCMP} and equation (42) of \cite{KMMO} was the continuous time limit of the quantity $\varphi (i \b| j)$, obtained by setting $\lambda = 1 - \varepsilon$, scaling time by $\varepsilon^{-1}$, and letting $\varepsilon$ tend to $0$. Then, the probabilities $\varphi (i \b| j)$ become rates $\Gamma (i \b|  j) = - \frac{\partial}{\partial \lambda} \varphi (i \b| j)$, evaluated at $\lambda = 1$. This gives
\begin{flalign*}
\Gamma \big( j \b| i \big) = \mu^j \displaystyle\frac{\textbf{1}_{j \le i} (q; q)_{i}}{(q; q)_{j} (q; q)_{i - j}} \displaystyle\frac{(q; q)_{j - 1} (\mu; q)_{i - j}}{ (\mu; q)_{i}} \displaystyle\frac{(q^{i} \mu \kappa; q)_{i - j} (q^{i - j + 1} \kappa; q)_{j}}{(q^{i - j} \mu \kappa; q)_{i - j} (q^{2i - 2j + 1} \mu \kappa; q)_{j}},
\end{flalign*} 

\noindent if $j \ne 0$, and 
\begin{flalign*}
\Gamma \big( 0 \b| i \big) &= \mu  \displaystyle\sum_{j = 0}^{i - 1} \left( \displaystyle\frac{q^{i + j} \kappa}{1 - \mu q^{i + j} \kappa} - \displaystyle\frac{q^j}{1 - \mu q^j} \right)\\
&= \mu  \displaystyle\sum_{j = 0}^{i - 1} \left( \displaystyle\frac{q^{i + j} \kappa-q^j}{(1 - \mu q^{i + j} \kappa)(1-\mu q^j)} \right)
\end{flalign*} 

\noindent The $\kappa = 0$ case of these identities match with a special case of equation (43) of \cite{KMMO}.

Letting $\mu$ tend to $0$ and rescaling time by $\mu^{-1}$, only the rates in the cases $j = 0$ and $j = 1$ remain nonzero; denote these by $\Psi (0)$ and $\Psi (1)$, respectively. Then, 
\begin{flalign*}
\Psi (0)  = \displaystyle\frac{(q^i - 1)(1 - q^i \kappa)}{1 - q}; \qquad \Psi ( 1 ) = \displaystyle\frac{(1 - q^i) (1 - q^i \kappa)}{1 - q}.
\end{flalign*}  
Note that these differ from the jump rates in Section \ref{diffe}.

\bibliographystyle{alpha}
\bibliography{WriteUp5}

\end{document}